\documentclass
[
    a4paper,
    DIV=11,
    abstracton
]
{scrartcl}
\usepackage{fullpage,amsmath,amssymb,amsthm}
\usepackage{bm}
\usepackage{enumerate}
\usepackage[table]{xcolor}
\usepackage{tikz}
\usepackage{pgfplots}
\usepackage{caption}
\usepackage{subcaption}
\usepackage{dsfont}
\usepgfplotslibrary{fillbetween}
\usepackage[pdffitwindow=false,
            plainpages=false,
            pdfpagelabels=true,
            pdfpagemode=UseOutlines,
            pdfpagelayout=SinglePage,
            bookmarks=false,
            colorlinks=true,
            hyperfootnotes=false,
            linkcolor=blue,
            urlcolor=blue!30!black,
            citecolor=green!50!black]{hyperref}

\newtheorem{definition}{Definition}
\newtheorem{corollary}[definition]{Corollary}
\newtheorem{lemma}[definition]{Lemma}
\newtheorem{proposition}[definition]{Proposition}

\newtheorem{assumption}[definition]{Assumption}
\newtheorem{remark}[definition]{Remark}
\newtheorem{theorem}[definition]{Theorem}

\newcommand{\R}{\mathbb{R}}
\newcommand{\ep}{\varepsilon}

\DeclareMathOperator{\sech}{sech}
\graphicspath{{./Media/}}

\title{Computation and optimal perturbation of finite-time coherent sets for aperiodic flows without trajectory integration}

\author{Gary Froyland\thanks{School of Mathematics and Statistics, The University of New South Wales, Sydney NSW 2052, Australia
.}
\and P\'eter Koltai\thanks{Institute of Mathematics, Freie Universit\"at Berlin, 14195 Berlin, Germany.}
\and Martin Stahn\thanks{Institute of Mathematics, Universit{\"a}t Potsdam, 14476 Potsdam, Germany.}}

\begin{document}

\maketitle

\begin{abstract}
Understanding the macroscopic behavior of dynamical systems is an important tool to unravel transport mechanisms in complex  flows. A decomposition of the state space into coherent sets is a popular way to reveal this essential macroscopic evolution.  To compute coherent sets from an aperiodic time-dependent dynamical system we consider the relevant transfer operators and their infinitesimal generators on an augmented space-time manifold. This space-time generator approach avoids trajectory integration, and creates a convenient linearization of the aperiodic evolution. This linearization can be further exploited to create a simple and effective spectral optimization methodology for diminishing or enhancing coherence.  We obtain explicit solutions for these optimization problems using Lagrange multipliers and illustrate this technique by increasing and decreasing mixing of spatial regions through small velocity field perturbations.
\end{abstract}

\section{Introduction}
\label{sec:intro}

Analysing complicated flows through their transport and mixing behavior has been and still is attracting a great amount of attention~\cite{mackay1984transport, romkedar_etal_1990, wiggins_92, haller1998finite, aref_02, jones2002invariant, wiggins2005dynamical, shadden2005definition, froyland_padberg_09, Thi12, FrPa14, KaKe16, KoRe18, HaKaKo18}, both from geometric and probabilistic points of view.
Non-autonomous time-aperiodic dynamics poses additional difficulties, especially the case of finite time, where asymptotic notions cannot be applied.

The current work has two main contributions.
\begin{enumerate}[(i)]
\item
Extending the work from \cite{FK17}, that deals with the time-periodic case, to the situation of general aperiodic finite-time dynamics. We detail a method to compute finite-time coherent sets~\cite{Froyland2013} of aperiodic flows that does not require any time-integration of trajectories.
\item
A technique to find a small perturbation of the underlying aperiodic vector field in a prescribed ball in a space or subspace of vector fields, which optimally enhances or destroys the existing finite-time coherent sets. This extends optimization results in~\cite{Froyland2016}, which considered the time-periodic setup of~\cite{FK17}, to (a) aperiodic dynamics and to (b) infinite-dimensional velocity field space.
\end{enumerate}

\subsection{Augmentation}

The key construction in \cite{FK17} is the representation of a $\tau$-periodically forced flow on phase space~$X \subset \R^d$ as an autonomous flow on time-augmented phase space~$\tau S^1\times X$, where $S^1$ denotes the unit circle.
On this time-expanded phase space, the time coordinate is simply advanced at a constant rate:
\[
x'(t) = v(t, x(t)) \quad\leadsto\quad
\left\{
\begin{aligned}
\theta'(t) &= 1, \\
x'(t) &= v\big(\theta(t), x(t)\big).
\end{aligned}
\right.
\]
Finite-time coherent sets on the time interval $[0,\tau]$, as introduced in~\cite{Froyland2010b,Froyland2013}, were extracted from singular functions of the transfer operator~$\mathcal{P}_{0,\tau}$, where the transfer operator is the linear operator describing the evolution of distributions under the dynamics subject to a small random perturbation. The crucial observation is that singular modes of $\mathcal{P}_{0,\tau}$ are eigenmodes of $\mathcal{P}_{0,\tau}^*\mathcal{P}_{0,\tau}$, where $\mathcal{P}_{0,\tau}^*$ denotes the dual of~$\mathcal{P}_{0,\tau}$, and that for area-preserving dynamics (corresponding to divergence-free velocity fields~$v$) the dual is the transfer operator of the time-reversed dynamics, i.e., the (again, slightly stochastically perturbed) flow governed by the time-reflected velocity field~$(t,x)\mapsto -v(\tau-t,x)$. The dynamic interpretation of this operator-based characterization is that (finite-time) coherent are those sets that are to a large extent mapped back to themselves by a noisy forward-backward evolution of the dynamics. This operator-based framework not only gives a qualitative framework for coherence; the singular values of $\mathcal{P}_{0,\tau}$ provide \emph{quantitative} bounds for coherence~\cite{Froyland2013,FrPa14}---namely the closer the singular value is to one, the less mixing occurs between the coherent set and its exterior under the noisy dynamics.

By concatenating the ``forward-time'' and ``backward-time'' velocity fields on time intervals $[0,\tau]$ and $[\tau,2\tau]$---see \eqref{hatvdef} below---we will construct a system on the augmented space~$[0,2\tau]\times X$ that mimics the forward-backward evolution of the dynamics. In this way, the eigenmodes of the (Fokker--Planck-) generator $\bm{\hat{G}}$ of this augmented system yield the eigenmodes of $\mathcal{P}_{0,\tau}^*\mathcal{P}_{0,\tau}$; i.e., the singular modes of $\mathcal{P}_{0,\tau}$, and from these singular modes the desired finite-time coherent sets. Again, the corresponding eigenvalues of the generator can be used to give quantitative bounds on coherence/mixing, see Theorem~\ref{thm:cohratio}. We note that a numerical approach to extract coherent sets from~$\mathcal{P}_{0,\tau}$ by solving the Fokker--Planck equation has been described in~\cite{DJM}. In contrast to \cite{DJM} we do not require time-integration over $[0,\tau]$, which is especially advantageous once we consider the optimization of coherence and mixing.

Connecting the spectral properties of the generator of the augmented-space system with the (finite-time) dynamical properties of the original system is a generalization of the results from~\cite{FK17}, where it has been done for time-periodic velocity fields on infinite time intervals.
Through time-reflection of the finite-time problem we construct a time-periodic one, to which we apply the concepts of~\cite{FK17}.
We also remove an assumption from~\cite{FK17} (the ``niceness'') on the so computed sets, thus strengthening the approach.
There are several further non-trivial adjustments needed to fit the theory of~\cite{FK17} to this time-reflected setting, and the necessary details are covered in sections~\ref{sec:convdiff} and~\ref{sec:aug_gen}.

The interplay between the spectra of the dynamics in augmented space and the non-autonomous dynamics in the original space has strong connections to the correspondence between \emph{evolution semigroups}~\cite{How74} and two-parameter evolution families, as elaborated in, e.g.,~\cite{Chicone1999}; see also~\cite[Section VI.9]{EnNa00} for a general introduction.
We mention that by a similar construction, spatio-temporal dynamical patterns were extracted in~\cite{GiDa17} by considering the generator of the Koopman operator (the adjoint of the transfer operators considered here) associated with the augmented-space dynamics.

\subsection{Manipulation of coherence and mixing}

There are several different ways to measure mixing and mixedness under (stochastically perturbed) dynamics, such as considering dispersion statistics or the change of variation in a concentration field; see e.g.~\cite{Pro99,LiHa04,ThDoGi04,Thi08}. Multiscale norms of mixing measure how ``oscillatory'' a concentration field is~\cite{MaMePe05}; see~\cite{Thi12} for a review.
The most widely used approaches to the problem of mixing optimization search for switching protocols between some fixed velocity fields in order to optimize some topological~\cite{BoArSt00} or other mixing measure~\cite{MMGVP07, CoAdGi08, OBP15}. Other strategies include optimising the diffusion component of the dynamics~\cite{FGTW16}, the optimal distribution of concentration sources~\cite{ThPa08} and geometric dynamical systems techniques~\cite{Bal15}. An interesting theoretical result is that arbitrary mixedness under advection-diffusion can be achieved in finite time solely by sufficiently increasing the strength of the (otherwise fixed) advective flow~\cite{CKRZ08}. If there are no restrictions to the choice of the velocity field, one can choose the one that is optimally mixing the actual concentration at every time instance~\cite{LiThDo11}. We also note that a related problem to mixing enhancement arises in statistical mechanics~\cite{LePa13} where the convergence toward the stationary distribution should be accelerated, e.g., to increase the efficiency of sampling.

Instead of focusing on one fixed concentration field, we will bound the mixing characteristics of a flow in terms of the objects that most inhibit mixing: coherent sets.
As we mentioned above, finite-time coherent sets are characterized by the singular vectors of the transfer operator, and, equivalently, by the eigenvectors of the generator of the augmented-space process, while the corresponding eigenvalue delivers an upper bound on transport between the coherent set and its exterior. Thus, we  can quantitatively access the mixing behavior of a flow on finite time through the spectrum of the augmented generator~$\bm{\hat{G}}$, and can target these eigenvalues if we want to enhance or diminish mixing.

Given a default velocity field $v$ and non-autonomous perturbations $u\in U$ from an admissible space~$U$ of divergence-free velocity fields, our approach considers ``small'' $u$ that change a targeted eigenvalue $\mu$ of $\bm{\hat{G}}$ (thus also the singular value $\sigma$ of $\mathcal{P}_{0,\tau}$) locally optimally. This procedure can be iterated to obtain a larger perturbation in a gradient-method fashion.
Optimizing singular values of the transfer operator $\mathcal{P}_{0,\tau}$ directly is difficult as it would necessarily involve the variation of the nonlinear dynamics under the velocity field~$u$.
Instead, a linearized optimization of the eigenvalue of the generator~$\bm{\hat{G}}$ leads to a very simple optimization problem~\eqref{eq:expl-sol}, which can be solved via a linear system of the same dimension as~$U$. Moreover, the theory holds for infinite-dimensional perturbation spaces $U$ as well, since Fr\'echet differentiability of the transfer operator and its spectrum with respect to perturbing velocity fields has been established in~\cite{KLP2018}.

\subsection{Overview}

This work is structured as follows. In section~\ref{sec:convdiff} we introduce the $L^2$-function based formalism to study advection-diffusion systems by the Fokker--Planck equation and its evolution operator, the transfer operator, in forward and backward time.
In section~\ref{sec:flux} we consider purely advective transport between a family of sets and its exterior in terms of fluxes through the boundary of the family; this is a geometric analogue of the operator-based considerations that follow.
Section~\ref{sec:aug_gen} introduces the new reflected augmented generator needed to handle aperiodic, finite-time driving. We state formal connections between the spectrum of the reflected augmented generator and the reflected transfer operator, and provide spectral-based bounds on the maximal possible coherence of sets in phase space under the aperiodic dynamics.
Section~\ref{sec:comput} contains a numerical demonstration of the efficacy of our trajectory-free approach.
Section~\ref{sec:opt} describes the formal setup of the optimization problem designed to manipulate the position of the dominant spectral values of the reflected augmented generator, culminating in an explicit expression for the optimal time-dependent local perturbation of the velocity field.
Section~\ref{sec:opt_num} specializes the infinite-dimensional results of section~\ref{sec:opt} to the numerical setting via discretization and includes a variety of examples of coherence reduction and enhancement.
We conclude in section~\ref{sec:conclusion}.

\section{Advective-diffusive dynamics}
\label{sec:convdiff}
Let $X \subset \R^d$ be a bounded and open set with compact and smooth  (
piecewise $C^4$) boundary.
We consider the time interval $[0,\tau]$ and the dynamics
\begin{equation}\label{dynsystem}
	d x_t = v(t,x_t) dt + \ep \, dw_t
\end{equation}
with reflecting boundary conditions for $v \in C^{(1,1)} ( [0,\tau]\times \overline{X}; \R^d)$\footnote{$C^{(1,1)}([0,\tau]\times \overline{X};\mathbb{R}^d)$ denotes the Banach space of functions $f \, : \, [0,\tau]\times \overline{X} \rightarrow \mathbb{R}^d$ that are continuously differentiable in $t$ and cotinuously differentiable in $x$.}and $(w_t)_{t\geq 0}$ being a standard Wiener process in $\R^d$.
The initial point $x_{0}$ is distributed according to some initial density $f_0\in L^2(X)$.
The evolution of the density of the governing equation (\ref{dynsystem}) is given by the Fokker--Planck equation or Kolmogorov forward equation \cite[Section 11.6]{lasotamackey},
\begin{align}
\label{FP0}
	\partial_t f(t,x) &= -\text{div}_x \big(f(t,x)v(t,x) \big)+\frac{\ep^2}{2} \Delta_x f(t,x)\\
	f(0,x) &= f_0(x)\label{eq:f_0}\\
	\frac{\partial f(t,\cdot)}{\partial n} &=0\mbox{ on $\partial X$,} \nonumber
\end{align}
where $\tfrac{\partial}{\partial n}$ is the normal derivative on the boundary.
Associated to \eqref{FP0} is an evolution operator $\mathcal{P}_{0,t}:L^2(X)\to L^2(X)$ that transports a density $f_0\in L^2(X)$ at time 0 to the solution density of \eqref{FP0} at time $t$.
The evolution operator $\mathcal{P}_{0,t}$ is an integral operator with stochastic\footnote{Doubly stochastic if the flow is volume-preserving.} kernel
$k(t,\cdot, \cdot):X \times X\to \mathbb{R}^+$ that satisfies~\cite[Assumptions 1 and~2]{Froyland2013}.

\subsection{Construction of a forward-backward process}

For simplicity of presentation we assume that the velocity field $v(t,\cdot)$ is divergence free for all $t\in[0,\tau]$.
We note that the remaining arguments in this section may be carried through for general velocity fields.
Denote by $\langle \cdot , \cdot \rangle_{H}$ the canonical scalar product of a Hilbert space $H$.
Following \cite{Froyland2013} in the volume-preserving setting\footnote{The operators $\mathcal{P}_{0,\tau}$ and $\mathcal{P}_{0,\tau}^*$ are replaced by normalised versions for nonzero divergence velocity fields;  these are denoted by $\mathcal{L}$ and $\mathcal{L}^*$ in \cite{Froyland2013}.}, coherent sets over the time interval $[0,\tau]$ are extracted from the eigenfunctions of $\mathcal{P}_{0,\tau}^*\mathcal{P}_{0,\tau}$ corresponding to large eigenvalues, where $\mathcal{P}_{0,\tau}^*$ is the {$L^2$-adjoint} of $\mathcal{P}_{0,\tau}$, defined to be the unique linear operator satisfying
\begin{equation*}
	\langle \mathcal{P}_{0,t} f, g \rangle_{L^2(X)} = \langle f , \mathcal{P}_{0,t}^{\ast} g \rangle_{L^2(X)}
\end{equation*}
for all $f,g \in L^2(X)$.
The eigenvalues of $\mathcal{P}_{0,\tau}^*\mathcal{P}_{0,\tau}$ (the singular values of $\mathcal{P}_{0,\tau}$) are known to lie in the interval $[0,1]$ (cf. \cite[p.3]{Froyland2013}). The rationale behind the operator $\mathcal{P}_{0,\tau}^*\mathcal{P}_{0,\tau}$ is that $\mathcal{P}_{0,\tau}$ describes evolution in forward time, $\mathcal{P}_{0,\tau}^*$ describes evolution under the time-reversed dynamics, and coherent sets are characterized exactly by the property that they are ``stable'' under a noisy forward-backward evolution of the dynamics.

The adjoint operator $\mathcal{P}_{t,\tau}^*$ is the solution operator to the Kolmogorov backward equation~\cite{Pavliotis2008}:
\begin{align}
\label{FP0dual}
	- \partial_t g(t,x)&=\langle \nabla_x(g(t,x)), v(t,x) \rangle_{\R^d}+\frac{\ep^2}{2}\Delta_x g(t,x)\\
	g(\tau,x)&= g_{\tau}(x) \nonumber\\
	\frac{\partial g(t,\cdot)}{\partial n} &=0\mbox{ on $\partial X$.} \nonumber
\end{align}
The operator $\mathcal{P}_{t,\tau}^{\ast}$ maps a density $g_{\tau}$ at time $t=\tau$ backward in time according to \eqref{FP0dual} to produce a density $g_t$ at time $t<\tau$.
We may simplify (\ref{FP0dual}) using volume preservation:
 \begin{equation}
\label{div-back}
	\begin{aligned}
		\langle \nabla_x (g(t,x)) , v(t,x) \rangle_{\R^d} &= \text{div}_x (g(t,x ) v(t,x )) - g(t,x) \text{div}_x(v(t,x)) \\
			&= \text{div}_x (g(t,x ) v(t,x )) \; .
	\end{aligned}
\end{equation}
Thus we may write \eqref{FP0dual} as
\begin{equation}
\label{FP1dual}
	- \partial_t g(t,x)= \text{div}_x (g(t,x)v(t,x))+\frac{\ep^2}{2} \Delta_x g(t,x) \; .
\end{equation}
Reversing time in (\ref{FP0dual}) to obtain an initial value problem we get
\begin{align}
\label{FP0dualreverse}
	\partial_t f(t,x)&= - \text{div}_x (f(t,x)\bar{v}(t,x)) +\frac{\ep^2}{2} \Delta_x f(t,x)\\
	f(0,x) &= \bar{f}_0(x)\nonumber \\
	\frac{\partial f(t,\cdot)}{\partial n} &=0 \mbox{ on $\partial X$} \nonumber
\end{align}
using $f (t, x ) = g (\tau - t, x )$, $\bar{f}_0 (x) = g_{\tau} (x)$ and the velocity field $\bar{v}(t,x ) = - v(\tau - t, x )$.
Comparing (\ref{FP0}) and (\ref{FP0dualreverse}) we see that the natural evolution of the adjoint problem (the  Kolmogorov backward equation) corresponds to the forward problem (the Kolmogorov forward equation) of the time reversed dynamics.

We wish to construct the process over the time interval $[0,2\tau]$ that corresponds to the operator $\mathcal{P}_{0,\tau}^*\mathcal{P}_{0,\tau}$.
We view $\mathcal{P}_{0,\tau}^*$ as evolution on the time interval $[\tau,2\tau]$ and we therefore shift (\ref{FP0dualreverse}) by $\tau$ time units, defining $\tilde{v}(t,x) := \bar{v}(t-\tau,x)=-v(2\tau-t,x)$ to obtain a forward problem on $[\tau, 2\tau]$:
\begin{align}
\label{FPBshift}
	\partial_t f(t,x)&= -\text{div}_x(f(t,x)\tilde{v}(t,x))+\frac{\ep^2}{2} \Delta_x f(t,x)\\
	f(\tau,x) &= \tilde{f}_{\tau}(x) = \bar{f}_0 (x) \nonumber\\
	\frac{\partial f(t,\cdot)}{\partial n} &=0 \mbox{ on $\partial {X}$.} \nonumber
\end{align}
We denote the solution operator of this problem as $\tilde{\mathcal{P}}_{\tau,t}$ ($=\mathcal{P}^{\ast}_{2\tau-t, \tau}$).

Finally, we concatenate the two forward problems \eqref{FP0} and \eqref{FPBshift} to make a single process over $[0,2\tau]$.
We mark objects that live on this extended interval $[0,2\tau]$ with a hat $\hat{\phantom{v}}$.
Define the velocity field
\begin{equation}
\label{hatvdef}
	\hat{v}(t,\cdot)= \zeta' (t) v(\zeta (t),\cdot ) = \left\{
    	\begin{array}{ll}
        	v(t,\cdot), & \hbox{$t\in [0,\tau]$;} \\
            - v(2\tau -t,\cdot), & \hbox{$t\in (\tau,2\tau]$,}
        \end{array}
        \right.
\end{equation}
using the reflection map
\begin{equation}
\label{reflmap}
	\zeta (t) = \left\{
    	\begin{array}{ll}
        	t, & \hbox{$t\in [0,\tau]$;} \\
            2 \tau - t, & \hbox{$t\in (\tau,2\tau]$.}
    	\end{array}
    	\right.
\end{equation}
The resulting velocity field $\hat{v}$ exhibits discontinuities in $0,\tau$ and $2\tau$ whenever it does not vanish there, but one-sided derivatives exist\footnote{The following regularity may not be most general but is meant to give some intuition: $\hat{v} \in C^{(1,1)} (((0,\tau)\cup (\tau,2\tau))\times \overline{X};\mathbb{R}^d)$ and $\hat{v} \in L^p((0,2\tau)\times X ; \mathbb{R}^d)$, for $1\leq p \leq \infty$.}
In what follows, we will solve the Fokker--Planck equation
\begin{align}
\label{FP2}
	\partial_t \hat{f}(t,x) &= -\text{div}_x(\hat{f}(t,x)\hat{v}(t,x))+\frac{\ep^2}{2} \Delta_x \hat{f}(t,x)\\
	\hat{f}(0,x) &=f_0(x)\nonumber \\
	\frac{\partial \hat{f}(t,\cdot)}{\partial n} &=0\mbox{ on $\partial X$,}\nonumber
\end{align}
over the interval $t\in[0,2\tau]$; more precisely on $(0,\tau)\cup (\tau,2\tau)$ with $L^2$-continuous concatenation at $t=\tau$ (see Prop. \ref{augmen}).
Let us summarize the above construction with the following proposition.
\begin{proposition}
\label{claim:FB}
	The concatenation $\mathcal{P}_{0,\tau}^*\mathcal{P}_{0,\tau} =: \hat{\mathcal{P}}_{0,2\tau}$ with $\tilde{f}_{\tau} = \mathcal{P}_{0,\tau}f_0$ comprises initializing (\ref{FP0}) at time 0, solving forward using the vector field $v(t,\cdot)$ until time $\tau$, then continuing to evolve~\eqref{FPBshift} for another $\tau$ time units, but now using the reflected and shifted vector field $-v(2\tau-t,\cdot)$ for $t\in [\tau,2\tau]$ corresponding to the reversed dynamics.
\end{proposition}

\section{Cumulative flux from a reflected family of sets}
\label{sec:flux}

Before proceeding with the operator-based description of finite-time coherence, in this section we analyse the reflected dynamics by its flux through the boundary of a moving (possibly coherent) set. Our intention behind connecting this to a flux in augmented space (i.e., space-time) in Proposition~\ref{prop:outflux} is partially to set the stage for the augmented-space operator-based description in section~\ref{sec:aug_gen}. Apart from strengthening the  intuition for the forward-backward construction that is used in here, the results of the section \ref{sec:outflux} are not formally needed for the rest of the paper.

\subsection{Augmentation and reflection}\label{sec:augmented}
For a family of sets $\{A_t\}_{t\in [0,\tau]}$, $A_t \subset X$, we consider the \textit{augmented set}
\begin{equation}
\label{augset}
	\bm{A}=\bigcup_{\theta=0}^\tau \{\theta\}\times A_\theta \subset [0,\tau]\times X,
\end{equation}
in the augmented state space $\bm{X} := [0,\tau]\times X$. Let $\bm{n(x)}$ denote the unit outer normal on $\partial \bm{A}$ at $\bm{x} \in \bm{X}$ and $\bm{v}$ the augmented velocity field defined by
\begin{equation*}
	\bm{v}(\bm{x}) = (1,v(\theta,x)) \; .
\end{equation*}

We define a \textit{reflected} family of sets $\{\hat{A}_t\}_{t\in [0,2\tau]} = \{A_{\zeta(t)}\}_{t \in [0,2\tau]}$ in synchrony with the reflected vector field $\hat{v}$:
\begin{equation}
\label{hatAdef}
\hat{A}_t=\left\{
                 \begin{array}{ll}
                   A_t, & \hbox{$t\in [0,\tau]$,} \\
                   A_{2\tau-t}, & \hbox{$t\in (\tau,2\tau]$}
                 \end{array}
               \right.
\end{equation}
and the augmented reflected set
\begin{equation*}
	\hat{\bm{A}} := \bigcup_{t=0}^{2\tau} \{t\} \times \hat{A}_t = \bigcup_{t=0}^{2\tau} \{t\} \times A_{\zeta(t)} \; ;
\end{equation*}
see Figure \ref{fig:augment}.
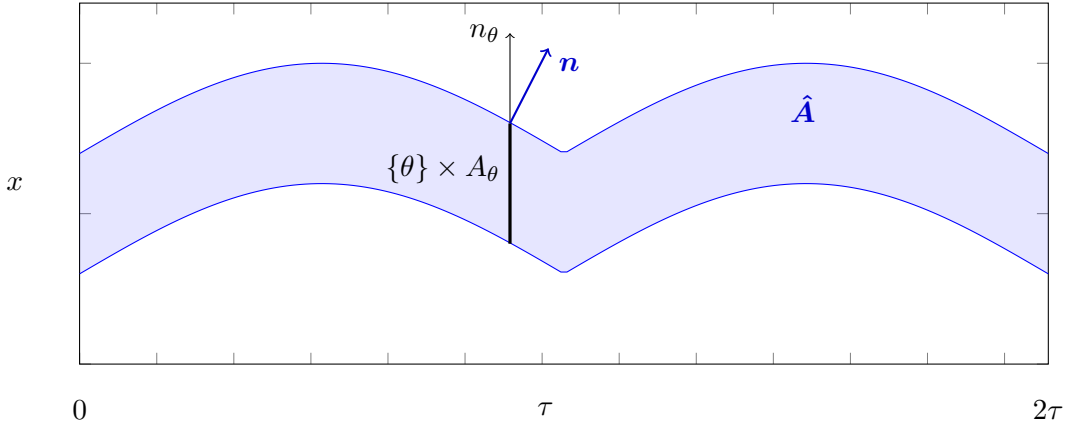
\begin{figure}[hbt]
	\begin{center}
	\begin{tikzpicture}[]
    \begin{axis}[
    enlargelimits=0.0001,
    xtick       = {},
    xticklabels = {},
    ytick = {},
    yticklabels = {},
    samples = 160,
    xmin = 0, xmax = 2*pi,
        ymin = 0, ymax = 1.2,
        width = 0.9\textwidth,
        height = 0.4\textwidth
    ]
    \addplot[name path=fb,domain= 0:2*pi,blue] {0.3*abs(sin((\x) r))+0.3};
    \addplot[name path=ft,domain= 0:2*pi,blue] {0.3*abs(sin(\x r))+0.7};
    \addplot [
        very thin,
        color=gray,
        fill=blue,
        fill opacity=0.1
    ]
    fill between[
        of=fb and ft,
        soft clip={domain=0:2*pi},
    ];
    \addplot[very thick] coordinates {
    (pi-0.35,0.4)
    (pi-0.35,0.8)
    };
        \pgfplotsset{
    after end axis/.code={
        \node[left] at (axis cs:pi+1.7,0.85){\textcolor{blue!80!black}{$\bm{\hat{A}}$}};
        \node[left] at (axis cs:-0.3,0.6){$x$};
        \node[left] at (axis cs:pi,-0.15){$\tau$};
        \node[] at (axis cs:0,-0.15){$0$};
        \node[] at (axis cs:2*pi,-0.15){$2\tau$};
        \draw[->] (axis cs:pi-0.35,0.8) -- (axis cs:pi-0.35,1.1);
        \node[ left] at (axis cs:pi-0.35,1.1) {$n_{\theta}$};
        \draw[->,thick,color=blue!80!black] (axis cs:pi-0.35,0.8) -- (axis cs:pi-0.1,1.05);
        \node[below right,color=blue!80!black] at (axis cs:pi-0.1,1.05) {$\bm{n}$};
        \node[left] at (axis cs:pi-0.35,0.65){$\{\theta\} \times A_{\theta}$};
        }
        }
    \end{axis}
\end{tikzpicture}
\caption{Illustration of the augmented reflected set $\hat{\bm{A}}$ and normal vectors $\hat{n}_t (x)$ in the case where $d=1$.}
\label{fig:augment}
\end{center}
\end{figure}

\subsection{Outflow flux}\label{sec:outflux}
We consider a family of $d$-dimensional sets $\{A_t\}_{t \in [0,\tau]}$, $A_t \subset X\subset \mathbb{R}^d$ satisfying the following assumptions (the boundaries are piecewise smooth in space and differentiable in time):
\begin{assumption}~
\label{assumpAt}
	\begin{enumerate}
		\item There exists a co-dimension 1 parameterisation set $R \subset \R^{d-1}$ such that for each $t \in [0,\tau]$ there is a bijective function $a(t,\cdot) \, : \, R \rightarrow \partial A_t$ with $a$ being piece-wise smooth, $a \in C^{(1,0)}([0,\tau] \times R;\mathbb{R}^d)$ \footnote{$C^1$ in $t$ and $C^0$ in $r$} piecewise.
		\item The mapping $b(t,x) := \frac{\partial a}{\partial t} (t,r)$ is well defined for all $t \in [0,\tau]$ and all $x \in \partial A_t$, where $a(t,r) = x$.
	\end{enumerate}
\end{assumption}

The \textit{cumulative outflow flux} under the vector field $v(t,\cdot),\, t\in[0,\tau]$ from a family of sets $\{A_t\}_{t\in [0,\tau]}$
is given by
\begin{equation}
\label{outflux0}
	\int_0^\tau \int_{\partial A_t} \langle v(t,x)-b(t,x),n_t(x)\rangle_{\R^d}^+\ dS(x)dt,
\end{equation}
where $(\cdot)^+$ denotes the positive part, $S (x)$ is the $d-1$ dimensional surface measure and $n_t(x)$ is the outer normal unit vector.
The result \cite[Theorem 2]{FK17} shows that  \eqref{outflux0} is equal to the \textit{instantaneous outflow flux} from the set $\bm{A}$, defined by
\begin{equation}
\label{augoutflux0}
	\int_{\partial \bm{A}}\langle \bm{v}(\bm{x}),\bm{n}(\bm{x})\rangle_{\R^{d+1}}^+\ d\bm{S}(\bm{x}),
\end{equation}
with $\bm{S} (\bm{x})$ denoting the $d$-dimensional surface measure.

We extend this result to the reflected velocity field $\hat{v}(t,\cdot), t\in [0,2\tau ]$, generated by a general aperiodic velocity field.

In particular, recalling $\zeta$ from~\eqref{reflmap}, for every $t \in [0,2\tau]$ the boundary $\partial \hat{A}_t$ has a parametrization $\hat{a}(t,\cdot) = a (\zeta(t),\cdot) \, :\, R \rightarrow \partial \hat{A}_t$ and
\begin{equation}
\label{hatwdef}
	\hat{b} (t,x) = \left\{
                 \begin{array}{ll}
                   \frac{\partial a}{\partial t} (t, r), & \hbox{$t\in [0,\tau ]$} \\
                   - \frac{\partial a}{\partial t} (2\tau -t , r ), & \hbox{$t\in ( \tau,2\tau]$.}
                 \end{array}
               \right.
\end{equation}
At $t=\tau$ the right- and the left-sided partial derivatives of $a$ with respect to $t$ exists but they may not be equal.
The family of normal vectors is mirrored in time (see Figure \ref{fig:augment}): $\hat{n}_t (x) = n_{\zeta (t)} (x)$ for $t \in [0,2\tau]$.

\begin{proposition}\label{prop:outflux}
The cumulative outflow flux from the family of sets $\hat{A}_t, t\in [0,2\tau]$ under the vector field $\hat{v}(t,\cdot), t\in [0,2\tau]$, is equal to the cumulative absolute flux in and out of the family of sets $A_t$, $t\in [0,\tau]$, under the vector field $v(t,\cdot),t\in[0,\tau]$; that is,
\begin{equation}
\label{absflux1}
\begin{aligned}
	\int_0^{2\tau} \int_{\partial \hat{A}_t} &\langle \hat{v}(t,x)-\hat{b}(t,x),\hat{n}_t(x)\rangle_{\R^d}^+\ dS(x)dt = \\
&\qquad \int_0^{\tau} \int_{\partial A_t} |\langle v(t,x)-b(t,x),n_t(x)\rangle_{\R^d} |\ dS(x)dt.
\end{aligned}
\end{equation}
Furthermore, the time-integrated flux \eqref{absflux1} is equal to the instantaneous absolute flux in augmented space:
\begin{equation}
\label{augabsflux0}
	\int_{\partial \bm{A}}|\langle \bm{v(x),n(x)}\rangle_{\R^{d+1}}|\ d\bm{S(x)}.
\end{equation}
\end{proposition}

\begin{proof}
	Let us first prove the first equality. Therefore we do not need objects of the augmented setting yet.
We split the integral
\begin{align}
	&\int_0^{2\tau} \int_{\partial \hat{A}_t} \langle \hat{v} (t,x) - \hat{w}(t,x) , \hat{n}_t(x) \rangle^+ \, dS (x) \, dt \nonumber\\
	& \qquad = \int_0^{\tau} \int_{\partial A_t} \langle v(t,x) - w(t,x), n_t (x) \rangle^+ \, dS (x) \, dt \nonumber \\
	&\qquad \quad + \int_{\tau}^{2\tau} \int_{\partial A_{\zeta (t)}} \langle - v(\zeta(t),x) - (- w(\zeta(t),x)), n_{\zeta (t)} (x) \rangle ^+ \, dS (x) \, dt \label{outfluxint}
\end{align}
and see that the first integral is already a part of what is needed. So we only need to treat the second integral~\eqref{outfluxint}. We use Fubini's theorem and substitute using $\zeta$ with $g_{\zeta(t)}(r)$ being the Gram determinant. (It is important to note that we need and use substitution in one dimension, the time dimension, because we need the sign we get from substitution, which we would not get in higher dimensions.)
\begin{align*}
	\eqref{outfluxint} &= \int_{\tau}^{2\tau} \int_R \langle v(\zeta(t), a (\zeta(t),r) - \frac{\partial a}{\partial t} (\zeta(t),r), n_{\zeta(t)} ( a (\zeta(t),r)) \rangle^- g_{\zeta(t)}(r) \, dr \, dt \\
	&= \int_R \int_{\tau}^{0} \langle v(t,a(t,r)) - \frac{\partial a}{\partial t} (t,r), n_t (a(t,r)) \rangle^- g_t (r) (-1) \, d S(x) \, dt \\
	&= \int_0^{\tau} \int_{\partial A_t} \langle v(t,x) - w(t,x) , n_t (x) \rangle^- \, dS (x) \, dt \; .
\end{align*}
Combining the two calculations above we get the desired result.
The second equality involving the objects of the augmented setting follows analogously to~\cite[Theorem~2]{FK17}.
\end{proof}

\section{Coherent families of sets and the generator on augmented phase space}
\label{sec:aug_gen}

In this section we create a so-called spectral mapping theorem for our reflected augmented process (Proposition \ref{augmen}) and derive a bound for the finite-time coherence of a family of sets $\{A_t\}_{t\in[0,\tau]}$ in terms of the second eigenvalue of a generator (the infinitesimal operator) of our augmented reflected advection-diffusion process (Theorem \ref{thm:cohratio}).
To do this we build on discrete-time theory from \cite{Froyland2013} with the periodic continuous-time theory from~\cite{FK17}.

\subsection{The evolution operator for the reflected process}
\label{ssec:reflectP}
It is well known that $\mathcal{P}_{s,s+t}$, $t>0$, is a compact, integral preserving, real and positive operator on $L^2(X)$ while $t\mapsto \mathcal{P}_{s,s+t}f$ is continuous as a mapping from $[0,\infty)$ to $L^2(X)$ for any fixed $f\in L^2(X)$\footnote{See \ref{thm:NACP} for compactness and continuity, and~\cite{lasotamackey} or~\cite{KlKoSch16} for the other properties.}.
Furthermore $\smash{ \hat{\mathcal{P}}_{0,2\tau} = \mathcal{P}^{\ast}_{0,\tau} \mathcal{P}_{0,\tau} }$ is a self-adjoint operator on $L^2(X)$ with simple largest eigenvalue $\lambda_1 (\hat{\mathcal{P}}_{0,2\tau}) =1$.
Following \cite{Froyland2013} one has that the second eigenvalue $\smash{ \lambda_2(\hat{\mathcal{P}}_{0,2\tau}) }$ satisfies
\begin{equation}\label{eq:eig-sing}
	\sqrt{\lambda_2 (\hat{\mathcal{P}}_{0,2\tau})} = \sigma_2 ( \mathcal{P}_{0,\tau}) = \max_{\substack{f_0 \in L^2(X,\mu_0) \\ g_{\tau} \in L^2(X,\nu_{\tau}) \\ \langle f_0 , 1 \rangle_{\mu_0} = 0 \\ \langle g_{\tau} , 1 \rangle_{\nu_{\tau}} = 0}} \left\{ \dfrac{\langle \mathcal{P}_{0,\tau} f_0 , g_{\tau} \rangle_{\nu_{\tau}}}{\| f_0 \|_{\mu_0} \| g_{\tau} \|_{\nu_{\tau}}} \right\} < 1,
\end{equation}
where in the volume-preserving\footnote{In the nonzero divergence case, $\mu_0$ is a reference measure describing the initial mass distribution of the (possibly compressible) fluid being evolved, and $\nu_{\tau}$ is the forward evolution of $\mu_0$ under a normalised version of $\mathcal{P}_{0,\tau}$ denoted by $\mathcal{L}$ in \cite{Froyland2013}.}
setting $\mu_0$ and $\nu_{\tau}$ are both simply the Lebesgue measure.
We now consider the problem \eqref{FP2} introduced in section \ref{sec:convdiff} as a time-periodic problem on $2 \tau S^1 \times X$ (we extend $\hat{v}$ periodically). Following the considerations of section \ref{sec:convdiff} the evolution operator $\hat{\mathcal{P}}_{s,s+t}$ starting from time $s$, w.l.o.g.\ $s \in [0,2\tau]$, flowing for time $t \geq 0$ to $s+t= k\tau + r$, $k = \lfloor \frac{s+t}{\tau} \rfloor \in \mathbb{N}\cup\{0\}$ and $  (s+t)\mod \tau = r \in [0,\tau)$, is given by
\begin{equation}
\label{Phat}
	\hat{\mathcal{P}}_{s,s+t}=\left\{
                       \begin{array}{ll}
                         \mathcal{P}_{s,s+t} & \hbox{$s\in [0,\tau], t\in [0, \tau-s]$, ($k = 0$)} \\ 
					  \mathcal{P}^*_{2\tau -r,\tau} ( \mathcal{P}_{0,\tau } \mathcal{P}_{0,\tau}^*)^\frac{k-1}{2} \mathcal{P}_{s,\tau}, & \hbox{$s\in [0,\tau] , t > \tau-s, k$ odd,} \\ 
					 \mathcal{P}_{0,r} \mathcal{P}^*_{0,\tau} ( \mathcal{P}_{0,\tau } \mathcal{P}_{0,\tau}^*)^\frac{k-2}{2} \mathcal{P}_{s,\tau}, & \hbox{$s\in [0,\tau] , t > \tau-s, 2\leq k$ even,} \\ 
					  \mathcal{P}^*_{2\tau - (s+t),2\tau-s}, & \hbox{$s\in [\tau,2\tau],t \in[0, 2\tau - s]$,($k=1$)}\\ 
                        \mathcal{P}_{0,r} (\mathcal{P}_{0,\tau}^* \mathcal{P}_{0,\tau})^{\frac{k-2}{2}} \mathcal{P}_{0,2\tau - s}^*, & \hbox{$s \in [\tau , 2\tau], t > 2\tau -s,2\leq k$ even,} \\
                        \mathcal{P}_{\tau-r,\tau}^* \mathcal{P}_{0,\tau} (\mathcal{P}_{0,\tau}^* \mathcal{P}_{0,\tau})^{\frac{k-3}{2}} \mathcal{P}_{0,2\tau - s}^*, & \hbox{$s \in [\tau , 2\tau], t > 2\tau -s, 3\leq k$ odd.}
                       \end{array}
                     \right.
\end{equation}
The situation when $t$ is exactly $2\tau$ is of particular importance:
\begin{equation}
\label{Phat2}
\hat{\mathcal{P}}_{s,s+2\tau}=\left\{
                       \begin{array}{ll}
                         \mathcal{P}_{0,s}\mathcal{P}^*_{0,\tau}\mathcal{P}_{s,\tau}, & \hbox{$s \in [0,\tau]$;} \\
                         \mathcal{P}^*_{2\tau - s,\tau} \mathcal{P}_{0,\tau} \mathcal{P}^{\ast}_{0,2\tau - s}, & \hbox{$s \in [\tau,2\tau]$.}
                       \end{array}
                     \right.
\end{equation}
Note that $\hat{\mathcal{P}}_{s, s + 2\tau}$ is self-adjoint for $s = k \tau$, $k \in \mathbb{N}\cup\{0\}$.

\subsection{The time-augmented generator and evolution family}
We now turn to the augmented reflected system
\begin{equation}
\label{eq:aug1}
	\left\{ \begin{aligned}
		d \hat{\theta}_t &= 1 dt \\
		d \hat{x}_t &= \hat{v}(\hat{\theta}_t, \hat{x}_t) dt + \ep d \hat{w}_t
	\end{aligned} \right.
\end{equation}
in $\bm{\hat{X}} :=2\tau S^1 \times X$, and note that $(\hat{w}_t)_{t\ge 0}$ is a standard Wiener process in $\R^d$;  in particular it is not constructed by time reflection.
We define augmented versions of $\hat x_t,\hat v$, $\ep$, and $\hat w_t$, denoting them with bold symbols:
\begin{equation*}
	\bm{\hat{x}}_t := (\hat{\theta}_t , \hat{x}_t), \quad \bm{\hat{v}}(\bm{x}) := (1,\hat{v}(\theta,x))\text{ for }\bm{x} = (\theta,x) \in\bm{\hat{X}}, \quad \bm{\varepsilon}  := \begin{pmatrix} 0_{1\times 1} & 0_{1\times d} \\ 0_{d\times 1} & \ep I_{d\times d} \end{pmatrix},
\end{equation*}
and $\bm{\hat{w}}_t$ is a $d+1$ dimensional standard Wiener process. The augmented system is
\begin{equation*}
	d \bm{\hat{x}}_t = \bm{\hat{v}}(\bm{\hat{x}}_t) dt + \bm{\varepsilon} d \bm{\hat{w}}_t\,.
\end{equation*}
Considering the time-periodic version of problem~\eqref{FP2} with $\hat{v}$ on $\bm{\hat{X}}$ we formulate a Fokker--Planck equation in augmented space.
To avoid confusion with the (new) state variable $\theta$, we write dependence on time $t$ as a subscript of the augmented function, i.e., $\bm{f}_t\, :\,  \bm{\hat{X}}\to\mathbb{R}$ for all $t\ge 0$. The augmented Fokker--Planck equation is
\begin{equation}
\label{augFP}
	\partial_t \bm{f}_t(\bm{x}) = - \text{div}_{\bm{x}} \big( \bm{\hat{v}} (\bm{x}) \bm{f}_t(\bm{x}) \big) + \Delta_{\bm{x}} \big(\frac{\bm{\varepsilon}^2}{2} \bm{f}_t(\bm{x})\big) \; .
\end{equation}
Note there is no diffusion in the $\theta$ direction, as per the definition of $\bm{\varepsilon}$.

We will now consider the augmented Fokker--Planck equation as a linear differential equation in the space $L^2(\bm{\hat{X}})$. Its right-hand side is given by the so-called (augmented) \emph{infinitesimal generator} $\bm{\hat G} \, : \, \mathcal{D}(\bm{\hat{G}}) \subset L^2(\bm{\hat{X}}) \rightarrow L^2( \bm{\hat{X}})$, with \emph{domain} $\mathcal{D}(\bm{\hat{G}})$ defined as the subspace of $L^2(\bm{\hat{X}})$ on which the generator is well-defined in terms of semigroup theory~\cite{Paz83,EnNa00}.
The augmented Fokker--Planck equation \eqref{augFP} in augmented space then reads as
\begin{equation}
\label{augG}
\begin{aligned}
	\partial_t \bm{f}_t &= \bm{\hat{G}} \bm{f}_t \text{ on } \big((0,\tau)\cup (\tau,2\tau)\big) \times X, \\
	\dfrac{\partial \bm{f}_t}{\partial \bm{n}} &= 0 \text{ on } \big((0,\tau)\cup (\tau,2\tau)\big) \times \partial X.
\end{aligned}
\end{equation}
We may also write (\ref{augFP}) and (\ref{augG}) in terms of the non-autonomous (``unaugmented'') dynamics:
\begin{equation}\label{eq:aug_gen}
\begin{aligned}
	\big(\bm{\hat{G} f \big) (x)} &= - \partial_{\theta} \bm{f}(\theta,x) + \big( \hat{G}(\theta) \bm{f}(\theta , \cdot) \big) (x)\\
	&= - \partial_{\theta} \bm{f}(\theta,x) - \text{div}_{x} \big( \hat{v} ( \theta,x) \bm{f}(\theta , x ) \big) + \frac{\varepsilon^2}{2} \Delta_x \bm{f}(\theta,x),
\end{aligned}
\end{equation}
where $\hat{G}({\theta})$ is the right-hand side operator of~\eqref{FP2} at time~$t=\theta\in 2\tau S^1$, i.e., the time-$\theta$ differential operator of the Fokker--Planck equation (on $[0,2\tau]$).

\begin{remark}
\quad
\begin{enumerate}[(a)]
\item The periodicity in $\theta$ and the boundary conditions from \eqref{augG} are, as is common in semigroup theory, encoded in the domain of $\smash{\bm{\hat{G}}}$. This domain $\smash{ \mathcal{D}(\bm{\hat{G}}) }$ enforces continuity  conditions  in $\theta$ at~$0,\tau, 2\tau$.

\item For the purposes of the current work we will not require the well-posedness of the problem~\eqref{augG}. We will only need the operator $\bm{\hat{G}}$ defining the right-hand side of this equation, and its relation to the transfer operator family $\hat{\mathcal{P}}_{s,s+t}$, $s<t$. This will be the focus of section~\ref{ssec:augEfun}.  For additional theory for these augmented problems, objects, and related results we refer the reader to~\cite{Chicone1999}.


\item Any non-constant solution $\bm{f} \, : \, (t,\theta ,x) \mapsto \bm{f}_t(\theta,x)$ to \eqref{augG} has three input variables, $t$, $\theta$ and $x$, and may have different regularity properties in each variable. In the following we will focus on eigenfunctions of $\bm{\hat{G}}$ that are of course constant in $t$.

\end{enumerate}
\end{remark}
We now comment on the crucial connection between solutions of~\eqref{augG} and $\hat{\mathcal{P}}_{s,s+t}$, neglecting the issue of solvability.
Note that the stochastic augmented differential equation \eqref{eq:aug1} allows for evolving the non-autonomous equation \eqref{dynsystem} from any initial time $s$ by setting $\theta_0 = s$. In an analogous manner, the augmented Fokker--Planck equation \eqref{augG} with initial condition $\bm{f}_0$ evolves \emph{every} initial condition $ \bm{f}_0(s,\cdot) $, $s\in [0,2\tau)$,---i.e., a configuration of initial conditions---by the non-autonomous reflected Fokker--Planck equation~\eqref{FP2}.
More precisely, the following holds for the evolution of \eqref{augG}:
\begin{equation}
\label{eq:evoFamily}
	\Big(e^{t\bm{\hat{G}}} \bm{f}_0 \Big) (\theta+t \!\!\!\!\mod 2\tau,\cdot ) = \bm{f}_t(\theta + t \!\!\!\! \mod 2\tau,\cdot) = \hat{\mathcal{P}}_{\theta,\theta+t} \big( \bm{f}_0(\theta, \cdot) \big).
\end{equation}


In the terminology of semigroup theory~\cite{Chicone1999,EnNa00} the solution operators of \eqref{augG}, here formally denoted by $\smash{ (e^{t\bm{\hat{G}}} )_{t\ge 0} }$, form an \emph{evolution semigroup} (or Howland semigroup), 
and it is given exactly by~\eqref{eq:evoFamily}. Informally, the action of $\smash{ e^{t\bm{\hat{G}}} }$ in the context of $\hat{\mathcal{P}}_{\theta,\theta +t}$ can be described as follows. On the left-hand side of \eqref{eq:evoFamily}, $\smash{ e^{t\bm{\hat{G}}} }$ takes the initial configuration $\bm{f}_0$ (on all of $\bm{\hat{X}}$) and evolves the entire configuration for the time duration $t$, to obtain $\bm{f}_t$.
The result is then evaluated at the $\theta + t \mod 2\tau$ fiber.
The $\theta+t  \mod 2\tau$ fiber of $\bm{f}_t$ corresponds to the $\theta$ fiber of $\bm{f}_0$ evolved for time $t$ due to the constant \textit{drift} in the $\theta$ variable: $-\partial_{\theta}$, see \eqref{eq:aug_gen}. That equation \eqref{eq:evoFamily} indeed gives the solutions to \eqref{augG}, is a conseqence of \cite[Theorem~6.20]{Chicone1999} adapted to the concatenation of the forward and backward evolutions described by the reflected system~\eqref{FP2}. As we will not require a result of this generality, we omit the details. For our purposes it will be sufficient to consider the special case, where $\bm{f}_0 = \bm{f}$ is an eigenfunction of~$\bm{\hat{G}}$. This is done next.

\subsection{Eigenfunctions of the time-augmented generator}
\label{ssec:augEfun}

Analogously to \cite[Lemma 22]{FK17} the following result holds. It can be obtained from~\eqref{eq:evoFamily} by noting that every eigenpair $(\mu,\bm{f})$ of $\bm{\hat{G}}$ gives a solution to \eqref{augG} by~$\bm{f}_t = e^{\mu t}\bm{f}$. However, to highlight the intuitive connection between the augmented generator and the non-autonomous problem, we prove the following proposition by other means.
\begin{proposition}\label{augmen}
Let $\bm{f}$ be an eigenfunction of $\bm{\hat{G}}$ corresponding to the eigenvalue $\mu \in \mathbb{C}$.
One has then
	\begin{equation}\label{eq:eigenfunction}
		\hat{\mathcal{P}}_{s,s+t} \bm{f}(s,\cdot)=e^{\mu t}\bm{f}(s+t \!\!\!\!\mod  2\tau , \cdot )
	\end{equation}
	for all $s \in 2\tau S^1$ and $t \geq 0$.
\end{proposition}
\begin{proof}
	We will prove \eqref{eq:eigenfunction} following ideas from \cite{FK17}. First we apply Theorem \ref{thm:NACP} piecewise on $[0,\tau]$ and $[\tau,2\tau]$ concerning well-posedness and regularity. Therefore we consider the original problem \eqref{FP0} and the reflected, shifted, time-reversed problem \eqref{FP2}. Now theorem \ref{thm:NACP} guarantees for any initial condition $f_0 \in L^p(X)$, $p \in (1,\infty)$, the unique existence of a function $f$ with the regularity
\begin{equation*}
	f \in C\big([0,2\tau];L^p(X)\big) , \quad f|_{[0,\tau]} \in C^1\big((0,\tau];L^p(X)\big), \quad f|_{[\tau,2\tau]} \in C^1\big((\tau,2\tau];L^p(X)\big)
\end{equation*}
and the properties
\begin{equation*}
	f|_{[0,\tau]} \text{ solves \eqref{FP0}}, \quad f|_{[\tau,2\tau]} \text{ solves \eqref{FP2}}, \quad f(t) = \hat{\mathcal{P}}_{s,t} f(s) \; , s<t \in [0,2\tau] \; .
\end{equation*}
Further $f(t) \in \mathcal{D}(\hat{G}(t))$ holds for all $t \in (0,2\tau]$.
Now we can proceed as in \cite[Lemma 22]{FK17}.
\label{proof:aug}
Let $\mu \in \mathbb{C}$ and $\bm{f} \in \mathcal{D}(\bm{\hat{G}})$ with $ \bm{\hat{G}}\bm{f} = \mu \bm{f}$. According to the construction above \eqref{eq:aug_gen} we know
\begin{equation*}
	\mu \bm{f} (\theta,\cdot ) = \bm{\hat{G}}\bm{f} (\theta, \cdot ) = -\partial_{\theta} \bm{f} (\theta , \cdot) + \hat{G}(\theta) \bm{f} (\theta,\cdot)
\end{equation*}
and this implies, in accordance with \eqref{augG}, for all $\theta \in 2\tau S^1\backslash \{0,\tau\}$
\begin{equation}
\label{shiftFP}
	\partial_{\theta} \bm{f} (\theta,\cdot ) = ( \hat{G}(\theta) - \mu ) \bm{f}(\theta,\cdot).
\end{equation}
Now $\hat{\mathcal{P}}_{\theta,\theta+t}$ is the evolution operator to the evolution equation
\begin{equation*}
	\partial_{\theta} u(\theta) = \hat{G} (\theta) u(\theta)
\end{equation*}
Therefore the function $e^{-\mu t} \hat{\mathcal{P}}_{\theta,\theta+t} \bm{f}(\theta)$ solves \eqref{shiftFP} uniquely and \ref{thm:NACP} guarantees continuity in $\theta$. Therefore we can connect the eigenfunctions of the augmented reflected generator $\bm{\hat{G}}$ with the evolution given by $\hat{\mathcal{P}}_{s,s,+t}$ for all $s,t$ as stated in the claim.
\end{proof}
Let $\mu$ be an eigenvalue of $\bm{\hat{G}}$ with an eigenfunction $\bm{f}$.
Inserting  $s = 0$ and $t = 2 \tau$ into Proposition \ref{augmen} yields
\begin{equation}\label{eq:eig_func-sing}
	\hat{\mathcal{P}}_{0,2\tau} \bm{f}(0,\cdot)=e^{\mu 2\tau} \bm{f}(2 \tau,\cdot ).
\end{equation}
This is a spectral mapping theorem type of result, as it connects the eigenvalues and eigenfunctions of the evolution operator $\hat{\mathcal{P}}_{s,s+2\tau}$ with those of an associated (infinitesimal) generator~$\bm{\hat{G}}$. We refer the reader to standard literature on classical results for operator semigroups~\cite{Paz83,EnNa00}.

Recalling that  $\hat{\mathcal{P}}_{0,2\tau}=\mathcal{P}_{0,\tau}^{\ast} \mathcal{P}_{0,\tau}$ is a compact self-adjoint positive operator, it must be that $e^{\mu 2\tau}=\sigma^2$ for some $0 < \sigma \in \R$.
This implies
\begin{equation}\label{eq:eig_im}
	0 < \sigma = \left( e^{2\mu \tau} \right)^{\frac{1}{2}} = ((e^{\mu \tau})^2)^{\frac{1}{2}} = e^{\tau \Re(\mu)} \left( (\cos ( \tau \Im(\mu)) + i \sin (\tau \Im (\mu)) )^2\right)^{\frac{1}{2}},
\end{equation}
from which it follows that $\Im (\mu) = \frac{k\pi}{\tau} $ for some $k \in \mathbb{Z}$.

\begin{remark}
Theorem~\ref{thm:NACP} guarantees that for initial conditions $f_s \in \mathcal{D}(G(s))$ the solution $\hat{\mathcal{P}}_{s,s+t}f_s = f(t) $ is in $t$ a continuous mapping to the domain of the generator, $\mathcal{D}(G(s+t)) = \mathcal{D}_p$. Theorem~\ref{thm:regularity} further gives for each eigenfunction $\bm{f}$ that $\bm{f}: \theta\mapsto \bm{f}(\theta,\cdot) \in C(2\tau S^1 ; \mathcal{D}_p)$ and that $\bm{f} \in C(\bm{\hat{X}})$. This regularity is utilized in the proof of Theorem~\ref{thm:cohratio} below.
\end{remark}

\subsection{Coherent families of sets}
In the specific case where the velocity field $v$ is periodic in time, \cite{FK17} shows that the families of sets
\begin{equation*}
	A_{\theta}^+ := \left\{ \bm{f}(\theta, \cdot ) \geq 0 \right\} \qquad
	A_{\theta}^- := \left\{ \bm{f}(\theta , \cdot ) \le 0 \right\}
\end{equation*}
have an escape rate (see \cite[Definition 8]{FK17}) of at most $\text{Re}(\mu_2)$, where $\mu_2$ is the first nontrivial eigenvalue of $\bm{\hat{G}}$ corresponding to the eigenfunction $\bm{f}$.
Because we consider the dynamics on a finite time interval, this notion of escape rate is replaced by the concept of a \emph{coherence ratio} \cite{Froyland2013}.
In the general setting  of aperiodic $v$ we will quantify the coherence of families $\{ A^{\pm}_{\theta}\}_{\theta \in [0,\tau]}$ and provide a construction of highly coherent families with associated rigorous coherence bound.
\begin{definition}[Coherence ratio]
Let~$\{A_t\}_{t\in[0,\tau]}$ be a family of measurable sets. Denote by~$\mathbb{P}_{m}$ the law of the process~$\{x_t\}_{t \in [0,\tau]}$ generated by the SDE~\eqref{dynsystem} initialised with $x_0\sim m$, where $m$ denotes normalised Lebesgue measure on~$X$. For $m (A_0) \neq 0$ we define the  coherence ratio of the family $\{A_t\}_{t\in [0,\tau]}$ as
\begin{equation}\label{eq:cohratio}
	\rho_m(\{A_t\}_{t\in [0,\tau]}) = \frac{\mathbb{P}_{m}\left(\cap_{t\in[0,\tau]}\{x_t\in A_t\}\right)}{m(A_0)}\,.
\end{equation}
\end{definition}
It was shown in~\cite[Appendix A.6]{FK17} that for a family of sets with sufficient regularity (called ``sufficient niceness'' therein) the quantity~\eqref{eq:cohratio} is well defined. Here we will alleviate this requirement entirely by showing regularity of the augmented eigenfunctions~$\bm{f}$, and showing that this is sufficient to prove the desired results.

Theorem \ref{thm:cohratio} makes a link between the coherence of a particular family of sets defined by zero super/sublevel sets of an eigenfunction of $\bm{\hat{G}}$ and the corresponding eigenvalue $\mu$.
It shows that the probability of a trajectory remaining in a family of sets constructed from the positive and negative parts of eigenfunctions of $\bm{\hat{G}}$ decays no faster than the rate given by the corresponding eigenvalues.
This result extends to aperiodically driven continuous-time systems, similar results for autonomous systems in discrete time \cite{FrSt10} and continuous time \cite{FJK13}, and periodically driven dynamics in continuous time~\cite{FK17}.
\begin{theorem}\label{thm:cohratio}
Let~$\bm{\hat{G}}\bm{f} = \mu\bm{f}$ with~$\mu<0$. If~$\bm{f}$ is scaled such that~$\|\bm{f}(\tau,\cdot)\|_{L^1}=2$, then it holds for the family~$\{A_t^\pm\}_{t \in [0,\tau]}$ of sets with~$A_t^\pm = \{\pm\bm{f}(t,\cdot)\ge 0\}$ that
\begin{equation}\label{eq:cohratbound}
	\rho_m(\{A_t^\pm\}_{t\in [0,\tau]}) \ge \frac{e^{\mu\tau}}{\|\bm{f}(0,\cdot)\|_{L^{\infty}} |A_0^\pm|}\,,
\end{equation}
where $|A|$ denotes the non-normalised Lebesgue measure of the set $A$.
In particular, eigenfunctions at eigenvalues $\mu\approx 0$ yield families of sets with high coherence ratio.
\end{theorem}
\begin{proof}
See Appendix~\ref{app:cohratio_proof}. We note that by Corollary \ref{cor:conti}  $\bm{f}$ is continuous on $\bm{\hat{X}}$.
\end{proof}
Intuitively, the left hand side of (\ref{eq:cohratbound}) quantifies the likelihood of \emph{escape} from the family of sets, and the right hand side of (\ref{eq:cohratbound}) is a  (scaled) measure of \emph{mixing};  the bound says that the likelihood of escape is less than the mixing incurred over the same time duration.
The bound is not intended to be sharp;  we remark that one could optimise the level set cutoff to improve the ratio $\rho_m$, as has been done in previous work on coherent sets \cite{froyland_padberg_09}.

In Section \ref{ssec:BickleyCohset}, we consider the dominant 6 eigenvectors of $\bm{\hat{G}}$ and apply sparse eigenbasis approximation (SEBA) \cite{seba} to find a sparsity-inducing rotation of this eigendata and separate individual slow escape / slow mixing subdomains.
The following proposition generalises Theorem \ref{thm:cohratio} so that it may apply to vectors formed from linear combinations of eigenvectors.
\begin{proposition}\label{thm:cohratio2}
Let~$\bm{\hat{G}}\bm{f}_i = \mu\bm{f}_i$, $i=1,\ldots,k\in\mathbb{N}$, with~$\mu_k \le\ldots\le\mu_1\le 0$.
For $\bm{f} = \sum_{i=1}^k \alpha_i \bm{f}_i$ with $\alpha_i\in\R$, the statement of Theorem~\ref{thm:cohratio} remains true with $\mu = \mu_k$ if
\begin{equation}
\label{eq:contribs}
\alpha_i \int_{A_{\tau}^+} \bm{f}_i(\tau,\cdot)\,dm \ge 0\quad\text{for }i=1,\ldots,k.	
\end{equation}
\end{proposition}
The proof Proposition~\ref{thm:cohratio2} is deferred to
Appendix~\ref{app:cohratio_proof}.

In the computations performed in the next sections we will numerically approximate $\bm{\hat{G}}$, compute its upper spectrum and associated eigenfunctions, and plot super/sublevel sets of the eigenfunctions.
In Section \ref{ssec:BickleyCohset} we will additionally apply SEBA, plot the sparse basis functions, and one of the superlevel sets.



\section{Computational aspects}
\label{sec:comput}

\subsection{Numerical discretization}
\label{ssec:num_disc}

We use the ``Ulam's discretization for the generator'' approach developed in \cite{FJK13} for autonomous flows and extended in~\cite[Sections 7.2 and 7.3]{FK17} for nonautonomous flows.
In brief, referring to the above papers for further details, the Ulam discretization for the generator yields a matrix that may be interpreted as a rate matrix of a finite-state, continuous-time Markov chain, with the states corresponding to a partition of $\tau S^1\times X$ into hypercubes (hyperrectangles) in~$\mathbb{R}^{d+1}$.
The entries which correspond to rates between boxes adjacent in the temporal coordinate direction (in which the time evolution is a rigid rotation of constant velocity 1) are given by~$1/h$, where~$\tau S^1$ is discretized into intervals of length~$h$.
The entries in the remaining $d$ space directions are computed from the rate of flux out of the hypercube faces by numerical integration 
of the component of the velocity field normal to (and pointing out of) the face;  see the expression for $G_n^{\mathrm{drift}}$ in~\cite[Section 7.2]{FK17}.
The entries of the rate matrix corresponding to diffusive dynamics (in the $d$ space coordinates only, there is no diffusion in the time coordinate) are computed from a finite-difference approximation of the Laplace operator;  see the expression for $G_n^{\mathrm{diff}}$ in~\cite[Section 7.2]{FK17}.
We then set $G_n:=G_n^{\mathrm{drift}}+G_n^{\mathrm{diff}}$.
The matrix $G_n$ can also be interpreted as a rate matrix for a finite-state Markov chain;  it has an eigenvalue $0$ and its spectrum is confined to the left half of the complex plane.

The reflected velocity field $\hat{v}$ from (\ref{hatvdef}) may be substituted for the velocity field $v$ used in \cite{FK17} and the methodology of \cite{FK17} employed;  this is the approach taken in the numerical experiments below.

\begin{remark}
The computation of $G_n^{\mathrm{drift}}$ uses only the outward-pointing velocity field values on the faces of the partition elements---similarly to how the outward flux is defined through the positive part of the inner product in~\eqref{augoutflux0}---, discarding the inward-pointing parts.
Because of the reflected structure of $\hat{v}$, a slightly more efficient implementation would be to store the evaluations of the velocity field normal to hypercube faces in both directions (not only in the outward-pointing direction).
The outward-pointing components would be used on the time interval $[0,\tau]$, while the inward-pointing components would be used on the time interval $(\tau,2\tau)$, where they are outward-pointing because of the sign flip in~\eqref{hatvdef}---similarly as it happens in the proof of Proposition~\ref{prop:outflux}.
This would reduce by half the computational effort in evaluating the velocity field components normal to the hypercube faces.
However, the assembly of the generator matrices is relatively fast anyway, and we have not tried to optimize our implementation of Ulam's method for the generator in this reflected setting.
\end{remark}

\subsection{Example: Periodically driven double gyre}
\label{ssec:perDG}

We consider the periodically driven double gyre system~\cite{shadden2005definition}:
\begin{equation*}
	x'(t) = - \pi A \sin (\pi f(t,x)) \cos(\pi y)
	\qquad y'(t) = \pi A \cos ( \pi f(t,x)) \sin (\pi y) \frac{df}{dx}(t,x)
\end{equation*}
on the time interval~$[0,\tau]=[0,4]$.
The forcing is $f(t,x)= \gamma \sin(2\pi \Omega t)x^2 + (1-2\gamma \sin(2\pi t))x$ and the parameters are $A=0.25$, $\Omega=2\pi$, and $\gamma =0.25$, implying the forcing period $1$, on the spatial domain~$X= [0,2] \times [0,1]$.
This system has been a standard example of coherent sets~\cite{FrPa14}. 
The purpose of this section is to show that our method reliably computes the singular functions and values of $\mathcal{P}_{0,\tau}$;  further analysis is deferred to later sections. 
In particular, we will revisit this example in context of optimal manipulation of these coherent sets in sections \ref{ssec:incr_coh} and \ref{ssec:decr_cohDG}.

The augmented reflected generator approach with a resolution of $40 \times (100\times 50)$ and noise intensity $\varepsilon=0.1$ gives us the non-trivial dominant eigenvectors of $\bm{\hat{G}}$ at time $t=0$ shown in Figure~\ref{fig:double-eig}.
\begin{table}[htb]
\centering
\begin{tabular}{c|l||c|l||c|l||c|l}
$\mu_1$ & $ 0$ & $\mu_4$ & $-0.35061$ &$\sigma_1 $ & $1$ & $\sigma_4 $ & $0.24599$\\
\hline
$\mu_2$ & $-0.09033$ & $\mu_5 $ & $-0.44766$ & $\sigma_2 $ & $0.69674$ & $\sigma_5 $ & $0.16685$\\
\hline
$\mu_3$ & $-0.34938$ & $\mu_6 $ & $-0.45702$ &$\sigma_3 $ & $0.24720$ & $\sigma_6 $ & $0.16072$
\end{tabular}
\caption{Eigenvalues ($\mu_k$) of $\bm{\hat{G}}$ ordered in ascending magnitude and corresponding approximate singular values ($\sigma_k$) of $\mathcal{P}_{0,\tau}$ according to~\eqref{eq:eig_im}.}
\label{tab:doublegyre}
\end{table}
Ordered by ascending magnitude, the $3$rd, $4$th and $6$th eigenvalues (Table \ref{tab:doublegyre}) and eigenvectors (not shown) correspond to features also detected in~\cite{FK17}, where they were connected to complex non-companion eigenvalues---the concept of companion eigenvalues will be introduced around~\eqref{eq:companion} below. These features become less coherent, i.e., their respective real eigenvalues decrease compared with the others, as the length $\tau$ of the time interval increases.

\begin{figure}[htbp]
\centering
\begin{subfigure}{0.49\textwidth}
\centering
\includegraphics[height=27mm]{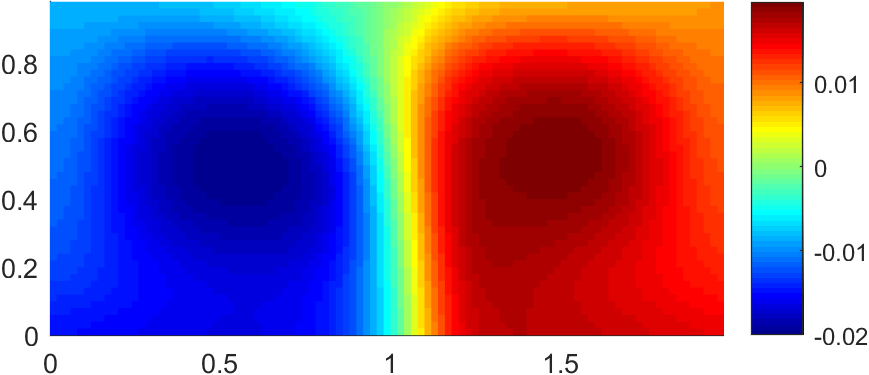}
\subcaption{$2$nd  eigenvector.}
\end{subfigure}
\begin{subfigure}{0.49\textwidth}
\centering
\includegraphics[height=27mm]{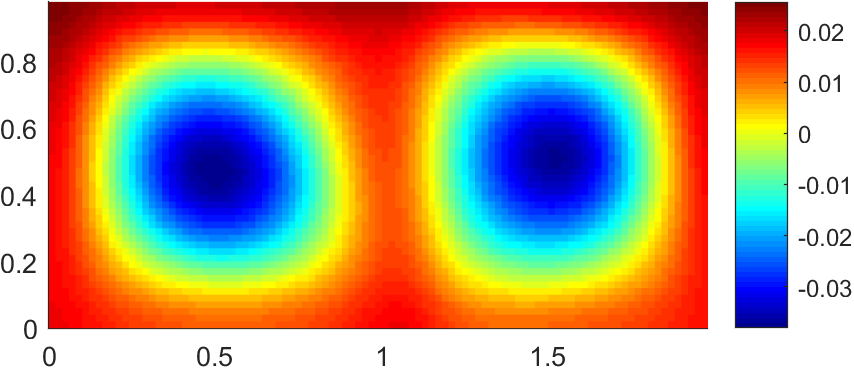}
\subcaption{$5$th eigenvector.}
\end{subfigure}

\caption{Time slice $t=0$ of the $2$nd and $5$th eigenvector of $\bm{\hat{G}}$ for the double gyre flow.}\label{fig:double-eig}
\end{figure}

\subsection{Example: Bickley jet}
\label{ssec:BickleyCohset}

We now apply the reflected augmented generator approach to a perturbed Bickley Jet~\cite{Rypina2007}. The following model describes an idealized zonal jet in a band around a fixed latitude, assuming incompressibility, on which two traveling Rossby waves are superimposed. The velocity field $v = (-\frac{\partial \Psi}{\partial y},\frac{\partial \Psi}{\partial x})$ is induced by the stream function
\begin{equation*}
	\Psi(t,x,y) = -U_0 L \tanh \left(\frac{y}{L}\right) + U_0 L \sech^2 \left( \frac{y}{L}\right) \sum_{n=2}^3 A_n \cos \left(k_n (x - c_nt)\right).
\end{equation*}
The constants are chosen according to \cite{Rypina2007}. The length unit is $Mm$ ($1 \, Mm = 10^6 m $) and the time unit is days. For the amplitudes $A_n$ and the speed of the Rossby waves $c_n$ we choose
\begin{equation*}
	c_2 = 0.205U_0, \qquad c_3 =0.461U_0, \qquad A_2 = 0.1, \qquad A_3 = 0.3, \qquad r_e = 6.371,
\end{equation*}
with $r_e$ being the Earth's radius. Further we choose
\begin{equation*}
U_0 = 5.4138, \qquad L = 1.77, \qquad k_n = \frac{2n}{r_e}.
\end{equation*}
The state space is periodic in $x$ direction and is given by $X = \pi r_e S^1 \times [-3,3]$ (in accordance with other literature), and the time interval will be chosen as $[0,\tau] = [0,9]$.
For good numerical tractability, we resolve our reflected space-time manifold with a spatially somewhat coarse $108 \times (120 \times 36)$ grid, that is uniform in space ($120\times 36$) leads to square boxes needed for isotropic diffusion) and sufficiently finely resolved in time ($108$). We choose $\varepsilon = 0.1$.

The system described above is equipped with homogeneous Dirichlet boundary conditions instead of homogeneous Neumann conditions on~$\partial X$. This leads to a slightly different spectral structure of the generator, which now generates a semigroup of sub-Markovian operators. Its leading eigenvalue is strictly less than zero. We expect Theorem~\ref{thm:cohratio} to hold in the case of Dirichlet boundary conditions. One possible theoretical justification would require ``close'' the open system that is represented by the homogeneous Dirichlet boundary conditions by introducing a virtual ``external'' state, then apply the Neumann theory to that system. The details would lead beyond the scope of this work, and will be discussed elsewhere.

We highlight that the computations we are about to perform here are different to those performed in \cite[Section 7.6]{FK17} in at least two respects.
Firstly, the Bickley jet under investigation is aperiodically driven, in contrast to the periodically driven Bicklet jet in \cite[Section 7.6]{FK17} (we use slightly different parameters in the velocity field).
Secondly, we wish to find functions that decay the least under \emph{finite-time} evolution, in contrast to the problem considered in \cite{FK17}, which sought functions that decayed at the slowest time-asymptotic ($t\to\infty$) rate under periodic driving.
In particular, even for a periodically driven Bickley jet, the finite-time question considered in the present paper is different to the infinite-time question addressed in \cite{FK17}.
Thus, even though we chose a flow interval of length~9 as in \cite{FK17}, the problem in consideration is different.

Analogously to \cite[Section 7]{FK17} our time-augmentation produces companion eigenvalues. Companion eigenmodes denote eigenmodes that are ``higher order harmonics'' of existing eigenmodes differing only in temporal modulation, and encoding the same coherence information; see below. For more details on the companion eigenvalues for the Ulam-discretization we refer to~\cite[Section 7.3]{FK17}. We will use and verify the relations derived there. Therefore we calculate the eigenvalues and vectors of $\bm{\hat{G}}$ with the smallest magnitude instead of largest real part using \texttt{eigs(G,10,'SM')} in Matlab.

\begin{table}[htbp]
\footnotesize
\begin{tabular}{c|l||c|l||c|l||c|l}
$\mu_1$ & $-0.02523$ & $\mu_6 $ & $-0.29908$ &$\sigma_1 $ & $0.79690$ & $\sigma_6 $ & $0.06776$\\
\hline
$\mu_2$ & $-0.21086$ & $\mu_7 $ & $-0.03534 -0.35003i$ & $\sigma_2 $ & $0.14990$ & $\sigma_7 $ & $-0.72750 +0.00634i$\\
\hline
$\mu_3$ & $-0.25710$ & $\mu_8 $ &$-0.03534 +0.35003i$ &$\sigma_3 $ & $0.09887$ & $\sigma_8 $ &  $-0.72750 -0.00634i$\\
\hline
$\mu_4$ & $-0.25836$ & $\mu_9 $ & $-0.39208$ &$\sigma_4 $ & $0.09776$ & $\sigma_9 $ & $0.02934$\\
\hline
$\mu_5$ & $-0.29905$ & $\mu_{10} $ & $-0.21995 - 0.33451i$ & $\sigma_5 $ & $0.06778$ & $\sigma_{10} $ & $-0.13695 +0.01805i$
\end{tabular}
\caption{Eigenvalues ($\mu_k$) of $\bm{\hat{G}}$ ordered in ascending magnitude and corresponding approximate singular values ($\sigma_k$) of $\mathcal{P}_{0,\tau}$ according to~\eqref{eq:eig_im}. The eigenvalues $\mu_7, \mu_8, \mu_{10}$ correspond to companion modes. They do not yield purely real singular values $\sigma_7,\sigma_8, \sigma_{10}$ through the exponentiation~\eqref{eq:eig_im}, because the numerically computed companions~\eqref{eq:companion} contain a bias induced by discretization; see \cite[Section 7.3]{FK17} for further details.}
\label{tab:bickley}
\end{table}

Table~\ref{tab:bickley} shows a gap after the first and sixth eigenvalue. Let us first discuss the leading $6$ eigenvectors. Figure~\ref{fig:bickley-eig} shows the eigenvectors corresponding to the dominant (i.e., smallest real part) $6$ eigenvalues. The first eigenvector, the quasistationary (or conditionally invariant) distribution highlights (in blue, see Figure~\ref{fig:bickley-eig}(a)) parts of the domain that get pushed out of the region $X$ of consideration.
The red regions are those parts of phase space that remain longest in $X$ under the diffusive dynamics~\eqref{dynsystem}.
This effect is due to the outflow conditions. Note that this example is in this sense explorative, as our theory in the previous sections was only considering reflecting and not outflow boundary conditions.

The second eigenvector indicates an upper/lower separation. The other four eigenvectors show combinations of coherent vortices.

\begin{figure}[htbp]
\centering
\begin{subfigure}{0.49\textwidth}
\includegraphics[width=\textwidth]{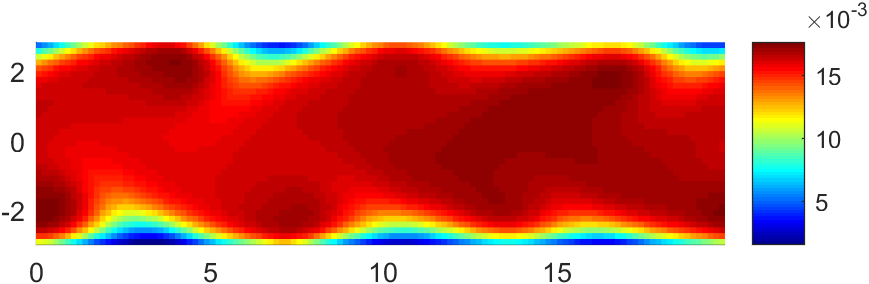}
\subcaption{First eigenvector (initial time).}
\end{subfigure}
\begin{subfigure}{0.5\textwidth}
\includegraphics[width=\textwidth]{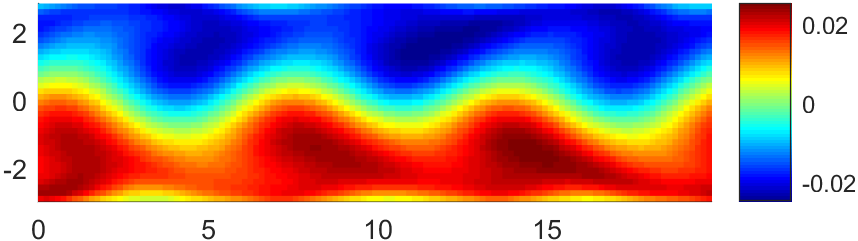}
\subcaption{Second eigenvector (initial time).}
\end{subfigure} \\

\begin{subfigure}{0.49\textwidth}
\includegraphics[width=\textwidth]{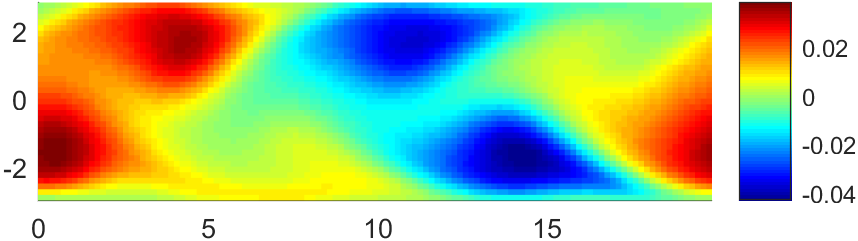}
\subcaption{Third eigenvector (initial time).}
\end{subfigure}
\begin{subfigure}{0.49\textwidth}
\includegraphics[width=\textwidth]{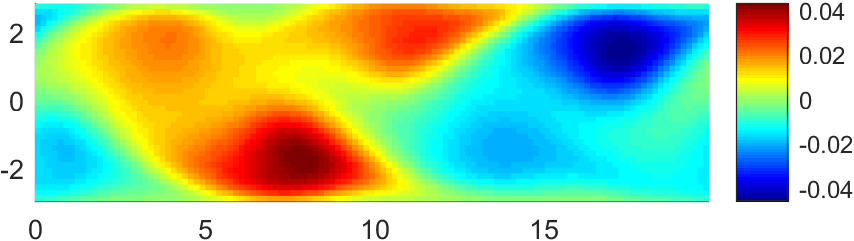}
\subcaption{Fourth eigenvector (initial time).}
\end{subfigure}\\

\begin{subfigure}{0.49\textwidth}
\includegraphics[width=\textwidth]{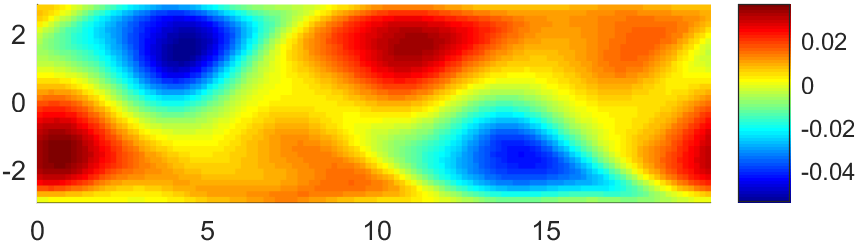}
\subcaption{Fifth eigenvector (initial time).}
\end{subfigure}
\begin{subfigure}{0.49\textwidth}
\includegraphics[width=\textwidth]{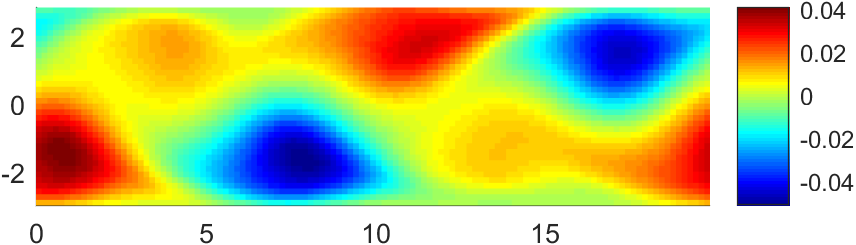}
\subcaption{Sixth eigenvector (initial time).}
\end{subfigure}

\caption{Approximate leading eigenvectors of $\bm{\hat{G}}$---that are according to \eqref{eq:eig_func-sing} singular vectors of $\mathcal{P}_{0,\tau}$---computed from the Ulam discretization of the reflected augmented generator with a $108 \times (120 \times 36)$ time-space resolution.}
\label{fig:bickley-eig}
\end{figure}

To investigate which elements of the spectrum of $\bm{\hat{G}}$ contain genuinely new dynamical information, we checked that the complex eigenvalues of~$\bm{\hat{G}}$ from $\mu_7$ to $\mu_{50}$ are all companion eigenvalues equal to (recall from section~\ref{ssec:num_disc} that $h$ is the temporal grid spacing)
\begin{equation} \label{eq:companion}
	\mu - \mu^{(k)} \quad \text{with} \quad \mu^{(k)} = \dfrac{1-\omega^k}{h}, \quad \omega = \exp \left( 2\pi i \frac{h}{2\tau}\right)
\end{equation}
for an eigenvalue $\mu$ of $\bm{\hat{G}}$ and a $k\in\mathbb{Z}$, as derived in~\cite[Section 7.3]{FK17}. Under the assumption that the eigenvector $\bm{w}$ is sufficiently smooth in time and time is sufficiently resolved (i.e., $w_t \approx w_{t-h}$), each eigenpair $(\mu,\bm{w})$ of the discretized generator $\bm{\hat{G}}$ has an approximate companion pair $( \mu-\mu^{(k)}, \bm{w} \psi_k)$, where $\bm{w}\psi_k$ is understood as pointwise multiplication and $\psi_k(t)=\omega^{kt}$, which only varies in time but not in space.
For additional verification we can check the correlation of the companion eigenvectors
\begin{equation}\label{eq:correl}
	c_n^{m}(k) := \dfrac{\langle  \psi_{k} \bm{w}_m, \bm{w}_n \rangle}{\| \psi_{k} \bm{w}_m \|_2 \| \bm{w}_n \|_2}
\end{equation}
as in \cite[Section 7.5]{FK17}. For instance, by looking at $\mu_1$ in Table~\ref{tab:bickley} and noting that $\mu^{(\pm 1)} = 0.01015\pm 0.34887i $ ($h=18/108$), we find that $\mu_7, \mu_8 \approx \mu_1 - \mu^{(\pm 1)}$ are candidates for companion eigenvalues for~$\mu_1$. The small difference in the numerical values of $\mu_1 - \mu_{7,8}$ and the shift $\mu^{(\pm 1)}$, around $2.3\cdot 10^{-3}$ in magnitude, is due to the first eigenvector not being constant in time, i.e., merely~$w_t \approx w_{t+h}$. Nonetheless, the complex eigenvalues $\mu_7$ and $\mu_8$ are companion eigenvalues to~$\mu_1$. This is supported by the correlation for the corresponding eigenvectors
\begin{equation*}
c_{7,8}^{1}(\pm 1) = \dfrac{ \langle \psi_{\pm 1}\bm{w}_1, \bm{w}_{7,8} \rangle_{\mathbb{R}^{120 \cdot 36 \cdot 108}}}{\| \psi_{\pm 1} \bm{w}_1 \|_2 \| \bm{w}_{7,8} \|_2} = 0.84597 \pm 0.53312i,\quad\text{i.e. } \big|c_{7,8}^{1}(\pm 1)\big| = 0.9999,
\end{equation*}
while the correlation with other eigenfunctions $n\in \{2,\ldots, 10\} \backslash \{ 7,8 \}$ satisfies $|c_{n}^1(\pm 1)| \leq 0.00323$. The construction above \eqref{eq:eig-sing} implies that every singular value is real. Our numerical calculations strongly suggest that within the first $50$ eigenvalues every complex eigenvalue is a companion to a real eigenvalue with smaller magnitude. The correlations using \eqref{eq:correl} yield results similar to those stated in the special case above.

\subsection{Vortex isolation by sparse eigenbasis approximation}
\label{ssec:SEBA}

The space-time signatures of six coherent vortices in the Bickley flow are captured in the leading six singular vectors shown in Figure \ref{fig:bickley-eig} (note only the initial time slice is displayed).
In order to isolate these six vortices in space-time, we apply an orthogonal rotation and some sparsification to the six-dimensional subspace of $\mathbb{R}^{108\times (120\times 36)}$ spanned by the leading six (space-time) eigenvectors shown in Figure \ref{fig:bickley-eig}.
The orthogonal rotation is chosen so as to construct an approximating basis of six sparse vectors.
To find such a sparse approximating basis, we applied the SEBA (Sparse EigenBasis Approximation) algorithm (see \cite[Algorithm 3.1]{seba}).
The six sparse basis vectors $\bm{\varphi}_k$, $k=1,\ldots,6$, produced by this algorithm are shown in Figure \ref{fig:seba}, and each of these vectors strongly isolates a single vortex.
We emphasise that we input the full space-time vectors to the SEBA algorithm, but in Figure \ref{fig:seba} display only the initial time slice.

\begin{figure}[tbp]
\centering

\begin{subfigure}{0.49\textwidth}
\includegraphics[width=\textwidth]{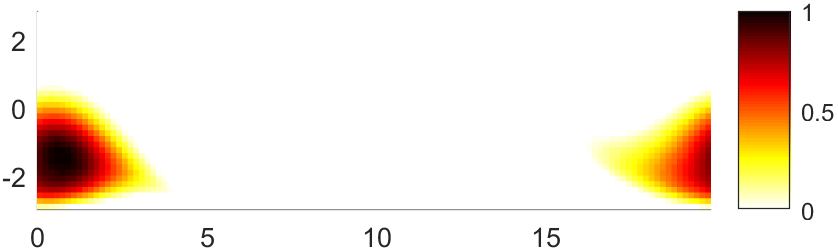}
\subcaption{First vector of SEBA output}
\end{subfigure}
\hfill
\begin{subfigure}{0.49\textwidth}
\includegraphics[width=\textwidth]{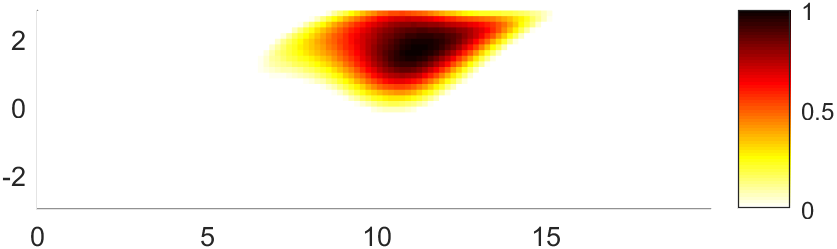}
\subcaption{Second vector of SEBA output}
\end{subfigure} \\
\begin{subfigure}{0.49\textwidth}
\includegraphics[width=\textwidth]{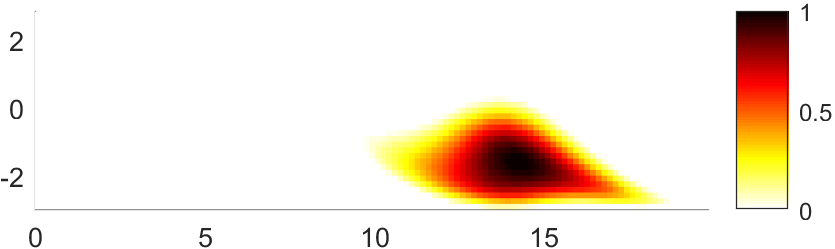}
\subcaption{Third vector of SEBA output}
\end{subfigure}
\hfill
\begin{subfigure}{0.49\textwidth}
\includegraphics[width=\textwidth]{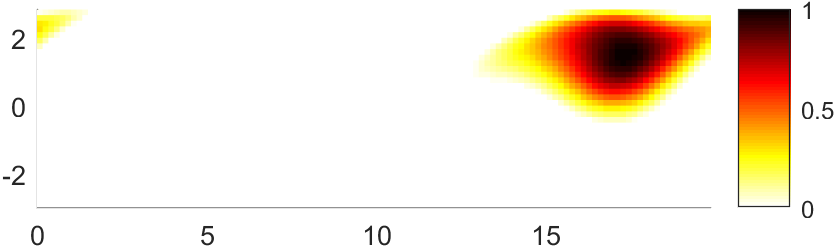}
\subcaption{Fourth vector of SEBA output}
\end{subfigure}\\
\begin{subfigure}{0.49\textwidth}
\includegraphics[width=\textwidth]{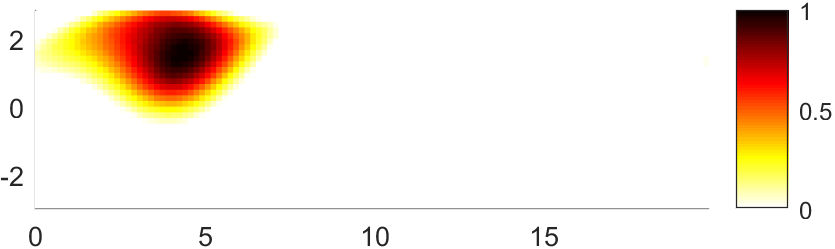}
\subcaption{Fifth vector of SEBA output}
\end{subfigure}
\hfill
\begin{subfigure}{0.49\textwidth}
\includegraphics[width=\textwidth]{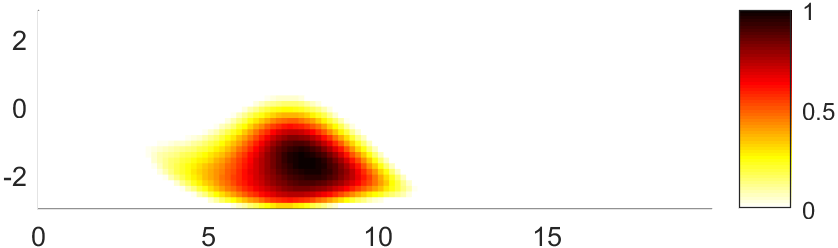}
\subcaption{Sixth vector of SEBA output}
\end{subfigure}
\caption{Space-time estimates of coherent sets extracted from the leading six eigenvectors using SEBA \cite{seba} (initial time slices shown only).}
\label{fig:seba}
\end{figure}

In Figure~\ref{fig:BickleyCohSimu} we seed particles inside the calculated vortical features (in the super-level set $\{\bm{\varphi}_k(0,\cdot) > 0.4\}$ if $\bm{\varphi}_k$ is scaled to have maximum-norm 1) and evolve them forward in time to visualize the coherence. In addition to the deterministic evolution we also visualize a stochastic evolution using $\varepsilon = 0.1$ as in our calculations above. Both simulations use a fourth-order Runge--Kutta (-Maruyama) scheme with step size $\frac{9}{4\cdot108} = \frac{1}{48}$. Figures~\ref{fig:BickleyCohSimu}(b) and (c) demonstrate the coherence of the single vortex in Figure~\ref{fig:BickleyCohSimu}(a).

\begin{figure}[tbp]
\centering

\begin{subfigure}{0.49\textwidth}
\centering
\includegraphics[width = \textwidth]{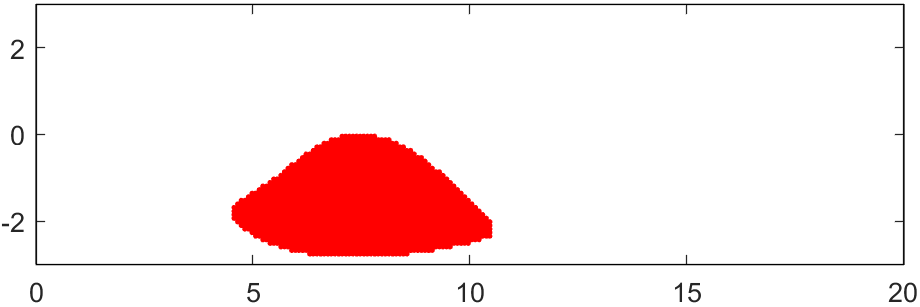}
\subcaption{Initial particles seeded in the gyre induced by the sixth SEBA-vector.}
\end{subfigure}
\begin{subfigure}{0.49\textwidth}
\centering
\includegraphics[width = \textwidth]{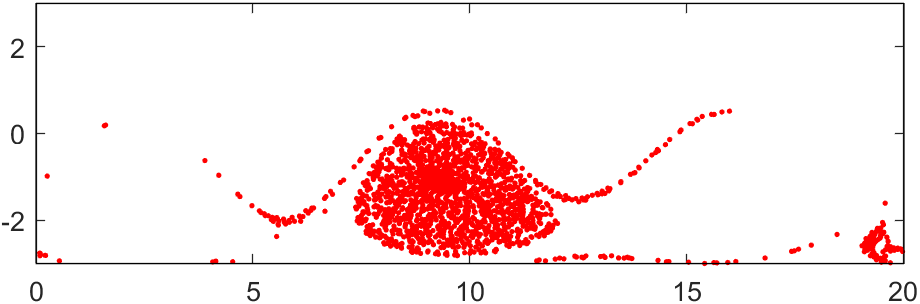}
\subcaption{Final time configuration of the seeded particles evolved by the Bickley jet flow.}
\end{subfigure}
\begin{subfigure}{0.49\textwidth}
\centering
\includegraphics[width = \textwidth]{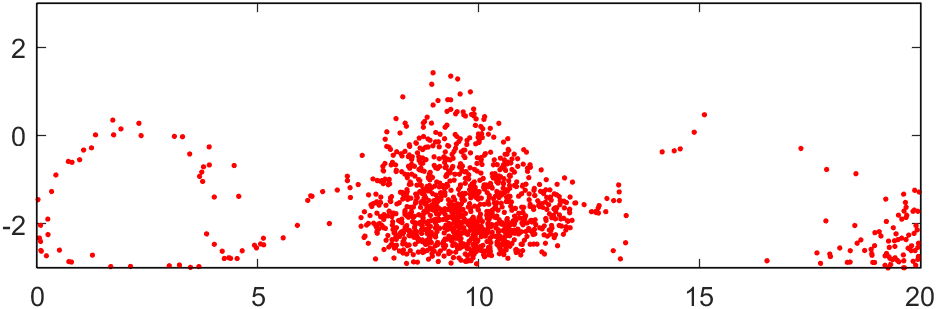}
\subcaption{Final time configuration of the seeded particles evolved by the Bickley jet flow with noise $\varepsilon = 0.1$.}
\end{subfigure}
\caption{Illustration of a coherent set provided by SEBA applied to the leading numerical eigenvectors.}
\label{fig:BickleyCohSimu}
\end{figure}

We wish to apply Proposition \ref{thm:cohratio2} to further demonstrate that the positive parts of the $\bm{\varphi}_k(0,\cdot)$ represent coherent sets.
Because the $\bm{\varphi}_k(0,\cdot)$, $k=1,\ldots,6$ do not \emph{exactly} span the leading six-dimensional eigenspace of $\bm{\hat{G}}$, Proposition \ref{thm:cohratio2} does not directly apply.
Nevertheless, using the linear combinations of eigenfunctions $\bm{f}_i$ that result in the SEBA-features $\bm{\varphi}_k(0,\cdot)$ we found that the hypotheses of Proposition~\ref{thm:cohratio2} were satisfied with the following modification: wherever the contributions \eqref{eq:contribs} were negative, the corresponding $\alpha_i$ in~\eqref{eq:contribs} were set to zero.
Any $\alpha_i$ that needed to be treated in this way was very close to zero.
We do not depict these slightly modified linear combinations as
they are still very close to the SEBA-features in Figure~\ref{fig:seba}.

%

\section{Optimization}
\label{sec:opt}

Having developed an efficient means of computing singular vectors of $\mathcal{P}_{0,\tau}$ as eigenfunctions of the augmented generator $\bm{\hat{G}}$, we turn our attention to manipulating these eigenfunctions. These manipulations will be used to control the mixing properties of aperiodic flows.
Theorem~\ref{thm:cohratio} provides a construction of a family of coherent sets $\{A_t^\pm\}_{t\in[0,\tau]}$ from eigenfunctions of $\bm{\hat{G}}$, with a coherence guarantee controlled by the corresponding eigenvalues.
Our goal now is to either enhance or diminish the coherence of a family of sets related to an eigenvalue $\mu_k$ by small time-dependent perturbations $u(t,x)$ of the velocity field $v(t,x)$.
This will be achieved by optimally manipulating $v(t,x)$ to increase or decrease the real part of the second eigenvalue $\mu_2$ of $\bm{\hat{G}}$.

For $m\ge 2$ let $H^m((0,\tau)\times X, \mathbb{R}^d)$ denote the Sobolev space of vector fields on $(0,\tau)\times X$ whose weak derivatives of order up to $m$ are square integrable\footnote{In order to use results from \cite{KLP2018} we will later make an additional assumption on $m$.}.
As we have previously assumed that $\mathrm{div}_x v = 0$ to simplify our presentation, for consistency we consider the subspace $\mathcal{D}_0$ of $H^m((0,\tau)\times X, \mathbb{R}^d)$ consisting of spatially divergence-free vector fields that satisfy homogeneous Neumann boundary conditions: $\frac{\partial u}{\partial n} = 0$ on $\partial X$.
We consider small perturbations $u$ lying in a bounded, closed and strictly convex subset $\mathcal{C}\subset \mathcal{D}_0$.

We adopt the approach of \cite{Froyland2016}, who select a perturbation $u$ so as to maximise the derivative of the real part of $\mu_2$ (or a group of leading eigenvalues) with respect to the perturbation.
The perturbation was made to the Ulam-discretized generator of the vector field, and optimized using linear programming;  the perturbed velocity field could then be inferred from the optimized generator.
In \cite{Froyland2016} the velocity field was assumed to be time-periodic, however, the same optimization approach could be applied to aperiodic velocity fields using the time-reflected velocity field introduced in this work;  one would need only to additionally impose the relevant time-reflection constraints in the optimization problem for the Ulam-discretized generator.
In the present work, we consider aperiodic velocity fields and in contrast to \cite{Froyland2016} we perturb the velocity field directly and solve the resulting optimization problem by Lagrange multipliers.
This potentially allows for greater flexibility in the discretization scheme.

To theoretically justify our approach in the infinite-dimensional setting we need some results from \cite{KLP2018} regarding the regularity of the spectrum of $\mathcal{P}_{0,\tau}$ with respect to perturbations of the velocity field.
Transferring these results to the regularity of the spectrum of $\bm{\hat{G}}$ with respect to velocity field perturbations, we derive a first variation of $\mu_k$ with respect to $u$ and detail the steps below in section~\ref{sec:obj}.
In section~\ref{sec:constraints} we specify the constraints and we then proceed with discussing necessary and sufficient (section~\ref{sec:conditions}) conditions for the optimization. Finally we summarize the result of the construction of our optimization.


\subsection{Objective functional and its smoothness}
\label{sec:obj}
We choose an eigenvalue $\mu_k$ of $\bm{\hat{G}}$ corresponding to a feature we want to enhance or diminish.
We want to alter the real part of $\mu_k$ with a perturbation $u$ of the original velocity field $v$ as much as possible within our constraints.
This is because by Theorem \ref{thm:cohratio} the real part of $\mu_k$ is a measure for the coherence of the family of features highlighted by the corresponding eigenvector.
Our chosen objective functional should be a good measure of the change of $\mu_k$ with respect to~$u$.
The response of an eigenvalue or a singular value with respect to a perturbation in this infinite-dimensional setting is in general complicated.
Therefore we approximate it locally via linearization;  that is by computing a first variation or first-order Taylor expansion.
In what follows, we assume that $\mu_k$ is real;  the obvious modifications can be made if $\mu_k$ is complex by considering the real parts.

Our domain $X$, the drift $v$, the perturbation $u$, for $m \geq 1+ \frac{d}{2} $ such that the space $H^m((0,\tau)\times X;\mathbb{R}^d)$ is conitnuously imbedded in $C^{(1,1)}([0,\tau]\times \overline{X};\mathbb{R}^d)$ \cite[Thm. 4.12]{AF2003}, and the noise $\varepsilon I_{d\times d}$ are smooth enough to apply the results of \cite{KLP2018} to $\mathcal{P}_{0,\tau}$.
Assuming that the singular value $\sigma_k$ is simple and isolated, \cite[Theorem 5.1]{KLP2018} and the paragraph following it guarantee Fr\'echet differentiability of $\sigma_k$ and the corresponding singular function with respect to~$u$.
Using  \eqref{eq:eig-sing} and the expression $\hat{\mathcal{P}}_{0,2\tau} = \mathcal{P}^{\ast}_{0,\tau} \mathcal{P}_{0,\tau}$ we can relate the singular values and functions of $\mathcal{P}_{0,\tau}$ and the eigenvalues and eigenfunctions of $\smash{ \hat{\mathcal{P}}_{0,2\tau} }$.
In particular, the eigenvalues of $\smash{ \hat{\mathcal{P}}_{0,2\tau} }$ are Fr\'echet differentiable with respect to $\hat{u}$.
The spectral mapping property\footnote{See~\cite[Chapters 2,3 and 6]{Chicone1999} for analogous results in the context of evolution semigroups, \cite[Lemma~22]{FK17} for periodically forced systems, and~\cite{Paz83,EnNa00} for spectral mapping results for one-parameter semigroups.}
of Proposition~\ref{augmen} asserts for corresponding eigenvalues $\mu(\bm{\hat{G}})$ of $\bm{\hat{G}}$, eigenvalues $\lambda(\hat{\mathcal{P}}_{0,2\tau})$ of $\hat{\mathcal{P}}_{0,2\tau}$ and singular values $\sigma (\mathcal{P}_{0,\tau})$ of $\mathcal{P}_{0,\tau}$ that
\begin{equation*}
	\exp (2\tau \mu (\bm{\hat{G}})) = \lambda \big(\hat{\mathcal{P}}_{0,2\tau}\big) = (\sigma (\mathcal{P}_{0,\tau}))^2,
\end{equation*}
which extends the differentiability to the spectrum of~$\bm{\hat{G}}$ and, in particular, $\mu_k$, hence $	\mu_k = \tfrac{1}{2\tau} \ln \lambda_k = \tfrac{1}{\tau} \ln \sigma_k$.

Having establishing the Fr\'echet differentiability of $\mu_k$ with respect to $u$, we now calculate the first variation of $\mu_k$ with respect to $\hat{u}$;  in other words, we compute the G{\^a}teaux derivative of $\mu_k$ at $\hat{v}$ in the direction induced by $u$, which exists and coincides with the Fr\'echet derivative of $\mu_k$ at $\hat{v}$ applied to the direction induced by ${u}$.
Consider $u$ as above and some small $\delta >0$.
We insert the reflected perturbed velocity field $\hat{v}+\delta \hat{u}$ into \eqref{eq:aug_gen}:
\begin{equation}\label{eq:pert-gen}
(\bm{\hat{G}} + \delta \hat{\bm{E}}) ( \hat{\bm{f}}) = - \text{div}_{(\theta ,x)} \left( \begin{pmatrix} 1 \\ \hat{v} \end{pmatrix} \bm{f} \right) + \dfrac{1}{2} \Delta_{(\theta,x)} \bm{\varepsilon}^2\hat{\bm{f}} \underbrace{- \delta \text{div}_{(\theta,x)} \left( \begin{pmatrix} 0 \\ \hat{u} \end{pmatrix} \hat{\bm{f}} \right)}_{= \delta \bm{\hat{E}}\bm{\hat{f}}},
\end{equation}
where
$\bm{\hat{v}}=(1,\hat{v})$, $\bm{\hat{u}} = ( 0, \hat{u})$, $\hat{E}(t) (\hat{f}) = - \hat{f} \text{div}_x (\hat{u}) - \hat{u} \nabla_x \hat{f}=- \hat{u} \nabla_x \hat{f}$, using $\text{div}_x (\hat{u})=0$ and the perturbation generator $\hat{\bm{E}} = -\text{div}_{(\theta,x)} \bm{\hat{u}} - \bm{\hat{u}} \nabla_{(\theta,x)}$.
Let $\hat{\bm{g}}^k(\delta)$ and $\hat{\bm{f}}^k(\delta)$ be the left and right eigenfunctions, respectively, of $\bm{\hat{G}}+\delta \bm{\hat{E}}$ for the eigenvalue $\mu_k(\delta)$; that is,
\begin{align*}
(\bm{\hat{G}}+\delta \bm{\hat{E}})\phantom{^{\ast}} \hat{\bm{f}}^k(\delta) &= \mu_k(\delta) \hat{\bm{f}}^k(\delta),\\
(\bm{\hat{G}}+\delta \bm{\hat{E}})^{\ast} \hat{\bm{g}}^k(\delta) &= \mu_k(\delta) \hat{\bm{g}}^k(\delta),
\end{align*}
normalising so that $\langle \hat{\bm{f}}^k(\delta) , \hat{\bm{f}}^k(\delta) \rangle_{L^2} = \langle \hat{\bm{g}}^k(\delta), \hat{\bm{f}}^k(\delta) \rangle_{L^2} = 1$.
For $\delta=0$, we use the shorthand $\hat{\bm{f}}^k=\hat{\bm{f}}^k(0)$ and $\hat{\bm{g}}^k=\hat{\bm{g}}^k(0)$.
To estimate the effect of the perturbation $u$ on $\mu_k$ we linearise $\mu_k(\delta)$ at $\delta = 0$.
We have
$$\frac{d}{d\delta} \mu_k (\delta)|_{\delta = 0} = \frac{d}{d\delta} \langle \hat{\bm{g}}^k(\delta) , (\bm{\hat{G}}+\delta\hat{\bm{E}}) \hat{\bm{f}}^k(\delta) \rangle_{L^2}|_{\delta = 0} = \langle \hat{\bm{g}}^k , \hat{\bm{E}} \hat{\bm{f}}^k \rangle_{L^2},$$
using the eigenproperties of $\hat{\bm{f}}^k$ and $\hat{\bm{g}}^k$, and the normalisations above; see also~\cite[Section 4.3]{Froyland2016}.
Now,
\begin{equation}
\label{c1}
\langle \hat{\bm{g}}^k , \hat{\bm{E}} \hat{\bm{f}}^k \rangle_{L^2}
= \int_{[0,2\tau]\times X} \hat{\bm{g}}^k(\bm{x})(-\text{div}_{(\theta,x)} \left( \begin{pmatrix}0 \\ \hat{u} \end{pmatrix} \hat{\bm{f}}^k \right) (\bm{x}) ) d\bm{x}=:c({u}),
\end{equation}
where $c:\mathcal{C}\to\mathbb{R}$ is a linear function of $u$.
If $\mu_k$ is complex, then one considers the real part of the functional $c$.
\begin{lemma}\label{lem:obj_prop}
The objective functional $c \, : \, H^m((0,\tau)\times X) \rightarrow \R$, with $m\geq 1$, is continuous, Fr\'echet differentiable and the Fr\'echet derivative is Lipschitz continous.
\end{lemma}
\begin{proof}
	Using (\ref{c1}), the following estimate 
shows that $c$ is continuous.
	\begin{eqnarray*}
		|c(u)| & \leq &\int_{[0,2\tau]\times X} | \hat{\bm{g}}^k \hat{\bm{f}}^k \text{div}_{\bm{x}} \begin{pmatrix} 0 \\ \hat{u} \end{pmatrix} | + | \hat{\bm{g}}^k \langle \begin{pmatrix} 0 \\ \hat{u} \end{pmatrix}, \nabla_{\bm{x}} \hat{\bm{f}}^k \rangle_{\R^{d+1}} | \, d\bm{x}\\
		& \leq & K_1 \| \hat{\bm{g}}^k \|_{\infty} \left( \| \hat{\bm{f}}^k \|_{\infty} \int_{[0,2\tau]\times X} | \text{div}_{\bm{x}} \begin{pmatrix} 0 \\ \hat{u} \end{pmatrix} | \, d\bm{x} +\int_{[0,2\tau] \times X} \| \nabla_{\bm{x}} \hat{\bm{f}}^k \|_{\infty} \| \hat{u} \|_2 \,d\bm{x} \right)\\
		& \leq & K_2 \| \hat{\bm{g}}^k \|_{\infty} \| \nabla_{\bm{x}} \hat{\bm{f}}^k \|_{\infty} \| u \|_{H^1}
	\end{eqnarray*}
	Here $\| \cdot \|_{\infty}$ denotes the canonical $L^\infty ((0,2\tau)\times X)$ norm. The Fr\'echet differentiability of $c$ is straightforward because $c$ is linear.
\end{proof}

We will prove further relevant properties of $c$ in section~\ref{sec:conditions}.

\subsection{Constraints}\label{sec:constraints}
As mentioned above we consider perturbations $u \in \mathcal{C}$, a bounded, closed and strictly convex subset of $\mathcal{D}_0 \subset H^m((0,\tau) \times X;\R^d)$.
For our objective functional $c$ to be a valid estimate of the change in $\mu_k$ due to the perturbation $u$, we restrict $u$ to be small using $R>0$.
We consider a ball or ellipsoid in the form of the following energy constraint. For multi-indices $\alpha$ and a weight vector $\omega = (\omega_{\alpha})_{|\alpha| \leq m}$ with $0<\omega_{\alpha} \in \mathbb{R}$ for all $|\alpha| \leq m$ we require
\begin{equation}\label{eq:B-omega}
\begin{aligned}
	0 \stackrel{!}{\ge} h(u)\ &{:=}\ B_{\omega}(u,u)-R^2 = \sum_{|\alpha| \leq m} \omega_{\alpha} \| D^{\alpha} u \|_{L^2}^2 -R^2\\
	&= \sum_{|\alpha| \leq m} \omega_{\alpha} \langle D^{\alpha} u , D^{\alpha} u \rangle_{L^2} - R^2,
	\end{aligned}
\end{equation}
where $D^{\alpha} = (\partial_1^{\alpha_1} \cdots \partial_d^{\alpha_d})$.
The functional $h$ is continuously Fr\'echet differentiable in $u$ from $H^m$ to $\R$ since $B_{\omega}$ is a bounded positive definite bilinear form.
We consider $\mathcal{C} = U \cap \{h\le 0\}$ where $U$ is a proper subspace of $H^m((0,\tau) \times X;\R^d)$ and might only have a relative interior.\footnote{For a rigorous definition of ``relative interior'' and related concepts see \cite[II.6]{Rockafellar1970}. For a (not entirely correct but) short intuition: The relative interior of a set $A$ can be thought of as the interior with respect to the trace topology of the smallest subspace that contains $A$.} $U$ describes the set of admissible perturbations to the original vector field~$v$.
By construction in \eqref{eq:B-omega} there are constants $\gamma>0$ ($B_{\omega}$ is bounded) and $\beta>0$ ($B_{\omega}$ is positive definite) such that
\begin{equation*}
	\beta \| u \|^2_{H^m} \leq B_{\omega}(u,u) \leq \gamma \| u \|^2_{H^m} \; .
\end{equation*}
Now, Theorem \ref{thm:bil_form} ensures the existence of a linear, bounded, injective and self-adjoint operator $J_B(\omega)$ with
\begin{equation*}
	\langle J_B(\omega) u,v \rangle_{H^m} = B_{\omega}(u,v) \quad \text{for all } u,v \in H^m \; .
\end{equation*}
We will use $J_B(\omega)$ to derive an explicit formula for the optimal solution.

\subsection{Optimality conditions}\label{sec:conditions}
We have an optimization problem with a continuous linear objective $c$ and a closed, bounded, strictly convex feasible set $\mathcal{C}$, defined by a single constraint $\{h\le 0\}$ on a subspace $U$.
Lagrange multipliers provide a convenient and explicit solution to this problem, e.g.\ \cite{Luenberger1985}, however, first we establish a general existence and uniqueness result.

\subsubsection{Unique optimum}
\begin{lemma}\label{lem:opt_sol}
Let $\mathcal{C}$ be a closed, bounded, and strictly convex subset of $H^m$ containing the zero element in its (relative) interior and $c \, : \, \mathcal{C} \rightarrow \mathbb{R}$ be a bounded linear functional that does not uniformly vanish on~$\mathcal{C}$.
Then the optimization problem $\min_{u\in \mathcal{C}} c(u)$ has a unique solution $u^{\ast} \in \mathcal{C}$.
\end{lemma}
\begin{proof}
	Continuity  of $c$ and boundedness of $\mathcal{C}$ imply
		$\inf_{u\in \mathcal{C}} c(u) = \alpha > - \infty$.
Let $u_k \in \mathcal{C}$ be such that $\lim_{k\rightarrow \infty} c(u_k) = \alpha$.
	This sequence is bounded, and so there is a weakly convergent subsequence $u_{n_k} \rightharpoonup u^{\ast}$. The set $\mathcal{C}$ is closed and convex and therefore also weakly closed, which implies $u^\ast \in \mathcal{C}$. By the definition of weak convergence $c(u^{\ast}) = \lim_{k\rightarrow \infty} c(u_{n_k}) = \alpha$ follows. Therefore $u^{\ast}$ is a solution for our optimization problem.

In order to demonstrate uniqueness assume we have two solutions $u_1\neq u_2$ with $c(u_1) = c(u_2) = \alpha$.
Strict convexity of $\mathcal{C}$ implies that $u_3:=u_1/2+u_2/2\in \text{int}(\mathcal{C})$, where the relative interior is meant.
Linearity of $c$ implies~$c(u_3)=\alpha$.
Let $r>0$ be such that an open ball of radius $r$ centred at $u_3$ is contained in~$\text{int}(\mathcal{C})$.

Because $c$ does not vanish on $\mathcal{C}$ and the zero vector is in the relative interior of $\mathcal{C}$, there exists a $v\in\mathcal{C}$ such that $c(v)<0$.
By linearity of $c$ we have $c(u_3+(r/2)v)<\alpha$, contradicting optimality of $u_3$ and establishing uniqueness of the optimum.
\end{proof}

\subsubsection{Necessary conditions}
The following property will be used for the necessary conditions in the theory of Lagrange multipliers in Lemma \ref{lem:nec}.
\begin{lemma}
	The unique optimal solution $u^{\ast}$ is a regular point for $h:U\to\mathbb{R}$.
\end{lemma}
In order to distinguish between the point at which the derivative is taken and the input of the resulting mapping, we will use square brackets for the reference point of the derivative and round brackets for the input of the resulting mapping.\footnote{So $h'[u](v)$ denotes the derivative of $h$ at $u$, which again is a linear mapping, applied to $v$.}
\begin{proof}
Following the definition on \cite[p.~240]{Luenberger1985}, the point $u^\ast$ is a regular point for $U\to\mathbb{R}$ if the derivative of $h$ at $u^{\ast}$, denoted by $h'[u^{\ast}] : U \rightarrow \mathbb{R}$ is surjective.	
Since $h(u)=B_\omega (u,u)-R^2$, and the functional $h'[u]$ acts as $v \mapsto  h'[u](v) = \langle J_B(\omega) v , u \rangle_{H^m}+ \langle J_B(\omega)u, v  \rangle_{H^m} $, $h'[u]$ is obviously surjective onto $\R$ for all~$u \neq 0$.
\end{proof}

The uniqueness argument in the proof of Lemma~\ref{lem:opt_sol} shows that we may replace our constraint $h(u)\le 0$ with~$h(u)=0$.
We now use \cite[Theorem 1, p.~243]{Luenberger1985}, stated below.
\begin{lemma}\label{lem:nec}
	If the  continuously Fr\'echet differentiable functional $c$ has a local extremum under the constraint $h(u)=0$ at the regular point $u^{\ast}$, then there exists an element $z \in \mathbb{R}$ such that the Lagrangian functional $L(u) = c(u) + z h(u)$	is stationary at $u^{\ast}$;  that is, $c'[u^{\ast}] + z h'[u^{\ast}] = 0$.
\end{lemma}
We thus obtain the two necessary conditions:
\begin{eqnarray}
\label{nec1} c'[u^{\ast}] + z h'[u^{\ast}] & = & 0,\\
\label{nec2} h(u^{\ast}) & = & 0.
\end{eqnarray}

\subsubsection{Sufficient conditions}
We now prove that the necessary conditions (\ref{nec1}) and (\ref{nec2}) are in fact also sufficient.
Because our objective is linear and our constraint is of inner product form, we take a direct approach to developing sufficient conditions, avoiding more complicated general theory.
\begin{proposition}\label{lem:suff}
Let $\mathcal{C}=U\cap \{h\le 0\}$.
There are exactly two elements of $\mathcal{C}$ that satisfy \eqref{nec1} and~\eqref{nec2}.
One is the unique minimizer (with $z>0$) and the other is the unique maximizer (with $z<0$).
\end{proposition}
\begin{proof}
Lemma \ref{lem:opt_sol} guarantees the existence of at least two extrema (one minimum and one maximum), and therefore at least two distinct elements $u,w\in \mathcal{C}$ satisfying (\ref{nec1}) and (\ref{nec2}).
We show that these are the only such elements. There exist $z_u, z_w\in\mathbb{R}$ such that
$$c'[u] + z_u h'[u] = 0\mbox{ and } c'[w] + z_w h'[w] = 0.$$
Subtracting these two equations and using $c'[u](\cdot)=c'
[w](\cdot)=c(\cdot)$, we obtain the functional equation
$h'[w]=(z_u/z_w)h'[u]$.
Thus, the linear functional $h'[w]$ is a scalar multiple of the linear functional $h'[u]$.
Since $h'[u](\cdot)=2 \langle \cdot , J_B(\omega) u \rangle_{H^m}$, by the Riesz representation theorem, we have $w=(z_u/z_w)u$.
However, the necessary condition $h(u)=h(w)=0$ implies that $h'[u](u)= 2 \langle u , u \rangle_{m,\omega}= 2 \|u\|_{m,\omega}^2=2 R^2$ and similarly that $\|w\|_{m,\omega}^2=R^2$.
Thus, either $z_u=z_w$ or $z_u=-z_w$.
If $z_u=z_w$, then we have $u=w$, while if $z_u=-z_w$, then $u=-w$.
Thus, the only possibility for distinct $u$ and $w$ is that $u=-w$, and therefore that there are at most two functions satisfying the necessary conditions.
Finally, without loss, assume that $u$ is a minimum.
Since $c'[u](u) + z_u h'[u](u)=c(u)+z_u R^2=0$ and $c(u)<0$ if $u$ is a minimum, we must have $z_u>0$.
Therefore $z_w<0$, implying that $c(w)>0$ and that $w$ is a maximum.
\end{proof}


Using the injective operator $J_B(\omega)$ we can solve the necessary and sufficient conditions \eqref{nec1} and \eqref{nec2} for the optimal solution $u^{\ast}$, leading to~\eqref{eq:expl-sol} below. First we can transform \eqref{nec1} into an equation in $H^m$ using the Riesz representation theorem. We now know that $u^{\ast}$ exists and fulfills the following equation.
\begin{equation}\label{eq:nec1b}
	c_R + 2 z J_B(\omega) u^{\ast} = 0 \text{ in } H^m.
\end{equation}
Here $c_R \in H^m$ is the Riesz representation of the functional $c$ on $H^m$. Thus $c_R$ is in the range of $J_B(\omega)$ and we can apply $J_B(\omega)^{-1}$ to $c_R$. Now we can solve \eqref{eq:nec1b} for $J_B(\omega) u^{\ast}$ and for $u^{\ast}$ because Proposition \ref{lem:suff} guarantees $z\neq 0$, giving
\begin{equation*}
	J_B(\omega) u^{\ast} = - \frac{1}{2z} c_R, \qquad u^{\ast} = -\frac{1}{2z} J_B (\omega)^{-1}c_R.
\end{equation*}
Inserting these expressions into \eqref{nec2} leads to
\begin{equation}\label{eq:nec2b}
	R^2 = \dfrac{\langle J_B(\omega)^{-1} c_R , c_R \rangle_{H^m}}{4z^2}.
\end{equation}
Solving \eqref{eq:nec2b} for $z>0$ (minimizer) and using this $z$ leads to the following explicit expressions
\begin{equation}\label{eq:expl-sol}
	0 < z = \dfrac{ \langle J_B(\omega)^{-1} c_R,c_R \rangle_{H^m}^{\frac{1}{2}}}{2R}, \qquad u^{\ast} = - \dfrac{J_B(\omega)^{-1}c_R}{2 z} \; .
\end{equation}

\section{Optimization of $\mu_k$ numerically}
\label{sec:opt_num}

In this section we apply the results derived in section~\ref{sec:opt} to some examples.

\subsection{Discrete optimization problem}
\label{ssec:disc_opt}

Before discretising the objective functional we want to construct $\mathcal{C}_N$, a finite dimensional version of the constraint set $\mathcal{C}$. Therefore we choose finitely many basis elements $\{\varphi_\ell\}_{\ell=1,\ldots , N}$ spanning the admissible subspace~$U_N := \text{span} \{\varphi_\ell\}_{\ell=1,\ldots , N}$ for our perturbations.
We then intersect $U_N$ with $\mathcal{C}$ and represent elements by their coefficient vectors in~$\mathbb{R}^N$ with respect to the chosen basis. Hence we define $\mathbb{R}^N \supset \mathcal{C}_N \simeq \mathcal{C} \cap U_N \subset H^m$.

Coefficient vectors will be denoted with a bar~$\bar{\phantom{u}}$ in the following. We will omit the $\hat{\phantom{X}}$ in the following calculations, but note that it can be done analogously using augmented reflected objects.
The energy neighborhood constraint \eqref{eq:B-omega} can be expressed as a quadratic constraint in the coefficient vector $\bar{u}  \in \R^N$,
\begin{equation*}
	(B_{\omega})_{ij} := \sum_{|\alpha|\leq m} \omega_{\alpha} \langle D^{\alpha} \varphi_i , D^{\alpha} \varphi_j \rangle_{L^2}, \qquad \bar{u}^T B_{\omega} \bar{u} \leq R^2 \Leftrightarrow \bar{u} \in \mathcal{C}_N \simeq \mathcal{C} \cap U_N \, .
\end{equation*}
This constraint describes a strictly convex set (ball or ellipsoid) in~$\R^N$. Regarding the objective functional
\begin{equation*}
 c(u) = \langle \bm{g}^k, \bm{E} \bm{f}^k \rangle_{L^2},
\end{equation*}
we have to account for the two possibly different bases for discretization: (i) the discretization of $u$, and (ii) the discretization of $\bm{\hat{G}}$, which can also involve a test and an ansatz basis.
Let us denote the basis functions for the discretization of $\bm{f}^k$ by $\{\chi_j\}_{j=1,\ldots ,m}$ and $\bm{g}^k$ by $\{\xi_i\}_{i=1,\ldots , n}$ for now.
Then for $\bm{f}^k=\sum_{j=1}^m \bar{f}_j^{k} \chi_j$ and $\bm{g}^{k}=\sum_{i=1}^n \bar{g}^{k}_i \xi_i$ we have
\begin{align*}
	c(u) &= \langle \bm{g}^{k}, \bm{E} \bm{f}^{k} \rangle_{L^2} = \sum_{i=1}^n \Big\langle \bar{g}^{k}_i \xi_i, \bm{E}\big(\sum_{j=1}^m \bar{f}_j^{k} \chi_j\big) \Big\rangle_{L^2}\\
	&\phantom{=} \quad = \sum_{\ell=1}^N \sum_{i=1}^n \sum_{j=1}^m \bar{g}^{k}_i \bar{f}^{k}_j \langle \chi_j , - \text{div}_{(\theta ,x)} \left( \begin{pmatrix}0 \\ \varphi_l \end{pmatrix} \xi_i \right) + \begin{pmatrix} 0 \\\varphi_\ell \end{pmatrix}^T \nabla_{(\theta , x)} \xi_i \rangle_{L^2} \bar{u}_\ell.
\end{align*}
Using this last equation we can calculate a finite-dimensional representation of $\bm{E}$ acting from $\text{span}\{\xi_j\}_j$ to $\text{span}\{\chi_i\}_i$.

Due to linearity we may decompose in $\ell$ and separately compute $\bm{E}_\ell$ in a similar way to the numerical approximation of  $\bm{G}$, outlined in Section \ref{ssec:num_disc}.
We use Ulam's method to discretize the generators, taking $m=n$ and $(\xi_j)_j = (\chi_i)_i$ to be indicator functions of space-time boxes.
The cost vector can be constructed by
\begin{equation*}
	\bar{c}_\ell:=c(\varphi_\ell) = \bar{g}^k \bm{\bar{E}}_\ell \bar{f}^k = \sum_{j=1}^m \sum_{i=1}^n \bar{g}^k_j \bar{f}^k_i \langle \chi_j , - \text{div}_{(\theta,x)} \left(\begin{pmatrix} 0 \\ \varphi_\ell \end{pmatrix} \xi_i \right) \rangle_{L^2}, \quad \ell=1,\ldots,N.
\end{equation*}
The discretized optimization problem then has a linear objective and a single quadratic constraint
\begin{align*}
	\max 		\, & \quad \bar{c}^T \bar{u} \\
	\text{s.t.} 	\, & \quad  \bar{u}^T B_{\omega} \bar{u} - R^2 \leq 0.
\end{align*}
Since the energy constraint is induced by a scalar product, the matrix $B_{\omega}$ is invertible by typical arguments for Galerkin discretization (i.e., $B_{\omega}$ is symmetric positive definite). Thus the optimal solution can be obtained with Lagrange multipliers, identical to the analysis leading to (\ref{eq:expl-sol}), and is given by
\begin{equation} \label{eq:expl-sol-discr}
	0 < z = \dfrac{\left( \langle B_{\omega}^{-1}\bar{c} , \bar{c} \rangle \right)^{\frac{1}{2}}}{2R}>0, \qquad \bar{u} = -\frac{1}{2z} B_{\omega}^{-1} \bar{c}.
\end{equation}

In practice we choose the ``admissible energy'' $R$ sufficiently small to  ensure the approximate validity of our linearised objective functional. We can then iterate the optimization process as a gradient ascent/decent method to invest more cumulative energy in the perturbation. Each step consists of constructing $\bar{c}$ and solving the equations \eqref{eq:expl-sol-discr}. The construction of $\bar{c}$ requires the calculation of $\bm{G}$, $\bm{g}^k$, $\bm{f}^k$, and $\bm{E}_\ell$ (or the respective augmented reflected objects); the latter are fixed through all optimization steps and do not need to be updated.
We apply this procedure in the following examples.

\begin{remark}
All of our finitely many perturbation ansatz functions $(\varphi_\ell)_{\ell=1,\ldots , N}$, introduced in the next section \ref{ssec:perturbations}, are $C^{\infty}$ thus $\| \cdot \|_{H^m}$ and $\| \cdot \|_{L^2}$ are equivalent on $\text{span}\{\varphi_l\}_l$.
Therefore for numerical convenience, in the numerical examples we use $m=0$ to calculate $B_{\omega}$, although the functional $c$ is only strictly well defined for $m\geq 1$.
\end{remark}

\subsection{Perturbing fields}
\label{ssec:perturbations}

We construct a suitable basis for velocity field perturbation as follows.
Our spatial domain will be a rectangle.
In order to impose zero velocity normal to the boundary and divergence-freeness, we construct the spatial components of the basis vectors $\{\varphi_\ell\}_{\ell=1,\ldots,N}$ with a possible constant movement in $x$ or $y$ direction from smooth stream functions $\Psi_{kl}$ and then multiply these components with time-dependent scalar (amplitude) functions $\phi_r$. For a rectangular domain $[a_x,b_x] \times [a_y,b_y]$ we take the streamfunctions
\begin{equation}\label{eq:pert-lib}
	\Psi_{kl} (t,x,y) = \sin \left(\frac{k\pi (x-a_x-c_xt)}{b_x-a_x}\right) \sin\left(\frac{l \pi (y-a_y-c_yt)}{b_y-a_y}\right),
\end{equation}
$k=1,\ldots,K$, $l = 1,\ldots,L$, which are slightly modified Fourier modes that induce a velocity field $\psi_{kl} := (-\frac{\partial \Psi_{kl}}{\partial y} , \frac{\partial \Psi_{kl}}{\partial x})$ with $k$ horizontal gyres and $l$ vertical gyres that satisfy the homogeneous Neumann boundary conditions in space and are divergence free in space. These fields may travel in $x$ direction with speed $c_x$ or in $y$ direction with speed~$c_y$.

We use $L^2 ((0,\tau)\times X)$-normalized versions of the functions
\[
\varphi_{kl,r}(t,x,y) := \phi_r(t) \psi_{kl}(t,x,y),
\]
where $\phi_{r}$ is a scalar (temporal) modulation of the amplitude of the spatial Fourier modes:
\begin{equation*}
	\phi_{-1}(t) := \frac{t}{\tau}, \qquad \phi_{r}(t) := \sin^{r}\left(\frac{t}{\tau}2\pi \right), \quad r = 0,2.
\end{equation*}
We omit using $r=1$ for the $\sin$-modulation, since it would have both positive and negative values, meaning a sign change of the perturbing velocity field during the evolution. Such perturbing fields proved to be less efficient in early numerical experiments. Thus, the increasing time-linear modulation is  assigned $r=-1$ to avoid confusion. In summary, time-modulation of the perturbing fields is described by~$\phi_r(t)$,~$r\in\{-1,0,2\}$, and we have $3KL=N$ basis functions $\{\varphi_\ell\}_{\ell=1,\ldots,N}$ in total.

\subsection{Increasing coherence: forced double gyre flow}
\label{ssec:incr_coh}

Extending the example from section \ref{ssec:perDG}, our goal in this section is to increase coherence of the left-right separation captured by the $2$nd eigenvalue of~$\bm{\hat{G}}$ (equivalently, the left $2$nd singular vector of the transfer operator~$\mathcal{P}_{0,\tau}$), Figure~\ref{fig:double-eig}~(a).

The original velocity field has a total energy (space-time $L^2$ norm) of $\approx 1.6$.
We will use an optimization budget of $R=0.05$
 per step to increase the $2$nd eigenvalue $\mu_2$ of $\bm{\hat{G}}$ which encodes the left-right separation of the domain.
We iterate our optimization procedure $8$ times to invest a total energy of $0.4$, which is $25\%$ of the energy of the original velocity field.
This seems like a moderate investment, but we note that while in general it is easy to destroy coherence by almost any perturbation, to increase it, the dynamics and the perturbation need to work together---making it harder to increase coherence than to decrease it. Following the formula \eqref{eq:expl-sol-discr} to optimize coherence (i.e, minimize mixing) we use~$-\bar{u}$, where we recall that~$\bar{u}$ denotes the coefficient vector representing the optimal perturbation. Of course the other eigenvalues of the generator also change as the velocity field is perturbed. In each iterative step we check \emph{a posteriori} that the second eigenvalue did indeed increase, which we consider an indicator for the validity of our objective functional. Our perturbation library consists of the functions from \eqref{eq:pert-lib} for $k=1,\ldots ,5$, $l=1,2,3$ with $c_x = 0 = c_y$ and the time modulation $r=-1,0,2$, hence $N=45 = 5\cdot 3 \cdot 3$. After the 8 iteration steps we arrive at the solution vector $\bar{u}_{(8)} \in \mathbb{R}^{N}$ and the effective change of the second eigenvalue singular value:
\begin{equation*}
\mu_2(u_{(8)}) - \mu_2(0) = 0.0233, \quad \dfrac{\mu_2(u_{(8)}) - \mu_2(0)}{| \mu_2(0)|} = 0.2575, \quad \dfrac{\sigma_2(u_{(8)}) - \sigma_2(0)}{| \sigma_2(0)|} = 0.0975.
\end{equation*}
Each step of the iteration increases the eigenvalue roughly by $0.003$.
The result is visualized in Figure~\ref{fig:dg-inc}.

\begin{figure}[hbt]
\centering
\begin{subfigure}{0.49\textwidth}
\includegraphics[width=1\textwidth]{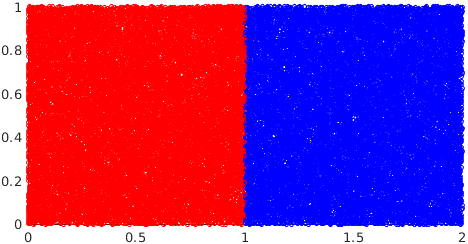}
\subcaption{Imposed coloring at final time.\\
\phantom{VOIDLINE}\\
\phantom{VOIDLINE}}\label{fig:DG-LR-orig-final}
\end{subfigure}
\hfill
\begin{subfigure}{0.49\textwidth}
\includegraphics[width=1\textwidth]{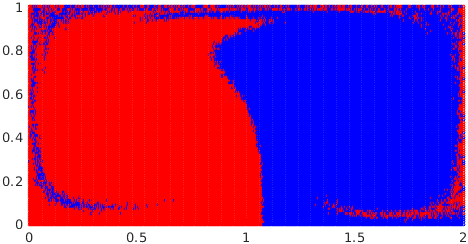}
\subcaption{Coloring at initial time corresponding to (a) using the original unperturbed velocity field.}\label{fig:DG-LR-orig}
\end{subfigure}\\
\begin{subfigure}{0.49\textwidth}
\includegraphics[width=1\textwidth]{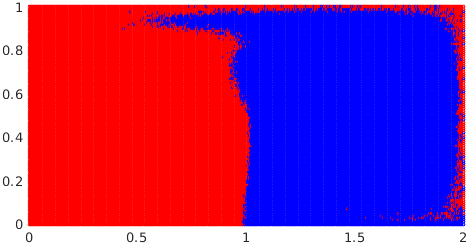}
\subcaption{Coloring at initial time corresponding to (a) using the velocity field optimized for increasing coherence.}
\label{fig:DG-LR-Inc}
\end{subfigure}
\hfill
\begin{subfigure}{0.49\textwidth}
\includegraphics[width=1\textwidth]{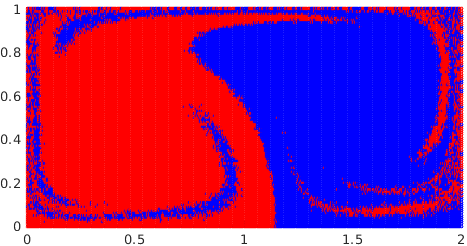}
\subcaption{Coloring at initial time corresponding to (a) using the velocity field optimized for decreasing coherence.}
\label{fig:DG-LR-Dec}
\end{subfigure}
\caption{(a) The left-right (red-blue) coloring imposed on evolved particles ($\varepsilon=0.01$) at final time $t=\tau$;  (b) The particles shown at time~$t=0$ for the original flow; (c) The particles shown at time~$t=0$ for the coherence-increasing optimized velocity field from Section \ref{ssec:incr_coh}; (d) The particles shown at time~$t=0$ for the coherence-decreasing optimized velocity field from Section \ref{ssec:decr_cohDG}.}
\label{fig:dg-inc}
\end{figure}

We seeded $200,000$ particles on the right side of the line $x=1$ and evolved them forward in time with the noisy flow, using~$\varepsilon = 0.1$ (Runge--Kutta--Maruyama with time step size $h = \frac{1}{100}$) for $\tau = 4$ time units using the original and the optimized drift. For the original velocity field roughly $15\%$ end up on the left side of the domain, where as for the optimally increased left-right separation only about $10\%$ of the particles end up on the left side (results not shown). We repeated the same noisy evolution procedure for $\varepsilon=0.01$; then the $9\%$ of particles changing sides originally were reduced to $5\%$ for the coherence-improved velocity field.
This is shown in Figure~\ref{fig:dg-inc} (b)--(d), where the seeded particles are shown at initial time and colored according to whether they end up left or right of the line $x=1$ after this noisy evolution, at time~$t=\tau$.
Note that the time direction is unimportant and we could have colored at the initial time $t=0$ and evolved forward.

The optimization of the left-right coherence in the periodically forced double gyre has also been considered in ~\cite[Section 6.3]{Froyland2016}; see, in particular Figure~18 therein, where the regular regions of the flow were dramatically increased.

\subsection{Decreasing coherence: forced double gyre flow}
\label{ssec:decr_cohDG}

We now wish to diminish the coherence of various coherent features.
Firstly, the left-right separation discussed in Section \ref{ssec:incr_coh}.
Using the same optimization protocols as in Section \ref{ssec:incr_coh}, but switching the sign of the objective, we produce a velocity field that should increase mixing across the left-right separatrix.
This is indeed indicated in Figure \ref{fig:dg-inc} (d).
These results are consistent with the results of \cite{Froyland2016}, where the ``lobes'' of stable and unstable manifold intersections greatly increased~\cite[Figures 11 and 14]{Froyland2016}.

We now turn our attention to the two central vortices, which are encoded in the $5$th eigenmode (and after one iteration step of the optimization in the $6$th eigenmode).
We use the same perturbation basis and energy criterion for the iteration as above ($8$ iterative optimization steps). After a first optimization step of the iteration described at the end of section \ref{ssec:disc_opt} the gyre feature (initially $5$th eigenvalue) is pushed to the $6$th eigenvalue spot, i.e., there is an interchange position with the mentioned features in terms of ranking with respect to their coherence. Thus, we have to keep track of the ranking of the eigenvalues during our iterative optimization procedure, which we do here manually between each step. This could be done in an automated fashion similarly to section~\ref{ssec:BickleyCohset}, by computing correlations between eigenvectors of successive iterates.

We obtain the following change in the eigen- and (corresponding) singular values:
\begin{equation*}
\mu_5(0) -\mu_6(u_{(8)}) = 0.16\qquad \dfrac{\mu_5(0) -\mu_6(u_{(8)})}{| \mu_5(0)|} = 0.36 \qquad \dfrac{\sigma_5(0) - \sigma_6(u_{(8)})}{| \sigma_5(0)|} = 0.48
\end{equation*}
Next we seed particles in the vortices induced by the level sets of the $5$th eigenvector shown in Figure~\ref{fig:double-eig} and evolve them with and without perturbation. The results are visualized in Figure~\ref{fig:DG-dec}.
\begin{figure}[htbp]
\begin{subfigure}{0.32\textwidth}
\includegraphics[width = 1\textwidth]{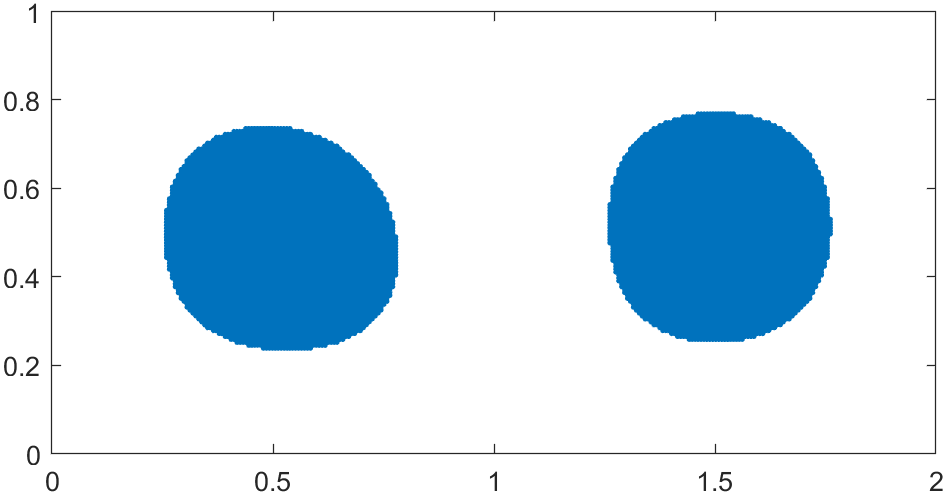}
\subcaption{Initial sets for the evolution (vortices seeded according to $5$th eigenvector)}
\end{subfigure}
\hfill
\begin{subfigure}{0.32\textwidth}
\includegraphics[width = 1\textwidth]{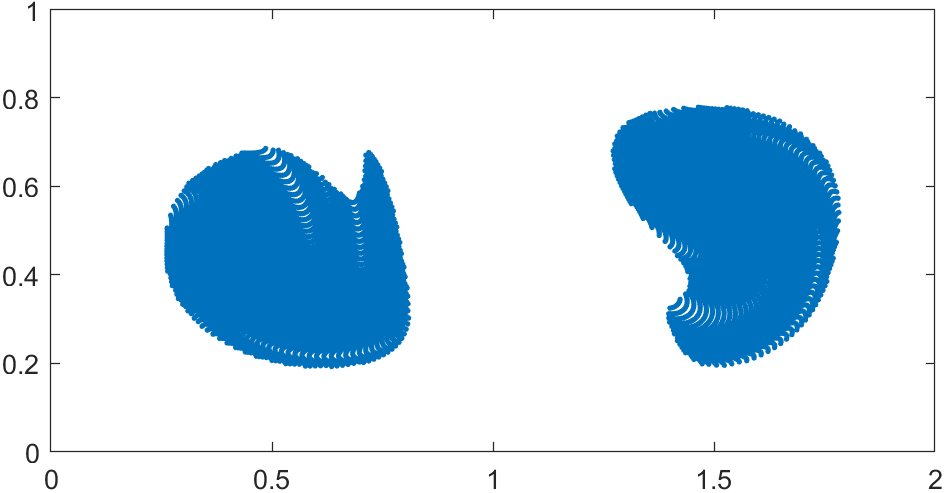}
\subcaption{Particles after original evolution for $\tau = 4$ time units.}
\end{subfigure}
\hfill
\begin{subfigure}{0.32\textwidth}
\includegraphics[width = 1\textwidth]{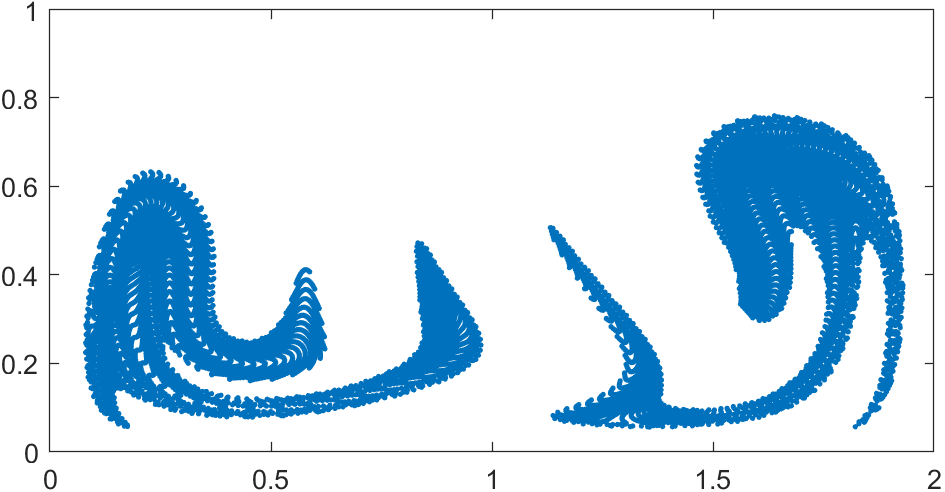}
\subcaption{Particles after optimally perturbed evolution for $\tau = 4$ time units.}
\end{subfigure}
\caption{Test particles and their forward-time evolution.}
\label{fig:DG-dec}
\end{figure}
We note that increasing mixing in the double gyre flow has been considered in \cite[Section 6.2]{Froyland2016}; in particular our Figure~\ref{fig:DG-dec} could be compared with Figures 13, 16, and 17 therein.

\subsection{Targeted manipulation of distinguished coherent features}
\label{ssec:decr_cohBJ}

In certain situations it may be of interest to manipulate the mixing of parts of phase space that do not arise as eigenmodes.
For example, these parts of phase space may be individual (or combinations of) SEBA vectors as described in section~\ref{ssec:SEBA}, or they may be related to the phase space geometry.
We describe a procedure to accomplish this.

We recall Proposition~\ref{augmen} and Theorem~\ref{thm:cohratio}, where we showed
\begin{equation} \label{eq:CohEquivRecap}
\begin{aligned}
\mathcal{P}_{0,\tau}^*\mathcal{P}_{0,\tau} \bm{f}(0,\cdot) &= e^{2\tau\mu}  \bm{f}(0,\cdot)\\
&\Updownarrow\\
\bm{\hat{G}}\bm{f} &= \mu \bm{f}, \text{ where }(\bm{\hat{G}}\bm{f})(\theta,\cdot) = -\partial_{\theta} \bm{f}(\theta,\cdot) + \hat{G}(\theta) \bm{f}(\theta,\cdot),
\end{aligned}
\end{equation}
and that the sign structure of $\bm{f}$ (recall $\bm{f}$ necessarily has zero mean) indicates a family of finite-time coherent sets. Further, the coherence ratio of the family is bounded by an expression involving $\mu$, indicating more coherence the closer $\mu$ is to~$0$.
By normalising $\bm{f}$ so that $\|\bm{f}\|=1$ we see
smaller $\|\bm{\hat{G}} \bm{f}\|$ corresponds to more strongly coherent features encoded in~$\bm{f}$.

Now let $\bm{\varphi} \in \mathcal{D}(\bm{\hat{G}})$ be a general normalised, zero-mean \emph{space-time feature};  that is, $\bm{\varphi}$ is not an eigenfunction.
We might think of $\bm{\varphi}$ being mean-removed SEBA vector or a mollified (such that it is in the domain $\mathcal{D}(\bm{\hat{G}})$ of the generator) version of $\bm{\varphi} = \mathds{1}_{\bm{C}} - |\bm{C}|\mathds{1}_{\bm{X}}$, the mean-centred\footnote{We note that the removal of the mean from $\bm{\varphi}$ makes no difference for the optimization of the objective function in~\eqref{eq:objfun_feature} below, since $\bm{\hat{G}}\mathds{1}_{\bm{X}}=0$, however we keep this for the intuitive connection with ``eigenfeatures'' and Theorem~\ref{thm:cohratio}.} indicator function of a possibly coherent family~$\bm{C}\subset \bm{X}$ of sets in augmented-space representation, where $|\bm{C}|$ denotes the augmented-space Lebesgue measure of~$\bm{C}$. As we would in general like $\bm{\varphi}$ to represent a finite-time coherent set, we should restrict our attention to features satisfying $\bm{\varphi}(t,\cdot) \approx \bm{\varphi}(2\tau-t,\cdot)$.

Analogously to the case of an eigenfunction $\bm{f}$, to quantify the \emph{coherence of a feature} $\bm{\varphi}$ that is not necessarily an eigenfunction, we  employ the heuristic of measuring~$\|\bm{\hat{G}}\bm{\varphi}\|$. The rationale for this is as follows. If a family of sets encoded by the eigenvector $\bm{f}$ is completely coherent (in the absence of diffusion), then the temporal change (``movement'') of the sets at any time $\theta$, namely $\partial_{\theta}\bm{f}(\theta)$, would be identical to how the dynamics transports the mass located in the set, i.e., $\partial_\theta\bm{f}(\theta)=\hat{G}(\theta)\bm{f}(\theta,\cdot)$ by the Fokker--Planck equation~\eqref{FP2}.
Thus, if the coherence of the feature~$\bm{\varphi}$ is strong, one has $\partial_{\theta}\bm{\varphi}(\theta,\cdot)-\hat{G}(\theta)\bm{\varphi}(\theta,\cdot) \approx 0$ for all $\theta$, leading to~$\|\bm{\hat{G}}\bm{\varphi}\|\approx 0$. Section~\ref{sec:flux} gives a geometric view on the very same situation: in~\eqref{outflux0} and~\eqref{absflux1}, if the boundary of a time-dependent set moves with a velocity $b(t,x)$ that is approximately equal to the velocity field $v(t,x)$ driving the dynamics, then the outflow from this family of sets will be small---and this can analogously be quantified by the space-time flux~\eqref{augabsflux0}.

Thus, to destroy a coherent feature encoded in $\bm{\varphi}$ we could maximize $\|\bm{\hat{G}}\bm{\varphi}\|^2$ with respect to the perturbing fields~$u$. Again, as this is a nonlinear problem, we approach it by local optimization, and aim to maximize the objective function given by the local linear change,
\begin{equation} \label{eq:objfun_feature}
\begin{aligned}
		c_{\bm{\varphi}} (u) &= \frac{d}{d \delta}\left( \big\| \bm{\hat{G}}(v+\delta u)\bm{\varphi} \big\|^2\right)\!\Big\vert_{\delta = 0} =  \frac{d}{d \delta} \left( \big\| \big( \bm{\hat{G}}(v)+\delta \bm{\hat{E}}(u)\big)\bm{\varphi} \big\|^2\right)\!\Big\vert_{\delta = 0} \\
		&= 2 \left\langle \bm{\hat{G}}(v)\bm{\varphi},\bm{\hat{E}}(u)\bm{\varphi}\right\rangle
\end{aligned}
\end{equation}
subject to constraints on the perturbation~$u$.
Conversely, if we wish to enhance a coherent feature $\bm{\varphi}$ we should minimize $\|\bm{\hat{G}}\bm{\varphi}\|^2$.
If we would simultaneously like to destroy coherence of a feature $\bm{\varphi}_1$ and enhance the coherence of other features encoded in $\bm{\varphi}_2$, then we would maximize
\begin{equation*}
c_{\bm{\varphi}_1,\bm{\varphi}_2} (u) = \alpha_1 c_{\bm{\varphi}_1} (u) - \alpha_2 c_{\bm{\varphi}_2} (u),
\end{equation*}
with weights~$\alpha_1, \alpha_2>0$.

\subsection{Non-eigenfeature optimization: traveling wave}
We consider a traveling wave example~\cite{Pie91,SaWi06,Froyland2010} given by
\begin{equation*}
	x'(t) = c_{\text{drift}} - A \sin (x-\nu t) \cos (y) \qquad y'(t) = A \cos (x-\nu t) \sin (y)
\end{equation*}
on the domain $[0,8]\times (2\pi S^1\times[0,\pi])$. Here, two rotating gyres move in $x$ direction with speed $\nu=0.25$ and are superimposed with a constant drift~$c_{\text{drift}}=1$. To account for this constant speed and periodicity in $x$-direction we choose $k=2,4, \ldots , 20$, $l=1,\ldots , 5$ and $c_x=\nu$ for our perturbation dictionary; see Section~\ref{ssec:perturbations} and equation \eqref{eq:pert-lib}.
We remark that Balasuriya~\cite{Bal15} has investigated a similar dynamical system and considered single ``one at a time'' (as opposed to general linear combinations of) perturbations drawn from a family similar to ours.
In \cite{Bal15}, flux out of a small ``gate'' connecting a stable and unstable manifold is taken as a measure of mixing.
In contrast, we measure mixing through the $L^2$-norm decay of an initial concentration field, and we find the unique perturbation in a convex subset of a 150-dimensional subspace that maximizes the change in mixing rate.

We use the resolution $80\times (80 \times 40)$ and $\varepsilon = 0.1$. We now take a feature that is not described by an eigenfunction and aim to increase its coherence. The feature we choose is the mean-centered version of the following time-constant and horizontally constant profile
\begin{equation}\label{eq:cos-feature}
	\bm{\varphi} (t,x,y) = 1 - \cos(2y),
\end{equation}
which is shown in the coloring of Figure~\ref{fig:feature} (a). We iteratively update the perturbation by solving problem \eqref{eq:objfun_feature} in each step. We iterate for 35 steps with an energy budget of $R=0.1$ ($\sim 1 \%$ of the original energy of $v$) per step. Figure \ref{fig:feature} (b) shows the final time slice of the original (deterministic) evolution of the particles in Figure~\ref{fig:feature} (a), while Figure~\ref{fig:feature} (c) shows the final time slice of the optimized evolution.

\begin{figure}[hbt]
\begin{subfigure}{0.32\textwidth}
\includegraphics[width = 1\textwidth]{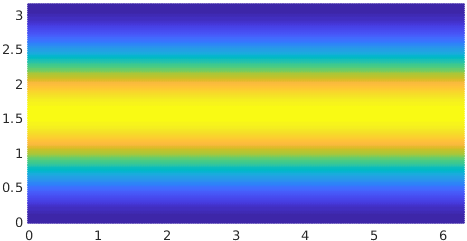}
\subcaption{Initial particles colored according to the chosen feature \eqref{eq:cos-feature}. Yellows is high density and blue is low density.}
\end{subfigure}
\hfill
\begin{subfigure}{0.32\textwidth}
\includegraphics[width = 1\textwidth]{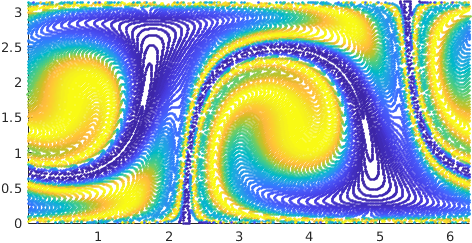}
\subcaption{Particles after original evolution for $\tau = 4$ time units.\\
\phantom{VOIDLINE}\\
\phantom{VOIDLINE}}
\end{subfigure}
\hfill
\begin{subfigure}{0.32\textwidth}
\includegraphics[width = 1\textwidth]{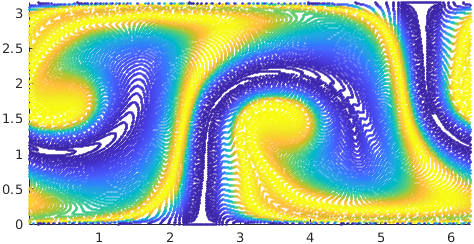}
\subcaption{Particles after optimally perturbed evolution for $\tau = 4$ time units.\\
\phantom{VOIDLINE}\\
\phantom{VOIDLINE}}
\end{subfigure}
\caption{Forward evolution of particles colored by the chosen feature.}
\label{fig:feature}
\end{figure}
We next investigate which perturbing basis functions are favored by the optimization. First, we note that the basis function induced by~$\varphi_{0,2,1}$ is equal to the original velocity field up to the constant horizontal drift $c_{\text{drift}}$. Thus, it is conceivable that this basis function is heavily used in order to partly cancel the original velocity field and slow the flow down.

\begin{figure}[hbt]
\includegraphics[width = 1.0\textwidth]{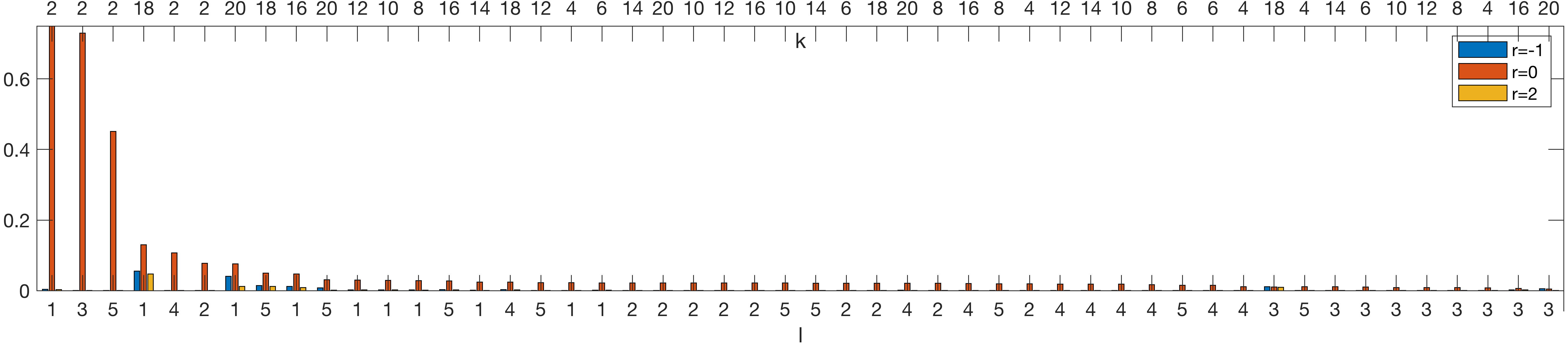}
\caption{Plot of optimal perturbation coefficients $\bar{u}_{\ell}$ ordered in decreasing magnitude and grouped by spatial mode $\psi_{kl}$ (cf.~\eqref{eq:pert-lib}). The corresponding $k$ and $l$ are labels on upper and lower $x$ axes, respectively. The plot is cut off at $y=0.75$ for visualization purposes. The first red bar has height~$7.6$.}
\label{fig:barplot}
\end{figure}
The amplitudes of all of the streamfunctions in the optimized solution are shown in Figure~\ref{fig:barplot}, ordered according to amplitude and grouped according to the spatial mode.
The streamfunctions of the spatial modes with the largest amplitudes in the optimal solution (first to fourth) are shown in Figure~\ref{fig:contours}.
Altogether, these modes combined to perturb the traveling double gyre towards a laminar horizontal flow (Figure~\ref{fig:final-fields} (c)).
\begin{figure}[hbt]
\begin{subfigure}{0.49\textwidth}
\includegraphics[width = 1\textwidth]{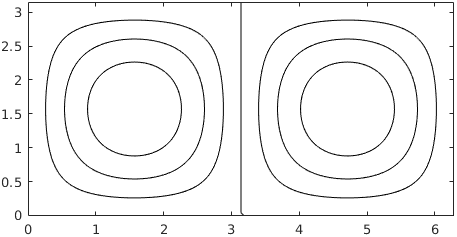}
\subcaption{Streamfunction $k=2$, $l=1$.}
\end{subfigure}
\hfill
\begin{subfigure}{0.49\textwidth}
\includegraphics[width = 1\textwidth]{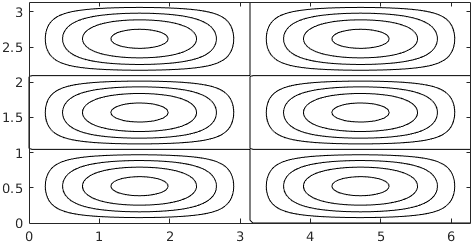}
\subcaption{Streamfunction $k=2$, $l=3$.}
\end{subfigure}\\
\begin{subfigure}{0.49\textwidth}
\includegraphics[width = 1\textwidth]{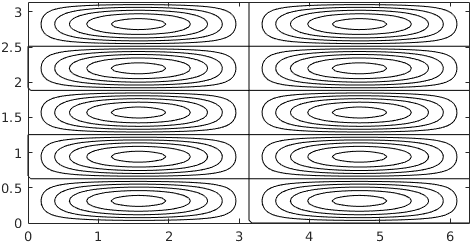}
\subcaption{Streamfunction $k=2$, $l=5$.}
\end{subfigure}
\hfill
\begin{subfigure}{0.49\textwidth}
\includegraphics[width = 1\textwidth]{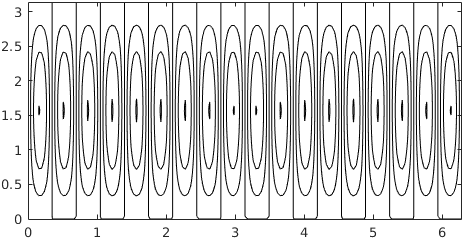}
\subcaption{Streamfunction $k=18$, $l=1$.}
\end{subfigure}
\caption{Streamfunctions of basis functions with highest amplitudes after the optimization.}
\label{fig:contours}
\end{figure}

\begin{figure}[hbt]
\begin{subfigure}{0.32\textwidth}
\includegraphics[width = 1\textwidth]{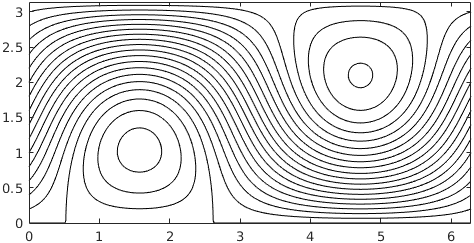}
\subcaption{Streamfunction of the original traveling double gyre.}
\end{subfigure}
\hfill
\begin{subfigure}{0.32\textwidth}
\includegraphics[width = 1\textwidth]{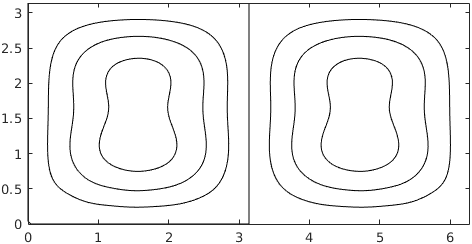}
\subcaption{Streamfunction of optimal perturbation~$u_{(35)}$.}
\end{subfigure}
\hfill
\begin{subfigure}{0.32\textwidth}
\includegraphics[width = 1\textwidth]{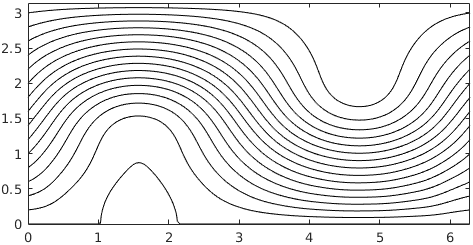}
\subcaption{Streamfunction of final velocity field $v+u$.}
\end{subfigure}
\caption{Streamfunctions of the original velocity field, optimal perturbation, and the perturbed velocity field at $t=0$.}
\label{fig:final-fields}
\end{figure}

We additionally extended the above computation to $100$ steps to investigate the asymptotic behavior under large perturbations. Around step $95$ the increase in coherence diminishes rapidly, as the value of the objective function $\|\bm{\hat{G}}(v+u)\bm \varphi\|^2$ approaches a plateau.
The corresponding velocity field approaches a purely laminar flow (results not shown here), as the optimal perturbation cancels the rotational part of the original velocity field.

\section{Conclusions}
\label{sec:conclusion}

Froyland and Koltai \cite{FK17} introduced a time-augmented construction to enable the efficient numerical construction of the infinitesimal generator of a periodically driven flow;  \cite{FK17} built on earlier work for steady flows \cite{FJK13}.
In the first part of this work, we extended these results to a finite-time flows with general aperiodic driving, by a novel ``reflected'' process.
Proposition \ref{prop:outflux} provided a formula for the accumulated outflow from a general time-dependent family of sets.
Proposition \ref{augmen} derived a spectral mapping theorem for our reflected process, connecting the spectrum of the time-augmented generator with the spectrum of the corresponding reflected evolution operator.
Using the sign structure of the time-augmented eigenfunctions of the augmented generator we built a family of coherent sets and Theorem \ref{thm:cohratio} lower bounded the coherence ratio of this family in terms of the corresponding eigenvalue.

In the second part of this work, we built on the optimization techniques of \cite{Froyland2016} to optimally enhance or destroy coherent features encoded in eigenfunctions.
We directly manipulated the underlying drift field, subject to energy constraints, and proved that the ``small perturbation'' problem has a unique optimum (Proposition \ref{lem:suff}).
Using Lagrange multipliers, we derived an explicit formula for the infinitesimal drift field perturbation;  this would be very difficult to achieve without using a generator framework.
In Section \ref{sec:opt_num} we implemented a multi-step gradient descent method, utilising the efficient time-augmented generator framework, and the explicit Lagrange multiplier solution, to optimize over relatively large energy budgets.

An advantage of our optimization approach is that the basis of perturbing velocity fields is fixed; thus their generators need only be computed once.
On the other hand, the generators are very large, sparse matrices, of which we need to compute the spectrum with the smallest real part. Hybrid spectral discretization techniques as in~\cite{FK17} or multilevel solvers can be a remedy to this.
Finally, the approach could be extended to the non-zero divergence case, and to open systems by considering homogeneous Dirichlet boundary conditions in the Fokker--Planck equation~\eqref{FP0}.


\section*{Acknowledgments}
We thank Andreas Denner for discussions leading to the inception of this work, 
and Nicolas Perkowski for suggestions concerning Theorem~\ref{thm:cohratio}.
GF thanks the DFG Priority Programme 1881 ``Turbulent Superstructures'' for generous travel support, and the Department of Mathematics at the Free University Berlin for hospitality.
GF is partially supported by an ARC Discovery Project. DP180101223.  MS is supported by the DFG Priority Programme 1881 ``Turbulent Superstructures'' and by the DFG CRC 1294. PK is partially supported by the DFG CRC 1114, project A01.

\appendix

\section{Non-autonomous Cauchy Problems}
\label{app:NACP}

In this appendix we summarize important results from the theory of non-auto\-nomous abstract Cauchy problems (NACPs) of the form:
\begin{equation}\label{eq:NACP}
	\text{(NACP)}\, \left\{ \begin{aligned}
	\frac{d}{dt} f(t) &= G(t) f(t)\\
	f(s) &= f_s,
	\end{aligned} \right.
\end{equation}
over some finite time interval $[s,T]$ in $L^p(X)$, for $1< p<\infty$. The linear operators $G(t)$ are so-called \emph{infinitesimal generators}, defined below.
We state the relevant assumptions for this section, which all have been made and used through out the paper:
\begin{itemize}
	\item $X\subset \mathbb{R}^d$ is an open bounded domain with piece-wise $C^4$ boundary.
	\item $v \in C^{(1,1)}([0,\tau]\times \overline{X};\mathbb{R}^d)$.
\end{itemize}
To see that our Fokker--Planck equations \eqref{FP0} and \eqref{FP2} fit in the setting of non-autonomous abstract Cauchy problems we will make use of the following notation:
\begin{equation*}
	t \mapsto f(t) := f(t,\cdot)  \in L^p((0,\tau);L^p(X))\quad \Leftrightarrow \quad (t,x) \mapsto f(t,x) \in L^p((0,\tau)\times X).
\end{equation*}
Equivalently, $C((0,\tau);L^p(X))$ will denote the space of functions mapping $(0,\tau)$ continuously to $L^p(X)$. This means that instead of considering the PDE pointwise in every $(t,x)$ the NACP considers $f$ as a mapping from $[0,\tau]$ into a the Banach space $L^p(X)$  this leads to a differential equation \eqref{eq:NACP} in $L^p(X)$.

Lemma~\ref{lem:NACP_cond} below establishes that our family of operators $(G(t))_{t\in[0,\tau]}$ considered over $L^p(X)$, $1<p<\infty$, defined by
\begin{equation}\label{eq:op-fam}
	G(t) f = -\text{div}_x(v(t,\cdot) f) - \frac{\varepsilon^2}{2} \Delta_x f \quad \text{ in } L^p(X) \text{ for } f \in W^{2,p}(X),
\end{equation}
for $\varepsilon >0$, satisfies the assumptions made in the references \cite{Tanabe1996} and \cite{Lunardi1995} using the domain $\smash{ \mathcal{D}(G(t)) = \mathcal{D}_p := \big\{ f \in W^{2,p}(X) \, \big\vert \, \frac{\partial f}{\partial n} = 0 \text{ on } \partial X \big\} }$ for all $t\in [0,\tau]$.
Theorem~\ref{thm:NACP} is concerned with the existence of unique solutions to \eqref{eq:NACP} and some regularity properties of the solution. Theorem \ref{thm:regularity} states further regularity results and Corollary \ref{cor:conti} is concerned with the regularity of eigenfunctions of $\bm{\hat{G}}$. The latter is needed in the proof of Theorem~\ref{thm:cohratio} in Appendix~\ref{app:cohratio_proof}.
\begin{remark}
\quad
	\begin{enumerate}
		\item The results in \cite{Tanabe1996} are formulated for the time interval $[0,T]$ but of course also hold for arbitrary time intervals $[a,b]$.
		\item We formulate all of the following on $[0,\tau]$. The reflected problem on $[\tau,2\tau]$ can be done analogously.
		\item The case $p=1$ can be treated as well with the theory of Tanabe~\cite{Tanabe1996}, cf.~\cite[Appendix]{FK17}, but is more complex and will be omitted here.
	\end{enumerate}
\end{remark}

\begin{lemma}\label{lem:NACP_cond}
	We consider the family $(G(t))_{t\in [0,\tau]}$ of unbounded linear operators over $L^p(X)$ with the domain $\mathcal{D}(G(t)) = \mathcal{D}_p$, defined by \eqref{eq:op-fam}. The following conditions are fulfilled:
	\begin{enumerate}[(1)]
		\item The spatial domain $X$ is a bounded open set of $\mathbb{R}^d$ of class $C^{4}$ \cite[p.279]{Tanabe1996} (locally) (and globally uniformly regular of class $C^2$ \cite[sec.5.2]{Tanabe1996}.
		\item The coefficients of the differential operator $G(\cdot)$ and the boundary operator $B(t,x)(\xi) := \langle n(x) , \xi \rangle$, with $n(x)$ being the unit outer normal at $x \in \partial X$ and $\xi \in \mathbb{R}^d$, are Hölder-continuous in $t$ \cite[sec.~6.13]{Tanabe1996}.
		\item The coefficients of $G(\cdot)$ are bounded and uniformly continuous in $x$ on $\overline{X}$ \cite[sec.~6.13]{Tanabe1996}.
		\item The differential operator $G(\cdot)$ and the boundary operator $B(t,x)(\xi) := \langle n(x) , \xi \rangle$, satisfy the complementing conditions \cite[p.131]{Tanabe1996} for all $t \in [0,\tau]$.
		\item The family $(G(t))_{t\in [0,\tau]}$ is uniformly strongly elliptic \cite[Def.5.4]{Tanabe1996} and $G(t)$ satisfies the root condition \cite[p.130]{Tanabe1996} for all $t\in [0,\tau]$.
		\item The conditions \cite[(P1), p.221]{Tanabe1996}, \cite[(P2), p.222]{Tanabe1996} and \cite[(P4), p.256]{Tanabe1996} are fulfilled.
	\end{enumerate}
\end{lemma}
Before we continue with the proof, let us remark, that most of the proof has already been done in \cite{Tanabe1996} and some conditions ((3),(4) and (5)) become trivial in our case because the highest order part (called principal part in \cite{Tanabe1996}) of $G(t)$ is the time-independent Laplace operator $\Delta$.
\begin{proof}
One important aspect of this lemma is to establish that our choice $\mathcal{D}_p \subset W^{2,p}(X)$ for the domain $\mathcal{D}(G(t))$ is appropriate.\\
Most of the statements above are proven at some point in \cite{Tanabe1996}.
\begin{enumerate}[(1)]
	\item The regularity required is compatible with our general assumption on the regularity of $X$. The condition for uniformly $C^2$ in \cite{Tanabe1996} is irrelevant for our context as it is only important for unbounded domains.
	\item $B(t,x)$ is constant in $t$ and therefore H\"older continuous. Our assumptions guarantee that $v$ and $\partial_x v$ are H\"older continuous in $t$.
	\item This is satisfied by the assumptions we make on $v$.
	\item This can be verified by straightforward calculations. The important thing to note is that the principal part (called $L^0$ in \cite{Tanabe1996}) of $G(t)$ is time-independently $\frac{\varepsilon^2}{2}\Delta =:L^0$, which implies the roots $r_1 = i, r_2 = -i$ of $L^0(x,\xi+r n(x))$ for $\xi$ perpendicular to $n(x)$ with unit norm, and that $B(t,x) (\xi + r n(x)) = r$ for $\xi$ being perpendicular to $n(x)$.
	\item It is well known that the operator family $(G(t))_{t\in[0,\tau]}$, defined by \eqref{eq:op-fam}, is uniformly elliptic in our setting. Now \cite[Thm.~5.4]{Tanabe1996} states that every strongly elliptic operator satisfies the Root Condition.
	\item \cite[sec.~6.13]{Tanabe1996} shows that the assumptions (P1), (P2) and (P4) are satisfied for advection-diffusion type operators as in our setting.
\end{enumerate}
\end{proof}

\begin{theorem}\label{thm:NACP}
	The NACP \eqref{eq:NACP} with the operators defined by \eqref{eq:op-fam} and the assumptions stated in this paper has a unique solution
	\[
	f \in C([0,\tau];L^p(X)) \ \cap\ C^1((0,\tau];L^p(X))
	\]
	given by $f(t) = \mathcal{P}_{s,t} f(s)$, and that $ f(t) \in \mathcal{D}(G(t))$ for $t\in (0,\tau]$. Further, the family $(\mathcal{P}_{s,t})_{t\geq s}$ is a family of linear, bounded and even compact operators (for $t>s$) on $L^p(X)$.
\end{theorem}
\begin{proof}
	Lemma~\ref{lem:NACP_cond} ensures that we can apply the results from \cite{Tanabe1996} that we will use in the following.
	The results from \cite[sec.~6.13]{Tanabe1996} that use more abstract results from Acquistapace and Terreni (\cite{AT86} and \cite{AT87}) give the existence of a unique solution $f$ to \eqref{eq:NACP} \cite[Thm.~6.6]{Tanabe1996} and a corresponding two parameter family of solution operators (called fundamental solution in \cite{Tanabe1996}) $(\mathcal{P}_{s,t})$.
	\begin{itemize}
		\item The function $f$ is a classical solution \cite[Def.6.1]{Tanabe1996}, i.e.
		\begin{equation*}
			f \in C([0,\tau];L^p(X)) \cap C^1((0,\tau];L^p(X))
		\end{equation*}
		and $f(t) \in \mathcal{D}(G(t)) = \mathcal{D}_p$ for $t>0$ follows directly from \cite[Thm.~6.6]{Tanabe1996}.
		\item For all $t>s$ we the solution can be expressed as $f(t)=\mathcal{P}_{s,t}f(s)$ with the two parameter solution family $(\mathcal{P}_{s,t})_{t\geq s}$. This follows from \cite[Thm.~6.5]{Tanabe1996}.
		\item For every $t>s$ the operator $\mathcal{P}_{s,t}$ is compact on $L^p(X)$. The result \cite[Thm.~6.6]{Tanabe1996}, i.e., $\mathcal{P}_{s,t} f(s) = f(t) \in \mathcal{D}_p$ for $f(s) \in L^p(X)$), gives that for $t>s$ $\mathcal{P}_{s,t} \, : \, L^p(X) \rightarrow \mathcal{D}_p \subset W^{2,p}(X)$ is a bounded linear operator from $L^p(X)$ to $W^{2,p}(X)$. Now the Rellich--Kondrachov embedding theorem \cite[Thm.~6.3]{AF2003} states that $W^{2,p}(X)$ is compactly embedded in $L^p(X)$. Thus $\mathcal{P}_{s,t}$ is a compact operator from $L^p(X)$ to $L^p(X)$, because it maps bounded sets in $L^p$ to bounded sets in $W^{2,p}(X)$ which are relatively compact sets in $L^p(X)$.
	\end{itemize}
\end{proof}

Due to the common domain of all $G(t)$ we get a better regularity for the the solution if the initial condition is more regular. More precisely, we use the result \cite[Cor.~6.1.9 (iv)]{Lunardi1995} to prove the following.

\begin{theorem}\label{thm:regularity}
\quad
	\begin{enumerate}[(i)]
		\item It holds that $t \mapsto f(t) = \mathcal{P}_{0,t} f_0 \in C([0,\tau];\mathcal{D}_p)) \cap C^1([0,\tau];X)$ if and only if $f_0 \in \mathcal{D}_p$ and $G(0) f_0 \in L^p(X)$.
		\item The regularity $f \in C([0,\tau]; W^{m,p}(X))$ implies H\"older continuity\footnote{$C^{\alpha}(\overline{X})$ denotes the space of conitinous functions that are H\"older continuous on $\overline{X}$ with exponent $\alpha \in (0,1)$.} in space for every time slice, i.e., $f(t,\cdot) \in C^{\alpha}(\overline{X})$ for all $p$ such that $0<\alpha \leq m - \frac{d}{p}$, with uniformly bounded H\"older-norm in $t$.
		\item Further, $f \in C([0,\tau];W^{2,p}(X))$ implies $f \in C([0,\tau]\times X)$.
	\end{enumerate}
\end{theorem}

\begin{proof}
\quad
	\begin{enumerate}[(i)]
		\item This result is basically \cite[Cor.~6.1.9 (iv)]{Lunardi1995} with $\mathcal{D}_p = \mathcal{D}(G(t))$ being the common domain of the family of unbounded operators $(G(t))_{t\in[0,\tau]}$ in the Banach space $L^p(X)$ ($\mathcal{D}_p$ is dense in $L^p(X)$), and the resulting two parameter evolution family is $(\mathcal{P}_{s,t})_{t\geq s}$.
		\item For every $t \in [0,\tau]$ we have $f(t,\cdot ) \in W^{m,p}(X)$. From the appropriate Sobolev embedding theorem \cite[Thm.~4.12]{AF2003} (also cf. Morrey) follows that for domains as smooth as our $X\subset \mathbb{R}^n$ the space $W^{m,p}(X)$ is continuously embedded in $C^{\alpha}(\overline{X})$ for $0<\alpha \leq m-\frac{d}{p}$, i.e., there exists a $K>0$ such that
		\begin{equation}\label{eq:cont_emb}
			\| g \|_{C^{\alpha}} \leq K \| g \|_{W^{2,p}}
		\end{equation}
		for all $g \in W^{2,p}(X)$. This with the regularity $f \in C([0,\tau]; W^{m,p}(X))$ further implies that the constant for the continuous embedding for $(f(t))_{t\in [0,\tau]}$ into $C^{\alpha}$ can be taken as independent of~$t$.
		\item We consider a function $f \in C([0,\tau];W^{2,p}(X))$. Now (ii) gives $f(t,\cdot) \in C^{\alpha}(X))$ for every $t$ with $0<\alpha \leq 2 - \frac{d}{p}$. Together with (i) this implies
		\begin{equation*}
			\sup_{t\in[0,\tau]} \| f(t,\cdot) \|_{C^{\alpha}} \leq K \sup_{t\in[0,\tau]} \| f(t,\cdot) \|_{W^{2,p}} \leq M < \infty
		\end{equation*}
		This immediately implies $f \in C([0,\tau];C(\overline{X}))$ which is equivalent to $f \in C([0,\tau ]\times \overline{X})$. Let us briefly show the direction of the last statement that we need: For $f \in C([0,\tau];C(\overline{X}))$ we have
		\begin{equation*}
			|f(t_n,x_n) - f(t,x)| \leq \| f(t_n, \cdot) - f(t,\cdot) \|_{\infty} + |f(t,x_n) - f(t,x)| \xrightarrow{n\rightarrow \infty} 0 \; .
		\end{equation*}
	\end{enumerate}
\end{proof}

\begin{corollary}\label{cor:conti}
	Let $\bm{f}$ be an eigenfunction of $\bm{\hat{G}}$ considered on $L^p$, $p> \frac{d}{2}$, then $\bm{f} \in C([0,2\tau] \times X)$.
\end{corollary}
\begin{proof}
	For an eigenfunction $\bm{f}$ follows
\begin{equation*}
	\hat{\mathcal{P}}_{0,t} \bm{f}(0,\cdot) = e^{\mu t} \bm{f}(t \! \!\!\! \mod 2\tau, \cdot)
\end{equation*}
from theorem \ref{augmen}. Further theorem \ref{thm:NACP} implies $e^{-\mu \tau}\hat{\mathcal{P}}_{0,2\tau} \bm{f} (0,\cdot) = \bm{f}(0,\cdot) \in \mathcal{D}(\hat{G}(2\tau))$. Now theorem \ref{thm:regularity} (i) and (iii) with $\mathcal{D}(\hat{G}(2\tau)) = \mathcal{D}_p$ gives the claim.
\end{proof}

\section{Proof of Theorem~\ref{thm:cohratio} and Proposition~\ref{thm:cohratio2}}
\label{app:cohratio_proof}
\begin{proof}[Proof of Theorem~\ref{thm:cohratio}]
The proof strongly follows the  lines of those of~\cite[Theorems 16 and 19]{FK17}, which, in turn, borrows ideas from~\cite{FrSt10,FrSt13}.
It consists of two main steps. First, we consider the events
\[
\mathcal{E}_n := \left\{\omega\,\vert\, x_{r_i}(\omega)\in A_{r_i}^+,\ \forall\,i=1,\ldots,n\right\}\,
\]
for a dense sequence of times, and show that $\mathcal{E}_n\downarrow\mathcal{E}:=\bigcap_{r\in[0,\tau]} \{\omega\,\vert\, x_r(\omega)\in A_r^+\}$. Second, we use this approximation to bound the retention probability in the family, hence the coherence ratio.

\emph{Step 1.}
Let $(r_i)_{i\in\mathbb{N}}$ be a dense sequence in $[0,\tau]$ such that $r_1 = \tau$; this latter condition is needed such that the decomposition in~\eqref{eq:eventsplit} below is always possible.
The events $\mathcal{E}_n $ are clearly measurable. Since the paths $t\mapsto x_t(\omega)$ of the process \eqref{dynsystem} are continuous, so is $t\mapsto \bm{f}(t,x_t(\omega)) =: F_t(\omega)$. By corollary \ref{cor:conti} we have that $\bm{f} \in C([0,2\tau] \times \overline{X})$.
To see $\mathcal{E}_n\downarrow\mathcal{E}$, note that by the continuity of $t\mapsto F_t$ it holds
\[
0 \le F_t(\omega) \ \forall t\in[0,\tau]
\quad \Leftrightarrow \quad
0 \le \inf_{t\in[0,\tau]} F_t(\omega)
\quad \Leftrightarrow \quad
0 \le \inf_{i \in \mathbb{N}} F_{r_i}(\omega).
\]

For a finite measure $\pi$ let $\mathbb{P}_{\pi} := \mathbb{P}(\cdot\,\vert\, x_0\sim \pi)$ denote the law of the process \eqref{dynsystem} with initial distributions $\pi$. If $\pi$ is not a probability measure, then we denote $\mathbb{P}_{\pi} = \pi(\Omega) \mathbb{P}_{\pi / \pi(\Omega)}$, where $\Omega$ is the entire event space. If, additionally, $\pi$ is a signed measure, then  $\mathbb{P}_{\nu} := \mathbb{P}_{\pi^+} - \mathbb{P}_{\pi^-}$,
where $\pi^+$ and $\pi^-$ denote the positive and negative parts of the signed measure $\pi$, respectively, in the sense of the Hahn decomposition, i.e. $\pi = \pi^+ - \pi^-$.
Now, the $\sigma$-additivity of $\mathbb{P}$ yields also $\mathbb{P}_{\pi}(\mathcal{E}_n) \to \mathbb{P}_{\pi}(\mathcal{E})$ as $n\to\infty$.

\emph{Step 2.}
By Proposition~\ref{augmen} we have that $\bm{f}(0,\cdot)$ is an eigenfunction of the integral preserving operator~$\mathcal{P}_{0,\tau}^*\mathcal{P}_{0,\tau}$ at the eigenvalue~$e^{2\mu\tau}<1$. Thus we have $\int_X \bm{f}(0,\cdot)\,dm = 0$, and with it $\int_X \bm{f}(t,\cdot)\,dm = e^{-\mu t}\int_X \mathcal{P}_{0,t}\bm{f}(0,\cdot)\,dm = 0$ again by the integral-preserving property.
We define the signed measure $\nu$ via $d\nu(x) = \bm{f}(0,x)\,dm(x)$. With the given scaling of~$\bm{f}$ we have that~$\int_{A_{\tau}^{\pm}}\bm{f}(\tau,\cdot)\,dm = \pm 1$. By Proposition~\ref{augmen} we have that~$\mathcal{P}_{0,t}\bm{f}(0,\cdot) = e^{\mu t}\bm{f}(t,\cdot)$, thus for every $n\in\mathbb{N}$ we have
\begin{eqnarray}
e^{\mu \tau} & = & e^{\mu \tau}\int_{A_{\tau}^+} \bm{f}(\tau,\cdot) dm\nonumber\\
& = & \int_{A_{\tau}^+} \mathcal{P}_{0,\tau} \bm{f}(0,\cdot) dm \nonumber\\
& = & \mathbb{P}_{\nu}(x_{\tau}\in A_{\tau}^+) \nonumber\\
& = & \mathbb{P}_{\nu}(x_{r_i}\in A_{r_i}^+,\ i=1,\ldots,n) + \sum_{j=2}^n \underbrace{\mathbb{P}_{\nu}(x_{r_j}\notin A_{r_i}^+,\ x_{r_i}\in A_{r_i}^+,\ \forall\ r_i>r_j)}_{=: p_j}. \label{eq:eventsplit}
\end{eqnarray}
The last equality follows from the decomposition of the event~$\{x_{r_1} \in A_{\tau}^+\} = \{x_{\tau} \in A_{\tau}^+\}$ into disjoint events $\{x_{r_i}\in A_{r_i}^+ \text{ for all }r_i>r_j, \text{ but }x_{r_j}\notin A_{r_j}^+\}$, $j=2,\ldots,n$, and $\{x_{r_i}\in A_{r_i}^+ \text{ for all }i=1,\ldots,n\}$. One can see~\cite{FJK13} that $p_j\le 0$, because the set of initial conditions $x_{r_j}\notin A_{r_j}^+$ is contained in the non-positive support of~$\nu$. It follows that
\begin{equation}
e^{\mu \tau}\le \mathbb{P}_{\nu}(x_{r_i}\in A_{r_i}^+,\ i=1,\ldots,n) = \mathbb{P}_{\nu}(\mathcal{E}_n).
\label{eq:eventestim}
\end{equation}
Thus, by step 1,
\[
e^{\mu \tau} \le \lim_{n\to\infty} \mathbb{P}_{\nu}(\mathcal{E}_n) = \mathbb{P}_{\nu}(\mathcal{E}) = \mathbb{P}_{\nu}\left(\cap_{t\in[0,\tau]}\{x_t\in A_t^+\}\right)\,.
\]
The same bound can be obtained for the family $\{A_t^-\}_{t\in [0,\tau]}$.
%
Noting that~$\mathbb{P}_{\nu^{\pm}} \le \|\bm{f}(0,\cdot)\|_{L^{\infty}} \mathbb{P}_{|X| m}$, the claim follows, since $|A_0^\pm| = |X|\, m(A_0^\pm)$.
\end{proof}

\begin{proof}[Proof of Proposition~\ref{thm:cohratio2}]
The proof is entirely analogous to that of Theorem~\ref{thm:cohratio}, except that the system of equations containing \eqref{eq:eventsplit} is altered. The deviating part is the system of inequalities,  which follows by the assumptions of the proposition:
\begin{equation}
\label{eq:LinCombEfun}
\begin{aligned}
e^{\mu_k \tau}\int_{A^+_{\tau}} \bm{f}(\tau,\cdot)\,dm &=  \sum_{i=1}^k  e^{\mu_k \tau} \alpha_i \int_{A^+_{\tau}} \bm{f}_i(\tau,\cdot)\,dm \\
&\le \sum_{i=1}^k e^{\mu_i \tau} \alpha_i \int_{A^+_{\tau}} \bm{f}_i(\tau,\cdot)\,dm\\
&= \int_{A^+_{\tau}} \sum_{i=1}^k \alpha_i \big( e^{\mu_i \tau} \bm{f}_i(\tau,\cdot)\big) \,dm\\
&{\stackrel{\eqref{eq:eigenfunction}}{=}} \int_{A^+_{\tau}} \sum_{i=1}^k \alpha_i \mathcal{P}_{0,\tau} \bm{f}_i(0,\cdot) \,dm\\
&=  \int_{A^+_{\tau}} \mathcal{P}_{0,\tau} \bm{f}(0,\cdot) \,dm\,.
\end{aligned}
\end{equation}
\phantom{X}
\end{proof}

\section{Bilinear form}

The following theorem is also mentioned in \cite[Appendix]{McIntosh1968} but only partially proven there. We prove the part that is important for this work.
\begin{theorem}\label{thm:bil_form}
	For a bounded bilinear form $B$ acting on a Hilbert-Space $(H,\langle \cdot, \cdot \rangle_H , \| \cdot \|_H)$,
\begin{equation*}
	B \, :\, H \times H \rightarrow \mathbb{R} \qquad \exists \gamma >0 \; : \; B(u,v) \leq \gamma \| u \|_H \|v\|_H  \qquad \forall u,v \in H
\end{equation*}
	there exists a linear, bounded operator $T \, : \, H \rightarrow H$ such that
\begin{equation*}
	B(u,v) = \langle u,Tv\rangle_H \qquad \forall u,v \in H \; .
\end{equation*}
If $B$ is symmetric, then $T$ is self-adjoint. Further if $B$ is additionally positive definite
\begin{equation*}
\exists \beta >0  \; : \;\beta \| u \|_H ^2 \leq B(u,u) \quad \forall u \in H
\end{equation*}
then $T$ is also injective and has a continuous inverse.
\end{theorem}
\begin{proof}
	First let us fix $v \in H$ and consider the linear functional $\ell_v \, : \, u \mapsto B(u,v) = \ell_v (u)$. By the Riesz representation theorem there exists and element $z_v \in H$ such that $\ell_v(u) = \langle u, z_v\rangle_H$ for all $u \in H$. It remains to show that the mapping $T\, : \, v \mapsto z_v = T(v)$ is linear and bounded.\\
Consider $v + \lambda w  \in H$ for $\lambda \in \mathbb{R}$ and $v,w \in H$, then
\begin{equation*}
	\langle u, T(v+\lambda w) \rangle_H = B(u,v+\lambda w ) = B(u,v) + \lambda B(v,w) = \langle u , T(v) \rangle_H + \lambda \langle u, T(w) \rangle_H
\end{equation*}
holds for all $u \in H$. Thus $T$ is linear. The boundedness follows trivially from
\begin{equation*}
	\| Tv \|_H = \sup_{\|u\|_H = 1} |\langle u,Tv \rangle_H| =  \sup_{\|u\|_H = 1} |B(u,v)|  \leq \gamma \|v\|_H \; .
\end{equation*}
Let us assume that $B$ is symmetric, then
\begin{equation*}
	\langle u,Tv \rangle_H  = B(u,v) = B(v,u) = \langle v, Tu\rangle_H
\end{equation*}
shows that $T$ is self-adjoint.\\
Now assume that $B$ is additionally positive definite. The estimate
\begin{equation*} \label{eq:T_inj}
	\|v-w\|_H\, \|Tv - Tw\|_H \ge \langle v-w, T(v-w) \rangle_H \geq \beta \| v-w\|_H^2 >0 \qquad \text{for } v\neq w \in H
\end{equation*}
immediately implies injectivity. Thus $T^{-1}$ is defined on the range of $T$ ($\text{ran}(T)$) and it is continuous.
\begin{align*}
	\| T^{-1} u - T^{-1} v \|_H = \| T^{-1}Tx - T^{-1}Ty \|_H = \|x -y \|_H \\
 \qquad \stackrel{\eqref{eq:T_inj}}{\leq} \frac{1}{\beta} \| Tx - Ty \|_H = \frac{1}{\beta} \| u-v\|_H \qquad \forall u,v \in \text{ran}(T) \; .
\end{align*}
\end{proof}

\bibliographystyle{myalpha}
\bibliography{fkp2bib}

\newcommand{\etalchar}[1]{$^{#1}$}
\begin{thebibliography}{RBBV{\etalchar{+}}07}

\bibitem[AF03]{AF2003}
R.~A. Adams and J.~J.~F. Fournier.
\newblock {\em Sobolev Spaces}.
\newblock Elsevier, 2003.

\bibitem[Are02]{aref_02}
H.~Aref.
\newblock The development of chaotic advection.
\newblock {\em Physics of Fluids}, 14(4):1315--1325, 2002.

\bibitem[AT86]{AT86}
P.~Acquistapace and B.~Terreni.
\newblock On fundamental solutions for abstract parabolic equations.
\newblock {\em Differential Equations in Banach Spaces (Berlin: Springer)},
  pages 1--11, 1986.

\bibitem[AT87]{AT87}
P.~Acquistapace and B.~Terreni.
\newblock A unified approach to abstract linear nonautonomous parabolic
  equations.
\newblock {\em Rend. Semin. Mat. Univ. Padova}, 78:47--107, 1987.

\bibitem[Bal15]{Bal15}
S.~Balasuriya.
\newblock Dynamical systems techniques for enhancing microfluidic mixing.
\newblock {\em Journal of Micromechanics and Microengineering}, 25(9):094005,
  2015.

\bibitem[BAS00]{BoArSt00}
P.~L. Boyland, H.~Aref, and M.~A. Stremler.
\newblock Topological fluid mechanics of stirring.
\newblock {\em Journal of Fluid Mechanics}, 403:277--304, 2000.

\bibitem[CAG08]{CoAdGi08}
L.~Cortelezzi, A.~Adrover, and M.~Giona.
\newblock Feasibility, efficiency and transportability of short-horizon optimal
  mixing protocols.
\newblock {\em Journal of Fluid Mechanics}, 597:199--231, 2008.

\bibitem[CKRZ08]{CKRZ08}
P.~Constantin, A.~Kiselev, L.~Ryzhik, and A.~Zlato{\v{s}}.
\newblock Diffusion and mixing in fluid flow.
\newblock {\em Annals of Mathematics}, pages 643--674, 2008.

\bibitem[CL99]{Chicone1999}
C.~Chicone and Y.~Latushkin.
\newblock {\em Evolution Semigroups in Dynamical Systems and Differential
  Equations}.
\newblock AMS, 1999.

\bibitem[DJM16]{DJM}
A.~Denner, O.~Junge, and D.~Matthes.
\newblock Computing coherent sets using the {F}okker--{P}lanck equation.
\newblock {\em Journal of Computational Dynamics}, 3(2):163--177, 2016.

\bibitem[EN00]{EnNa00}
K.-J. Engel and R.~Nagel.
\newblock {\em One-parameter semigroups for linear evolution equations}, volume
  194.
\newblock Springer Science \& Business Media, 2000.

\bibitem[FGTW16]{FGTW16}
G.~Froyland, C.~Gonz{\'a}lez-Tokman, and T.~M. Watson.
\newblock Optimal mixing enhancement by local perturbation.
\newblock {\em SIAM Review}, 58(3):494--513, 2016.

\bibitem[FJK13]{FJK13}
G.~Froyland, O.~Junge, and P.~Koltai.
\newblock Estimating long-term behavior of flows without trajectory
  integration: the infinitesimal generator approach.
\newblock {\em SIAM Journal on Numerical Analysis}, 51(1):223--247, 2013.

\bibitem[FK17]{FK17}
G.~Froyland and P.~Koltai.
\newblock Estimating long-term behavior of periodically driven flows without
  trajectory integration.
\newblock {\em Nonlinearity}, 30(5):1948, 2017.

\bibitem[FLS10]{Froyland2010}
G.~Froyland, S.~Lloyd, and N.~Santitissadeekorn.
\newblock Coherent sets for nonautonomous dynamical systems.
\newblock {\em Physica D}, 239:1527--1541, 2010.

\bibitem[FP09]{froyland_padberg_09}
G.~Froyland and K.~Padberg.
\newblock Almost-invariant sets and invariant manifolds -- connecting
  probabilistic and geometric descriptions of coherent structures in flows.
\newblock {\em Physica D}, 238:1507--1523, 2009.

\bibitem[FPG14]{FrPa14}
G.~Froyland and K.~Padberg-Gehle.
\newblock Almost-invariant and finite-time coherent sets: Directionality,
  duration, and diffusion.
\newblock In W.~Bahsoun, C.~Bose, and G.~Froyland, editors, {\em Ergodic
  Theory, Open Dynamics, and Coherent Structures}, pages 171--216, New York,
  NY, 2014. Springer New York.

\bibitem[Fro13]{Froyland2013}
G.~Froyland.
\newblock An analytic framework for identifying finite-time coherent sets in
  time-dependent dynamical systems.
\newblock {\em Physica D}, 250:1--19, 2013.

\bibitem[FRS18]{seba}
G.~Froyland, C.~P. Rock, and K.~Sakellariou.
\newblock Sparse eigenbasis approximation: multiple feature extraction across
  spatiotemporal scales with application to coherent set identification.
\newblock {\em arXiv preprint arXiv:1812.02787}, 2018.

\bibitem[FS10]{FrSt10}
G.~Froyland and O.~Stancevic.
\newblock Escape rates and {P}erron--{F}robenius operators: {O}pen and closed
  dynamical systems.
\newblock {\em Discrete and Continuous Dynamical Systems - Series B},
  14:457--472, 2010.

\bibitem[FS13]{FrSt13}
G.~Froyland and O.~Stancevic.
\newblock Metastability, {L}yapunov exponents, escape rates, and topological
  entropy in random dynamical systems.
\newblock {\em Stochastics and Dynamics}, 13(04), 2013.

\bibitem[FS17]{Froyland2016}
G.~Froyland and N.~Santitissadeekorn.
\newblock Optimal mixing enhancement.
\newblock {\em SIAM Journal on Applied Mathematics}, 77(4):1444--1470, 2017.

\bibitem[FSM10]{Froyland2010b}
G.~Froyland, N.~Santitissadeekorn, and A.~Monahan.
\newblock Transport in time-dependent dynamical systems: Finite-time coherent
  sets.
\newblock {\em Chaos: An Interdisciplinary Journal of Nonlinear Science},
  20(4):043116, 2010.

\bibitem[GD17]{GiDa17}
D.~Giannakis and S.~Das.
\newblock Extraction and prediction of coherent patterns in incompressible
  flows through space-time {K}oopman analysis.
\newblock {\em arXiv preprint arXiv:1706.06450}, 2017.

\bibitem[HKK18]{HaKaKo18}
G.~Haller, D.~Karrasch, and F.~Kogelbauer.
\newblock Material barriers to diffusive and stochastic transport.
\newblock {\em Proceedings of the National Academy of Sciences},
  115(37):9074--9079, 2018.

\bibitem[How74]{How74}
J.~S. Howland.
\newblock Stationary scattering theory for time-dependent {H}amiltonians.
\newblock {\em Mathematische Annalen}, 207(4):315--335, 1974.

\bibitem[HP98]{haller1998finite}
G.~Haller and A.~C. Poje.
\newblock Finite time transport in aperiodic flows.
\newblock {\em Physica D: Nonlinear Phenomena}, 119(3):352--380, 1998.

\bibitem[JW02]{jones2002invariant}
C.~Jones and S.~Winkler.
\newblock Invariant manifolds and {L}agrangian dynamics in the ocean and
  atmosphere.
\newblock {\em Handbook of dynamical systems}, 2:55--92, 2002.

\bibitem[KK16]{KaKe16}
D.~Karrasch and J.~Keller.
\newblock A geometric heat-flow theory of {L}agrangian coherent structures.
\newblock {\em arXiv preprint arXiv:1608.05598}, 2016.

\bibitem[KKS16]{KlKoSch16}
S.~Klus, P.~Koltai, and C.~Sch{\"u}tte.
\newblock On the numerical approximation of the {P}erron--{F}robenius and
  {K}oopman operator.
\newblock {\em J.\ Comput.\ Dyn.}, 3(1):51--79, 2016.

\bibitem[KLP19]{KLP2018}
P.~Koltai, H.~C. Lie, and M.~Plonka.
\newblock Fr{\'{e}}chet differentiable drift dependence of
  {P}erron{\textendash}{F}robenius and {K}oopman operators for
  non-deterministic dynamics.
\newblock {\em Nonlinearity}, 32(11):4232--4257, sep 2019.

\bibitem[KR18]{KoRe18}
P.~Koltai and D.~M. Renger.
\newblock From large deviations to semidistances of transport and mixing:
  coherence analysis for finite {L}agrangian data.
\newblock {\em Journal of nonlinear science}, 28(5):1915--1957, 2018.

\bibitem[LH04]{LiHa04}
W.~Liu and G.~Haller.
\newblock Strange eigenmodes and decay of variance in the mixing of diffusive
  tracers.
\newblock {\em Physica D: Nonlinear Phenomena}, 188(1-2):1--39, 2004.

\bibitem[LM13]{lasotamackey}
A.~Lasota and M.~C. Mackey.
\newblock {\em Chaos, fractals, and noise: stochastic aspects of dynamics},
  volume~97.
\newblock Springer Science \& Business Media, 2013.

\bibitem[LNP13]{LePa13}
T.~Leli{\`e}vre, F.~Nier, and G.~A. Pavliotis.
\newblock Optimal non-reversible linear drift for the convergence to
  equilibrium of a diffusion.
\newblock {\em Journal of Statistical Physics}, 152(2):237--274, Jul 2013.

\bibitem[LTD11]{LiThDo11}
Z.~Lin, J.-L. Thiffeault, and C.~R. Doering.
\newblock Optimal stirring strategies for passive scalar mixing.
\newblock {\em Journal of Fluid Mechanics}, 675:465--476, 2011.

\bibitem[Lue97]{Luenberger1985}
D.~G. Luenberger.
\newblock {\em Optimization by vector space methods}.
\newblock John Wiley \& Sons, 1997.

\bibitem[Lun95]{Lunardi1995}
A.~Lunardi.
\newblock {\em Analytic Semigroups and Optimal Regularity in Parabolic
  Problems}.
\newblock Basel: Birkh\"auser, 1995.

\bibitem[McI68]{McIntosh1968}
A.~McIntosh.
\newblock Representation of bilinear forms in {H}ilbert space by linear
  operators.
\newblock {\em Transactions of the American Mathematical Society},
  131(2):365--377, 1968.

\bibitem[MMG{\etalchar{+}}07]{MMGVP07}
G.~Mathew, I.~Mezi{\'c}, S.~Grivopoulos, U.~Vaidya, and L.~Petzold.
\newblock Optimal control of mixing in {S}tokes fluid flows.
\newblock {\em Journal of Fluid Mechanics}, 580:261--281, 2007.

\bibitem[MMP84]{mackay1984transport}
R.~MacKay, J.~Meiss, and I.~Percival.
\newblock Transport in {H}amiltonian systems.
\newblock {\em Physica D}, 13(1):55--81, 1984.

\bibitem[MMP05]{MaMePe05}
G.~Mathew, I.~Mezi{\'c}, and L.~Petzold.
\newblock A multiscale measure for mixing.
\newblock {\em Physica D: Nonlinear Phenomena}, 211(1-2):23--46, 2005.

\bibitem[OBPG15]{OBP15}
S.~Ober-Bl{\"o}baum and K.~Padberg-Gehle.
\newblock Multiobjective optimal control of fluid mixing.
\newblock {\em PAMM}, 15(1):639--640, 2015.

\bibitem[Paz83]{Paz83}
A.~Pazy.
\newblock {\em Semigroups of linear operators and applications to partial
  differential equations}.
\newblock Springer-Verlag, New York, 1983.

\bibitem[Pie91]{Pie91}
R.~Pierrehumbert.
\newblock Chaotic mixing of tracer and vorticity by modulated travelling rossby
  waves.
\newblock {\em Geophysical \& Astrophysical Fluid Dynamics}, 58(1-4):285--319,
  1991.

\bibitem[Pro99]{Pro99}
A.~Provenzale.
\newblock Transport by coherent barotropic vortices.
\newblock {\em Annual review of fluid mechanics}, 31(1):55--93, 1999.

\bibitem[PS08]{Pavliotis2008}
G.~Pavliotis and A.~Stuart.
\newblock {\em Multiscale methods: averaging and homogenization}.
\newblock Springer Science \& Business Media, 2008.

\bibitem[RBBV{\etalchar{+}}07]{Rypina2007}
I.~Rypina, M.~Brown, F.~Beron-Vera, H.~Kocak, M.~Olascoaga, and
  I.~Udovydchenkov.
\newblock On the {L}agrangian dynamics of atmospheric zonal jets and the
  permability of the stratospheric polar vortex.
\newblock {\em Journal of the Atmospheric Sciences}, 2007.

\bibitem[RKLW90]{romkedar_etal_1990}
V.~Rom-Kedar, A.~Leonard, and S.~Wiggins.
\newblock An analytical study of transport, mixing and chaos in an unsteady
  vortical flow.
\newblock {\em Journal of Fluid Mechanics}, 214:347--394, 1990.

\bibitem[Roc70]{Rockafellar1970}
R.~Rockafellar.
\newblock {\em Convex Analysis}.
\newblock Princeton University Press, 1970.

\bibitem[SLM05]{shadden2005definition}
S.~C. Shadden, F.~Lekien, and J.~E. Marsden.
\newblock Definition and properties of {L}agrangian coherent structures from
  finite-time {L}yapunov exponents in two-dimensional aperiodic flows.
\newblock {\em Physica D: Nonlinear Phenomena}, 212(3):271--304, 2005.

\bibitem[SW06]{SaWi06}
R.~M. Samelson and S.~Wiggins.
\newblock {\em Lagrangian transport in geophysical jets and waves: The
  dynamical systems approach}, volume~31.
\newblock Springer Science \& Business Media, 2006.

\bibitem[Tan96]{Tanabe1996}
H.~Tanabe.
\newblock {\em Functional Analytic Methods for Partial Differential Equations},
  volume vol 204.
\newblock Boca Raton, FL: CRC Press, 1996.

\bibitem[TDG04]{ThDoGi04}
J.-L. Thiffeault, C.~R. Doering, and J.~D. Gibbon.
\newblock A bound on mixing efficiency for the advection--diffusion equation.
\newblock {\em Journal of Fluid Mechanics}, 521:105--114, 2004.

\bibitem[Thi08]{Thi08}
J.-L. Thiffeault.
\newblock Scalar decay in chaotic mixing.
\newblock In {\em Transport and Mixing in Geophysical Flows}, pages 3--36.
  Springer, 2008.

\bibitem[Thi12]{Thi12}
J.-L. Thiffeault.
\newblock Using multiscale norms to quantify mixing and transport.
\newblock {\em Nonlinearity}, 25(2):R1, 2012.

\bibitem[TP08]{ThPa08}
J.-L. Thiffeault and G.~A. Pavliotis.
\newblock Optimizing the source distribution in fluid mixing.
\newblock {\em Physica D: Nonlinear Phenomena}, 237(7):918--929, 2008.

\bibitem[Wig92]{wiggins_92}
S.~Wiggins.
\newblock {\em Chaotic Transport in Dynamical Systems}.
\newblock Springer-Verlag, New York, NY, 1992.

\bibitem[Wig05]{wiggins2005dynamical}
S.~Wiggins.
\newblock The dynamical systems approach to {L}agrangian transport in oceanic
  flows.
\newblock {\em Annu. Rev. Fluid Mech.}, 37:295--328, 2005.

\end{thebibliography}

\end{document}